\documentclass[11pt]{article}
\usepackage[utf8]{inputenc}
\usepackage{mathrsfs}
\usepackage{amsmath,amssymb,amsthm,amsfonts,bm}
\usepackage{esint}
\usepackage[english]{babel}
\usepackage[a4paper]{geometry}
\usepackage{enumerate}
\usepackage{enumitem}
\usepackage{newtxtext}
\usepackage{hyperref}                   
\hypersetup{                            
		colorlinks=true,
		linkcolor=blue,
		citecolor=black,                     
		urlcolor=black                       
}

\title{$C^\infty$ partial regularity of the singular set in the obstacle problem}

\date{}
\newtheorem{theorem}{Theorem}[section]
\newtheorem*{theorem*}{Theorem}
\newtheorem{lemma}[theorem]{Lemma}
\newtheorem{corollary}[theorem]{Corollary}
\newtheorem*{corollary*}{Corollary}

\newtheorem{definition}[theorem]{Definition}
\newtheorem{proposition}[theorem]{Proposition}
\newtheorem{remark}[theorem]{Remark}

\numberwithin{equation}{section}

\renewcommand*\thetheorem{\arabic{section}.\arabic{theorem}} 
\renewcommand*{\theequation}{\arabic{section}.\arabic{equation}}

\let\emptyset\varnothing 

\newcommand*{\weakstar}{\stackrel{*}{\rightharpoonup}}
\newcommand*{\weak}{\rightharpoonup} 
\newcommand*{\de}{\partial} 
\newcommand*{\der}{\nabla} 
\newcommand*{\lap}{\mathit{\Delta}}
\newcommand*{\N}{\mathbb{N}}
\newcommand*{\R}{\mathbb{R}}
\newcommand*{\PP}{\bm{P}}
\newcommand*{\Sign}{\mathcal{S}}
\newcommand*{\Signeven}{\mathcal{S}^{\text{\normalfont even}}}
\newcommand*{\Signodd}{\mathcal{S}^{\text{ \normalfont odd}}}

\newcommand*{\curlyA}{\mathcal{A}}
\newcommand*{\curlyP}{\mathcal{P}}
\newcommand*{\eps}{\varepsilon}
\DeclareMathOperator{\lipschitz}{Lip}
\DeclareMathOperator{\supp}{supp}
\DeclareMathOperator{\loc}{loc}

\DeclareMathOperator{\dist}{dist}

\DeclareMathOperator{\Span}{span}
\DeclareMathOperator{\reg}{Reg}

\DeclareMathOperator{\interior}{int}

\begin{document}
\author{Federico Franceschini and Wiktoria Zato\'n }
\newcommand{\Addresses}{{
  \bigskip

  F.~Franceschini, \textsc{ETH Z\"urich, Department of Mathematics, R\"amistrasse 101, 8092 Z\"urich, Switzerland}\par\nopagebreak
  \textit{E-mail address:} \texttt{federico.franceschini@math.ethz.ch}

  \medskip

  W.~Zato\'n, \textsc{Universit\"at Z\"urich, Institut f\"ur Mathematik, Winterthurerstrasse 190, 8057 Z\"urich, Switzerland}\par\nopagebreak
  \textit{E-mail address:} \texttt{wiktoria.zaton@math.uzh.ch}

}}

\maketitle

\begin{abstract}
We show that the singular set $\Sigma$ in the classical obstacle problem can be locally covered by a $C^\infty$ hypersurface, up to an ``exceptional'' set $E$, which has Hausdorff dimension at most $n-2$ (countable, in the $n=2$ case). Outside this exceptional set, the solution admits a polynomial expansion of arbitrarily large order. We also prove that $\Sigma\setminus E$ is extremely unstable with respect to monotone perturbations of the boundary datum. We apply this result to the planar Hele-Shaw flow, showing that the free boundary can have singular points for at most countable many times.
\end{abstract}
\section{Introduction}
\subsection{The classical obstacle problem}
The classical obstacle problem consists in studying the solutions of the variational problem
\begin{equation*}
    \min\left\{\frac 1 2\int_{B_1} |\nabla w|^2: w\geq \varphi \text{ in }B_1\subseteq \R^n, w=g \text{ on }\de B_1\right\},
\end{equation*}
where $g\colon \de B_1 \to \R$  and $\varphi\colon B_1\to \R$ are given, with $\varphi<g$ on $\de B_1$. In two dimensions an intuitive interpretation of this problem is the following. The graph of the minimizer $w$ represents the shape of a thin membrane stretched over $\overline B_1$ and fixed on $\de B_1$ along the profile $g$. The hypograph of $\varphi$ represents a solid ``obstacle'' above which the membrane must lie, possibly touching it. The Dirichlet energy, finally, corresponds to the linearization of the surface energy of the membrane, which is assumed proportional to the area of the graph of $w$.

It is well-known that there exists a unique optimal shape $v$ and that it enjoys $C^{1,1}_{\loc}$ regularity, provided that $\varphi$ is smooth enough. Furthermore, $u:=v-\varphi$ solves the Euler-Lagrange equation
\begin{equation*}
    \lap u= -\lap\varphi\,  \chi_{\{u>0\}} \quad \text{in }\quad B_1.
\end{equation*}
One of the most challenging problems is to understand the a priori unknown interface $\de\{ u>0\}$, called the ``free boundary''. Unless $\lap \varphi<0$, simple examples show that the free boundary can be any closed set, hence it is standard to assume that $\varphi$ is superharmonic. Thus, we deal with solutions of
 \begin{equation}\label{eq:intro_obst}
\begin{cases}
\lap u =f \chi_{\{u>0\}}&\mbox{in } B_1,\\
u\ge 0 \quad &\mbox{in } B_1,
\end{cases}
\end{equation}
where $f\in C^\infty(B_1)$ is given and positive.

By classical works of Caffarelli \cite{C77,C98} the free boundary $\de \{u>0\}$
splits into a regular and a singular part:
\begin{equation*}
    \de\{u>0\}= \reg(u) \cup \Sigma(u).
\end{equation*}
Points in these sets can be characterized, for example, by density considerations:
\begin{align*}
&x_\circ \in \reg(u)& \Longleftrightarrow&\, & |\{u=0\}\cap B_r(x_\circ)|=\frac{|B_r|}{2}+o(r^n);\\
&x_\circ \in \Sigma(u)& \Longleftrightarrow&\,& |\{u=0\}\cap B_r(x_\circ)|=o(r^n).
\end{align*}
Caffarelli showed that $\reg(u)$ is relatively open in the free boundary and, locally, is a $C^1$ hypersurface (smoothness and analiticity were proved later in \cite{KN77}). On the other hand, $\Sigma(u)$ can always be covered, locally, by a single $C^1$ hypersurface. Thus, when we speak about ``regularity of $\Sigma(u)$'' we are actually speaking about the regularity of the manifold covering it. In this paper we will improve the smoothness of this hypersurface. We remark that $\Sigma(u)$ can display a very wild structure, as long as it is contained inside a single hypersurface, see example \eqref{eq:simpleexample}.  

\subsection{The singular set: important examples and known results} The following simple example shows that $\Sigma(u)$ can be rather wild. Furthermore, it can have Hausdorff dimension equal to $n-1$, thus it can be as ``large'' as $\reg(u)$.  Consider the function
\begin{equation}\label{eq:simpleexample}
u(x):= x_n^2 +h(x')\qquad\text{ for } (x',x_n)\in \R^{n-1}\times \R, 
\end{equation}
where $h\in C^\infty(\R^{n-1})$ is non-negative. Possibly multiplying $h$ by a small factor, $u$ solves \eqref{eq:intro_obst} (with some $f$ depending on $h$) and
$$
\Sigma(u)=\{x_n=0\}\cap \{h=0\},\quad \reg(u)=\emptyset.
$$
Hence $\Sigma(u)$ can be any closed subset of $\{x_n=0\}$. See point 2 below for even wilder examples due to Schaeffer (\cite{Sch76}), where the contact set has non-empty interior. In this paper we show that, locally, $\Sigma(u)$ is always contained in a $C^\infty$ hypersurface (the hyperplane $\{x_n=0\}$, in this example), except for an $n-2$ dimensional piece (the empty set, in this example). Before turning to the statements, let us try to give an overview of what is known about $\Sigma(u)$.
\begin{itemize}
    \item Concerning the pointwise structure, $\Sigma(u)$ consists precisely of those points $x_\circ \in \de \{u>0\}$ such that
    $$r^{-2}u(x_\circ +rx) \to p_{2,x_\circ}(x)\qquad \text{ as }r\downarrow 0,$$
    where $p_{2,x_\circ}$ is a convex and $2$-homogeneous polynomial with $\lap {p_{2,x_\circ}}=f(x_\circ)$. Thus, zooming in on  $x_\circ$, one sees the contact set $\{u=0\}$ clustering around the linear space $\{ p_{2,x_\circ}=0\}$. When $\dim\{ p_{2,x_\circ}=0\}=m$ for some integer $m\leq n-1$, this suggests that $\Sigma(u)$ displays, qualitatively, an $m$-dimensional structure at $x_\circ$.
    \item Concerning the local structure, when $n=2$ and $f$ is real analytic, Sakai gave a complete characterisation of the possible shapes of the free boundary around a singular point \cite{Sak93}. In brief, $\Sigma(u)$ is locally either an analytic curve, or an isolated point, around which $\de\{u>0\}$ is the union of at most two analytic arcs. In particular, $\Sigma(u)$ has codimension $1$ inside the free boundary. We remark that his approach relies on complex analysis techniques \cite{Sak91}. On the other hand, Schaeffer, in \cite{Sch76}, constructed examples of rather wild free boundaries, in the case where $f$ is ``just'' $C^\infty$. He showed that 
    $$\interior\left(\{u=0\}\cap \{x_n=0\}\right) \text{ and }\{u=0\}\cap \{x_n=0\}$$
    can be any nested couple of open and closed subsets of $\{x_n=0\}$. In particular, the contact set might form infinitely many cusps, which in turn produce an arbitrarily closed subset of $\{x_n=0\}$ as singular set. This shows that the sharpness of Sakai's results is due to analytic rigidity.
\item Despite the counterexamples of Schaeffer, it is still possible to obtain refined statements about the shape of $\Sigma(u)$. Caffarelli \cite{C77,C98} showed that $\Sigma(u)$ is locally contained in a $C^1$ hypersurface. More precisely, if we partition $\Sigma(u)$ into the strata
    $$
    \Sigma_m:=\left\{x_\circ\in \Sigma(u) :\dim\ker{A_{x_\circ}}=m \right\},\qquad\text{for } m=0,\ldots,n-1,
    $$
    then each $\Sigma_m$ is locally contained in an $m$-dimensional $C^1$ manifold. Building on Weiss \cite{W99} and Monneau \cite{M03}, Figalli and Serra extended these results in \cite{FS19}, by showing that $\Sigma(u)\setminus E$ can be covered by $C^{1,1}$ manifolds, where the excluded set $E$ has Hausdorff dimension at most $n-2$.
    \item By the implicit function theorem, this type of covering results are tightly linked with the fact that $u$ admits a polynomial expansion around singular points. In this perspective, Caffarelli showed that at each $x_\circ\in \Sigma(u)$ there holds
    $$
    u(x)=p_{2,x_\circ}(x-x_\circ)+\sigma(x-x_\circ)|x-x_\circ|^2,
    $$
    with an abstract dimensional modulus of continuity $\sigma$. This was improved in \cite{CSV17}, showing that $\sigma(x-x_\circ)\leq C|\log |x-x_\circ||^{-\eps_\circ}$ for some dimensional $C,\eps_\circ>0$. Figalli and Serra, in \cite{FS19}, proved that when $x_\circ\in \Sigma_{n-1}$, it holds $\sigma(x-x_\circ)\leq C|x-x_\circ|^{\alpha_\circ}$, and also $\sigma(x-x_\circ)\leq C|x-x_\circ|$ provided $x_\circ \in \Sigma(u)\setminus E$, where $\dim_{\mathcal H} E \leq n-2$. Similar results were recovered in \cite{Sav19}, with independent methods. This analysis was pushed further by Figalli, Ros-Oton and Serra in \cite{FRS19}, in the framework of ``generic'' regularity. They showed that for all $\eps>0$ small, there is a set $E\subseteq \Sigma(u)$ with Hausdorff dimension at most $n-2$, such that, if $x_\circ \in \Sigma(u)\setminus E$, then
    $$
    u(x)=P_{4,x_\circ}(x)+O(|x-x_\circ|^{5-\eps}),
    $$
    for some polynomial $P_{4,x_\circ}$ with $\lap P_{4,x_\circ}=1$. Unfortunately, the approach in \cite{FRS19} was blocked at
order 5, and new ideas were needed to go further: see Section \ref{subsubsec:almgrenintro} below for more details.
\end{itemize}
\subsection{Main results of this paper}
Our main contribution is a natural strengthening of the analysis of $\Sigma(u)$ conducted so far: building on \cite{FRS19}, we push the $C^{5-\eps}$ regularity all the way up to $C^\infty$, with a unified method, for a general smooth right hand side $f$. We remark that crucial steps in \cite{FRS19} do not directly ``bootstrap'' to a higher orders, and new ideas were required to deal with them. The model case to which our main Theorem applies is example \eqref{eq:simpleexample}.
\begin{theorem}\label{thm:maintheoremintro}
	Let $n\geq 2,\mu>0$ and $f\in C^\infty(B_1)$ be given, with $f\geq \mu$. Let $u\in C^{1,1}_{\loc}(B_1)$ be a solution of the obstacle problem \eqref{eq:intro_obst} and let $\Sigma$ be its singular set. Then there exists a closed set $\Sigma^\infty\subseteq\Sigma$ such that
	\begin{enumerate}[label={\upshape(\roman*)}]
		\item $\dim_{\mathcal H} \left(\Sigma\setminus\Sigma^\infty\right)\leq n-2$ (countable, if $n=2$);
		\item locally, $\Sigma^\infty$ is contained in one $(n-1)$-dimensional $C^\infty$ manifold.
	\end{enumerate} Moreover, at every point $x\in \Sigma^\infty$ the solution $u$ has a polynomial expansion of arbitrarily large order. That is, for every $x\in\Sigma^\infty$ and $k\in \N$ there exists a unique polynomial $P_{k,x}$ with $\deg P_{k,x} \leq k$ such that the expansion
	\begin{equation}
	\label{eq:taylorexpansionintro}
	u(x + h)= P_{k,x}(h) + O(|h|^{k+1})
	\end{equation}
	holds with a modulus of continuity locally depending only on $n,k,\mu,\|f\|_{C^{k+1}}$. We further have that $\lap P_{k+2,x}=f_{k,x}$ where $f_{k,x}$ is the $k^{th}$ Taylor polynomial of $f$ centered at $x$. Finally, the map $\Sigma^\infty\ni x\mapsto (P_{k,x})_{k\in\N}$ is smooth in the sense of Whitney.
\end{theorem}
We denoted with $\dim_{\mathcal H}$ the Hausdorff dimension and we refer to \cite[Section 3]{W34} for Whitney's definition of smoothness on a closed set. We remark that, in dimension 2, Theorem \ref{thm:maintheoremintro} can be also read as a version of Sakai's result for non-analytic right hand sides (see the proof of Corollary \ref{thm:Hele-Shawintro} for a comparison of the two results). 

Our analysis further shows that the set $\Sigma^\infty$ is extremely unstable with respect to monotone perturbations of the boundary datum. Indeed, following \cite{FRS19}, we prove:
\begin{theorem}\label{thm:instabilityintro}
Let $\{u^t\}_{t\in(-1,1)}$ be a family of solutions to \eqref{eq:intro_obst}, with $f$ independent from $t$, which is ``uniformly monotone'' in the sense that for every $t \in (-1,1)$ and any compact set $K\subseteq \partial B_1\cap \{u^t>0\}$ there exists $c=c(t,K)>0$ such that
\begin{equation}\label{eq:bbbb}
\min_{x\in K} \big( u^{t+h}(x)- u^t(x)  \big) \ge  c\, h , \qquad \text{for all }-1 <t< t+h< 1.
\end{equation}
With the notations of Theorem \ref{thm:maintheoremintro}, define the singular sets 
\begin{equation*}
\begin{array}{rcl}
\bm{\Sigma} &:=& \{ (x_\circ, t_\circ): x_\circ\in {\Sigma(u^{t_\circ})}\big\}, \vspace{1mm}\\
\bm{\Sigma^{\infty}} &:=& \big\{ (x_\circ, t_\circ): x_\circ\in {\Sigma}^{\infty}(u^{t_\circ})\big\}.
\end{array}
\end{equation*}
Then, denoting with $\pi_t\colon B_1\times (-1,1)\to (-1,1)$ and $\pi_x\colon B_1\times (-1,1) \to  B_1$ the standard projections, we have
\begin{enumerate}[label={\upshape(\roman*)}]
\item $\dim_{\mathcal H} \left(\pi_t(\bm{\Sigma^{\infty}})\right)=0$ in any dimension $n\ge 2$.
\item $\dim_{\mathcal H} \left(\pi_t(\bm{\Sigma})\right)=0$ in dimension $n=2$.
\item $\dim_{\mathcal H} \left(\pi_x(\bm{\Sigma\setminus \Sigma^{\infty}})\right)\leq n-2$, (countable, if $n=2$).
\end{enumerate}
\end{theorem}
Combining this result with Sakai's classification we get an improvement of \cite[Theorem 1.2]{FRS19}, concerning the generic regularity of the free boundary in the planar Hele-Shaw flow.
\begin{corollary}\label{thm:Hele-Shawintro}
Let $O\subseteq \R^2$ be an open and bounded set with Lipschitz boundary and $\Omega:=\R^2\setminus \overline O$. For each $t>0$ let $u^t$ be a weak solution of
\begin{equation}
\begin{cases}
\lap u^t =\chi_{\{u^t>0\}}&\quad \text{in } \Omega, \\
u^t = t & \quad\text{on } \de \Omega, \\
u^t  \geq 0 & \quad \text{in } \Omega. 
\end{cases}
\end{equation}
Then the set of $t\in(0,\infty)$ such that $\Sigma(u^t)\neq\emptyset$ is countable.
\end{corollary}

\subsection{On the proofs of the main results}
Let us now explain the main ideas used in the proof of Theorem \ref{thm:maintheoremintro}. As pointed out, the key feature is the Taylor expansion \eqref{eq:taylorexpansionintro}. Furthermore, we can work in the top-dimensional stratum $\Sigma_{n-1}$, as the lower strata $\Sigma_m$, $m\in\{1,\ldots,n-2\}$, have Hausdorff dimension at most $m$, thanks to Caffarelli's covering result. In brief, we will perform a blow-up analysis on the functions $u-\curlyP_k$, where $\curlyP_k$ are suitable polynomials. The core tool in this blow-up analysis is an Almgren-type monotonicity formula. 
\subsubsection{Polynomial Ans\"atze} 
Following \cite{FRS20}, we construct $\curlyP_k$: the prototypical $k^{th}$ Taylor polynomials of $u$, at $0\in\Sigma_{n-1}$. These polynomials should be approximate solutions of 
\eqref{eq:intro_obst}, that is to say
\begin{equation*}
\begin{cases}
\lap \curlyP_k =f +O(|x|^{k-1})&\text{in } B_1,\\
\curlyP_k\ge -O(|x|^{k+2}) \quad &\text{in } B_1.
\end{cases}
\end{equation*}
Furthermore, they should have an $n-1$ dimensional zero set (we are in the top-dimensional stratum $\Sigma_{n-1}$). Together with non negativity, this suggests that $\curlyP_k$ is almost a square:
$$
\curlyP_k=\left(\text{polynomial}\right)^2+O(|x|^{k+2}).
$$
The coefficients of $\curlyP_k$ can be chosen with some freedom, which can be used to modulate the shape of the hypersurface $\{\curlyP_k =0\}$ around $0$. Notice that by Caffarelli's analysis $\curlyP_0=\curlyP_1=0$ and (in suitable coordinates) $\curlyP_2=x_n^2/2$. We will see that $\curlyP_3$ will have a more complicated form:
$$
\curlyP_k=(x_n+p_3/x_n+x_n R_2)^2+O(|x|^{k+2}),
$$
where $p_3$ is any 3-homogeneous harmonic polynomial odd in $x_n$, and $R_2=R_2[p_3]$ is a 2-homogeneous polynomial determined uniquely by $p_3$.
\subsubsection{Almgren monotonicity formula}\label{subsubsec:almgrenintro} In \cite{FS19}, it was first noticed that the Almgren frequency function of $u-p_2$,
\begin{equation}\label{eq:monotonicity2}
    r\mapsto \frac{\|r\nabla (u-p_2)(r\,\cdot)\|_{L^2(B_1)}}{\|(u-p_2)(r\,\cdot)\|_{L^2(\de B_1)}}=:\phi(r,u-p_2),
\end{equation}
is increasing. This cornerstone observation paved the way for most of the results of \cite{FS19}. Similarly, the key observation that allowed us to prove the expansion \eqref{eq:taylorexpansionintro} is that for all $k\geq 2$ (a version of) the Almgren frequency function is (almost) monotone on the functions of the form $u-\curlyP_k$, for all $k\geq 2$. This fact is crucial and understanding how to prove it for a generic $k$ is one of the key contributions of this paper. Actually, the case $k=3$ was tackled in \cite{FRS19} (with slightly different polynomials $\curlyP_k$), introducing a modified frequency function $\phi^\gamma$ and showing that it is almost increasing. More precisely:
\begin{equation}\label{eq:monotonicity34}
r\mapsto    \phi^\gamma(r,u-\curlyP_k)+ Cr^\eps, \quad \text{ with }\phi^\gamma(r,v):=\frac{\|\nabla v_r \|^2_{L^2(B_1)}+\gamma r^{2\gamma}}{\|v_r\|^2_{L^2(\de B_1)}+r^{2\gamma}},
\end{equation}
is increasing. Here $v_r:=v(r\,\cdot\,)$, $C,\eps$ are positive constants and $\gamma \in (k+1,k+2)$ is the truncation parameter. Let us explain what blocked their approach to $k=3$. While monotonicity of \eqref{eq:monotonicity2} is neat computation, monotonicity of \eqref{eq:monotonicity34} for $k=3$ is much trickier and ultimately relies on a Lipschitz estimate of the following kind
\begin{equation}\label{eq:lipestimateold}
    r\|\der(u-\curlyP_3)\|_{L^\infty(B_r)} \leq C  \left(\|(u-\curlyP_3)(r\, \cdot)\|_{L^2(B_{2})}+ r^{5}\right),
\end{equation}
for some constant $C$ independent from $u,\curlyP_3$ and $r\in(0,1/4)$. In \cite{FRS19}, this estimate was deduced from a more powerful semi-convexity estimate along circular vector fields \cite[Lemma 4.7]{FRS19}, which could never work for $k>3$. The technical reason is that the laplacian ``at most'' commutes with rotations (which were generated by such vector fields). Thus our efforts were directed to an independent proof of \eqref{eq:lipestimateold}: in Section \ref{subsec:lipschitz} we will prove the following delicate estimate, which is still enough to get the monotonicity of $\phi^\gamma$. For all $k\geq 2$ it holds
\begin{equation}\label{eq:lipestimateintro}
r\|\der(u-\curlyP_k)\|_{L^\infty(B_r)} \leq C  \left(\|(u-\curlyP_k)(r\, \cdot)\|_{L^2(B_{\theta})}+ r^{k+2}\right)^{1-\beta},
\end{equation}
where $\beta>0$ can be chosen arbitrarily small and $C,\theta>0$ are constants independent from $u$ and $r\in (0,1/\theta)$. 

\subsubsection{Blow-up analysis} The monotonicity formula allows us to pursue a blow-up analysis to every order. Following \cite{FS19}, we classify the possible blow-ups, that is study the possible limits of the normalized sequence
    $$
    \widetilde v_{r_\ell}:=\frac{v_{r_\ell}}{\|v_{\ell}\|^2_{L^2(\de B_1)}}\qquad \text{ as }r_\ell\downarrow 0, \quad \text{ where }v_{r}:=(u-\curlyP_k)(r\cdot).
    $$ 
    We show that $\widetilde v_{r_\ell}\to q$ where $q$ is a nontrivial global solution of a certain PDE: the Signorini problem. Furthermore, $q$ is homogeneous of degree $\phi^\gamma(0^+,v)$. The blow-up can be performed at each point of $\Sigma_{n-1}$ of course, but $q$ could be a non-polynomial function or even have non-integral homogeneity. At the points where this happens, there cannot be a Taylor polynomial of order $k+1$, so the expansion \eqref{eq:taylorexpansionintro} must stop. These points must be shown to be ``rare''. 
    
\subsubsection{Dimension reduction} We show that, outside of a set of dimension at most $n-2$, it holds $\phi^\gamma(0^+,u-\curlyP_{k})=k+1$ for a suitable choice of $\curlyP_k$ and the blow-up $q$ is an harmonic polynomial vanishing on $\{p_{2}=0\}$ (a particular class of solutions of the Signorini problem). This allows us to determine the next Ansatz $\curlyP_{k+1}$ in terms of $\curlyP_k$ and $q$, and prove the Taylor expansion up to order $k+1$ with remainder $o(r^{k+1})$. The various dimension reduction techniques are inspired from \cite{FS19}, but new barrier arguments are introduced to deal with the points with even frequency.

\subsection{Structure of the Paper}
In Section \ref{sec:preliminaries} we fix the notation and recall some basic results on the obstacle problem and the Signorini problem. In Section \ref{subsec:polynomial ansatz} we give the construction of the polynomial Ansatz $\curlyP_k$. In Section \ref{subsec:lipschitz} we prove our Lipschitz estimate \ref{eq:lipestimateintro}. In Section \ref{sec:monotonicity frequency} we prove the almost monotonicity of the truncated frequency $\phi^\gamma(\,\cdot\,,u-\curlyP_k )$. In Section \ref{sec:blowup} we perform and classify the blow-ups. In Section \ref{sec:dimensionreduction} we perform the dimension reductions distinguishing various cases, the proof of Theorem \ref{thm:maintheoremintro} is given in Section \ref{sec: geometry}. In Section \ref{sec:family} we give the proof of the instability result Theorem \ref{thm:instabilityintro}.

In Appendix \ref{sec:importantlemma} we reprove a result from \cite[Remark 2.14]{FS19} for a general right hand side. In Appendix \ref{sec:rhs f} we explain line by line, which modifications are needed for a smooth right hand side in the previous proofs. Finally, Appendix \ref{sec:auxiliarylemmas} contains two technical lemmas.
\section{Preliminaries}
\label{sec:preliminaries}
\subsection{Notation}
We work in $\R^n$ endowed with its euclidean structure and assume $n\geq 2$. We will often perform blow-ups: given a function $v\colon B_1\to \R$ we set $v_r:=v(r\,\cdot\,)$ which is defined in $B_{1/r}$, the parameter $r>0$ is thought to be small. We remark that $\nabla v_r=r\, (\nabla v)_r$. We will sometimes write $X\lesssim_{a,b}Y$ meaning that $X\leq C Y$ for some constant $C>0$ which depends only on $a$ and $b$.

\subsection{Known results}
Fix $\mu>0$ and a function $f\in C^\infty(B_1)$ such that $f\geq \mu$ in $B_1$. We will denote with $u$ any solution of
\begin{equation}\label{eq:obstacle}
\begin{cases}
\lap u = f \chi_{\{u>0\}}&\text{ in } B_1,\\
u\ge 0 \quad &\text{ in } B_1,\\
0\in \de\{u>0\}.
\end{cases}
\end{equation}
The last condition is added for normalization purposes as we want to stay away from $\de B_1$. We recall some basic properties of the solution $u$, relying on the classical theory by Caffarelli \cite{C77,C98}, see also \cite[Section 3]{FRS19} for a summary of the known results. There exists $C=C(n,\mu,\|f\|_{L^\infty(B_1)})>0$ such that
\begin{equation}\label{eq:optimalreg+nondeg}
\|u\|_{C^{1,1}(B_{1/2})} \leq C   \qquad \text{ and }\qquad \sup_{B_{1/2}} u \geq \frac 1 C.
\end{equation}
Thus we will assume throughout the paper that $u\in C^{1,1}_{\loc}(B_1)$. We remark that the problem has a natural scaling, in fact for any $r>0$ we have that $r^{-2}u_r$ solves 
\begin{equation*}
\begin{cases}
\lap (r^{-2}u_r) = f_r \chi_{\{u_r>0\}}&\text{ in } B_{1/r},\\
u_r\ge 0 \quad &\text{ in } B_{1/r},\\
0\in \de\{u_r>0\}.
\end{cases}
\end{equation*}

The free boundary $\de \{u>0\}$ consists of regular points (in the neighbourhood of which $\de \{u>0\}$ is an analytic hypersurface) and singular points $\Sigma\subseteq \de \{u>0\}$ (at which the volume density of $\{u=0\}$ is $0$). It is well known that the singular points are characterized by the condition that the blow-up
\[
p_{2,x_\circ}(x) := \lim_{r\to 0} r^{-2}u(x_\circ+ rx)
\]
exists and is a convex $2$-homogeneous polynomial with $\lap p_{2,x_\circ} \equiv f(x_\circ)$. When $x_\circ= 0$ we denote the blow-up simply by $p_2$. The singular set $\Sigma$ stratifies according to $\dim \{ p_{2,x_\circ}=0\}$. The strata
$$\Sigma_m:= \big\{x_\circ\in \Sigma : \dim \{p_{2,x_\circ} =0\}=m\big\}\qquad \text{ for } m=0,\dots,n-1$$
are locally contained in $m$-dimensional $C^1$ manifolds. As we want to prove a statement ``up to sets of co-dimension 2'', we will be mostly interested in the top dimensional stratum $\Sigma_{n-1}$. 

The following lemma is crucial for our analysis. It shows that in $\Sigma_{n-1}$, the rate of convergence of $u$ to its blow up is more than quadratic. It was proved in \cite[Remark 2.14]{FS19} for $f\equiv 1$, for completeness we give the proof for a general $f\in C^\delta(B_1)$ in the Appendix.
\begin{lemma}\label{lem:importantlemma}
	Assume that $0\in\Sigma_{n-1}$ and $r^{-2}u_r\to p_2$. Then there are $C,\alpha_\circ>0$ such that
	 \begin{equation}
	\label{eq: important eq1}
	\sup_{B_r}|u-p_{2}|\leq C\, r^{2+2\alpha_\circ} \quad \text{ for all } r\in (0,1/2).
	\end{equation}
	In particular, we have that
	\begin{equation}
	\label{eq: important eq2}
\{u_r=0\}\cap B_1 \subseteq \big\{ x: {\dist}(x, \{ p_2=0\}) \leq C r^{\alpha_\circ}\big\}, \quad \text{ for all } r\in (0,1).
	\end{equation}
	The constants $C,\alpha_\circ$ depend only on $n,\mu,\delta,\|f\|_{C^\delta(B_1)}$ where $0<\delta\leq 1$ can be freely chosen.
\end{lemma}
Notice that \eqref{eq: important eq2} immediately follows from \eqref{eq: important eq1} because ${\dist}(\, \cdot\,, \{ p_2=0\})^2=p_2$, since $0\in\Sigma_{n-1}$. 

\subsection{Truncated frequency function}
We will make extensive use of the following functionals. For $w\in C^{1,1}_{\loc}(B_1)$, $r\in (0,1)$ and a parameter $\gamma\ge0$, let us define the adimensional quantities 
\begin{equation}
\label{eq:H and D}
 D(r,w) :=   r^{2-n} \int_{B_r} |\nabla w|^2 =\int_{B_1} |\nabla w_r|^2, \quad H(r,w):= r^{1-n}\int_{\partial B_r} w^2 =\int_{\partial B_1} w_r^2,
\end{equation} 
and the truncated frequency function
\begin{equation}
\label{eq:trun freq}
 \phi^\gamma(r,w)  : = \frac{D(r,w) + \gamma r^{2\gamma}}{H(r,w)+ r^{2\gamma}},
\end{equation} which has been introduced in \cite{FRS19}.
By \cite[Lemma 2.3]{FRS19}, the following formula is valid for all $w\in C^{1,1}_{\loc}(B_1)$ and $r\in (0,1)$ 
$$
 \frac{d}{dr} \phi^\gamma(r,w) \ge \frac{2}{r}  \frac{ \big( r^{2-n} \int_{B_r}w\Delta w\big)^2  + E^\gamma(r,w) }{ \big( H(r,w) + r^{2\gamma}\big)^2 },
$$
where
\begin{equation*}
E^\gamma(r,w) :=   \left(r^{2-n} \int_{B_r} w\lap w\right) \big( D(r,w)+  \gamma r^{2\gamma}  \big) - \left(r^{2-n} \int_{B_r} (x\cdot \nabla w)\lap w\right) \big(H(r,w) +r^{2\gamma} \big).
\end{equation*}
Thus we have
\begin{equation}
    \label{eq:der trun freq}
    \frac{d}{dr} \phi^\gamma(r,w) \ge \frac{2}{r}  \frac{\int_{B_1} \left(\phi^\gamma(r,w) w_r - x\cdot \nabla w_r\right)\lap w_r }{ H(r,w) + r^{2\gamma} }.
\end{equation}
We recall from \cite{FRS19} the following result which says, roughly speaking, that the value of $\phi^\gamma(\cdot,v)$ corresponds to the power at which $H(\cdot,v)$ grows. This Lemma will be used extensively to pass from $L^2$ norms over spheres to $L^2$ norms over thick shells.
\begin{lemma}[{\cite[Lemma 4.1, Remark 4.2]{FRS19}}]\label{lem:H ratio}
Let $R\in (0,1)$,  and let  $w:B_R\to  [0,\infty)$ be a $C^{1,1}$ function.
Assume that for some $\eps\in (0,1)$ and a constant $C_\circ>0$ we have
$$
\frac{d}{dr} \Big(\phi^\gamma(r,w) + C_\circ r^{\eps} \Big)\ge \frac 2 r  
\frac{\left(r^{2-n} \int_{B_r} w\lap w\right)^2}{\big(H(r,w) + r^{2\gamma}\big)^2} \quad \text{ for all } \,r\in (0,R).
$$
Then the following holds:
\begin{enumerate}[label={\upshape(\alph*)}]
	\item Suppose  that $0<\underline \lambda \le \phi^\gamma(r, w)\le \overline \lambda$  for all $r\in (0,R)$. Then,  for any given $\delta >0$ we have
	$$
	\frac{1}{C_\delta} \left(\frac{R}{r} \right)^{2\underline \lambda-\delta}  \le \frac{H(R,w)+R^{2\gamma}}{H(r,w)+r^{2\gamma}} \le  C_\delta  \left(\frac{R}{r} \right)^{2\overline \lambda+\delta}\quad \text{for all}\quad r\in(0,R),
	$$
	where $C_\delta$ depends on $n$, $\gamma$, $\eps$, $\overline \lambda$, $C_\circ$, $\delta$.
	\item Assume in addition that
	$$
	\frac{ r^{2-n} \int_{B_r} w\lap w }{H(r,w) + r^{2\gamma}}  \ge -C_\circ r^{\eps}\quad \text{for all}\quad r\in(0,R).
	$$
	Then, for $\lambda_*:=\phi^\gamma(0^+,w)$, we have $\lambda_*\leq \gamma$ and
	$$
	e^{-\frac{C_\circ}{\eps^2}}   \left(\frac{R}{r}\right)^{2 \lambda_*} \le \frac{H(R, w) +R^{2\gamma}}{H(r,w) +r^{2\gamma}}.
	$$
\end{enumerate}
\end{lemma}

\subsection{The Signorini problem}
\label{subsec:signorini}
The Signorini problem, called also the thin obstacle problem, consists of the following system of PDEs
 \begin{equation}\label{eq:signoriniprelim}
\begin{cases}
\lap q\leq 0\ \text{ and }\ q\lap q= 0 & \text{ in }\R^n,\\
\lap q=0 &\text{ in }\R^n\setminus L,\\
q\geq 0 &\text{ on }L,
\end{cases}
\end{equation}
where $L\subseteq \R^n$ is a hyperplane and $q$ is at least continuous. Recall that the following regularity is known (see \cite{AC06}) for weak solutions, if $L=\{x_n=0\}$ then $q\in C^{1,1/2}_{\loc}(\{x_n\geq 0\})$.

For each solution $q$ we will consider its \emph{singular set}, defined by
\begin{equation}\label{eq:definitionsingularsetsignorini}
\Sigma (q):=\left\{x\in L : q=|\nabla q|=0 \right\}.
\end{equation}
Using Proposition \ref{pro:signorini} it is easy to check that if $n=2$ and $\lambda \in \N, \lambda\geq 2$ then $\Sigma(q^{\text{even}})=\{0\}$.

We will be interested in homogeneous solutions, so for every $\lambda\geq 0$ and every hyperplane $L\subseteq \R^n$ let us define
\begin{align}\label{eq:defSignsol}
\Sign_\lambda(L):=\left\{q\in W^{1,2}_{\loc}\cap C^0_{\loc}(\R^n) : q\text{ is }\lambda\text{-homogeneous and solves }\eqref{eq:signoriniprelim}\right\}.
\end{align}
We will use the following characterization of homogeneous solutions:
\begin{lemma}\label{lem:homogeneousconstantfreq}
Let $q\in W^{1,2}_{\loc}\cap C^0_{\loc}(\R^n)$ be a weak solution of \eqref{eq:signoriniprelim} and let $\lambda\geq 0$. Then $q$ is $\lambda$-homogeneous if and only if
$$
\frac{D(r,q)}{H(r,q)}=\lambda\quad \text{ for all }r>0.
$$
\end{lemma}
\begin{proof}
Setting to zero the derivative with respect to $r$ of the left hand side one formally gets $q(x)=\lambda x\cdot \nabla q(x)$. One can easily make the computation rigorous using the $C^{1,1/2}$ regularity of $q$. 
\end{proof}
Every $q\in \Sign(\nu^\perp)$, where $\nu$ is a unit vector, can be split into its even and odd part with respect to $L$, namely 
\begin{align}\label{eq:qevenodd}
q^{\text{even}}(x):=\frac{q(x)+q(x-2(x\cdot\nu)\nu)}{2},\quad  q^{\text{odd}}(x):=\frac{ q(x)-q(x-2(x\cdot\nu)\nu)}{2},
\end{align}
so that $q=q^{\text{even}}+q^{\text{odd}}$. It is easy to show that $q^{\text{even}}$ and $q^{\text{odd}}$ solve \eqref{eq:signoriniprelim} separately thus it is natural to define
\begin{align*}
\Signeven_\lambda(L)&:=\left\{q\in \Sign_\lambda(L) : q\text{ is even with respect to } L\right\},\\
\Signodd_\lambda(L)&:=\left\{q\in \Sign_\lambda(L) : q\text{ is odd with respect to } L\right\}.
\end{align*}
When it is not relevant we will drop the dependence from $L$. Below we collect some information about these sets.
\begin{proposition}\label{pro:signorini}
	For every $m\in\N$ the following holds.
	\begin{enumerate}[label={\upshape(\roman*)}]
		\item Every element of $\Sign_{2m+1}(L)$ vanishes on the obstacle $L$.
		\item $\Signodd_\lambda(L)$ consists exactly of those $\lambda$-homogeneous harmonic polynomials that vanish on $L$, thus it's empty when $\lambda\notin \N$.
		\item $\Signeven_{2m}(L)$ consists exactly of those $2m$-homogeneous harmonic polynomials that are  non-negative on $L$.
		\item If $q\in\Signeven_{2m+1}(e_n^\perp)$ then  $q(x)=-|x_n|\left(q_0(x')+x_n^2q_1(x)\right)$ where $q_0$ and $q_1$ are polynomials such that $q_0\geq 0$ and $\lap\left(-x_n q_0(x')+x_n^3q_1(x)\right)=0$.
		\item The (real) values of $\lambda$ for which $\Signeven_\lambda(L)$ is not empty are known only in dimension $n=2$, in which case we have $\lambda\in \N\cup \left\{2m+\frac{3}{2}: {m\in\N}\right\}$.
	\end{enumerate}
\end{proposition}
\begin{proof} For (i) see \cite[Lemma 5.1]{FRS19}.
	To show (ii) notice that $q$ is harmonic in an half-space and coincides with its odd reflection, thus it is harmonic everywhere. The third point is proven in \cite[Lemma 1.3.4]{GP09}. For (iv) see Appendix B in \cite{FRS19}. The last point follows by separating variables and explicitly solving the resulting ODE. 
\end{proof}

\section{Lipschitz estimates}
	 For the sake of readability, we deal first with the case $f\equiv 1$ and $\mu=1$. A list of notational changes needed to address a general $f$ is given in the Appendix. This section is devoted to the derivation of the Lipschitz estimate \eqref{eq:lipestimateintro}, which contains the most original part of this work.

\subsection{Polynomial Ansatz}
\label{subsec:polynomial ansatz}
We denote with $V_j$ the vector space of homogeneous polynomials of degree $j\geq 1$. We introduce the projection map $\pi_j\colon \R[x]\to V_j$, that sends a polynomial to its $j$-homogeneous part, and the map $\pi_{\leq j}$, which truncates it at degree $j$. We define for $k\geq 2 $ the set $\PP_k\subseteq V_2\times \ldots \times V_k$ by saying that $(p_2,\ldots,p_k)\in \PP_k$ if and only if
\begin{enumerate}[label={\upshape(\roman*)}]
\item $p_j$ is a $j$-homogeneous polynomial, for each $2\leq j\leq k$;
\item $p_2\geq 0,\lap p_2=1$ and $\dim \{p_2=0\}=n-1$;
\item $\lap p_j=0$ and $p_j$ vanish on $\{p_2=0\}$ for each $3\leq j\leq k$.
\end{enumerate}
Notice that any $p_2$ satisfying point 2 must be of the form $p_2(x)=\frac 1 2 (\nu \cdot x)^2$ for some unit vector $\nu$, of course $\nu$ is not unique as we can always choose $-\nu$, but that is the only freedom we have.
Furthermore, for every $(p_2,\ldots,p_k)\in V_2\times\ldots\times V_k$ we set
$$
|(p_2,\ldots,p_k)|:=\sum_{2\leq j\leq k}\|p_j\|_{L^2(\de B_1)}.
$$
\begin{lemma}\label{lemma:ansatzalg}
Let $k\geq 2$ and $(p_2,\ldots,p_k)\in \PP_k$ be given. Then there exists a unique collection of polynomials
$$
(R_1,\ldots R_{k-1})\in V_1\times\ldots\times V_{k-1},
$$
such that the following holds. If $p_2(x)=\frac 1 2 (\nu \cdot x)^2$ for some unit vector $\nu$ and we define
$$
\curlyA_{k,\nu}(x):= (\nu \cdot x) +\sum_{j=1}^{k-1} (\nu \cdot x)R_j(x)+\sum_{j=3}^{k} \frac{p_j(x)}{(\nu \cdot x)},
$$
then $\lap \left(\frac1 2 \curlyA_{k,\nu}^2\right) =1+O(|x|^k)$. Furthermore, each $R_j$ is determined only by $(p_2,\ldots, p_{j+1})$ and does not depend on which $\nu$ we choose, so that $\curlyA_{k,-\nu}=-\curlyA_{k,\nu}$. Finally, $\frac 1 2\curlyA_{2,\nu}^2=p_2$.
\end{lemma}
\begin{proof}
We prove the full statement by induction on $k$, beginning with $k=2$. We compute
\begin{equation*}
 \lap \left(\frac1 2 \curlyA_{2,\nu}^2\right)= 1+\lap (2p_2 R_1)+O(|x|^2),
\end{equation*}
thus $R_1$ must solve $\lap(p_2 R_1)=0$, this is true if and only if $R_1=0$, as we can see with the following general argument. For $m\geq 1$, we consider the linear map $\delta_m\colon V_m\to V_m$ given by $\delta_m (q):=\lap(p_2 q)$. We claim that $\delta_m$ is a isomorphism. Indeed, if $\delta_m(q)=0$, then the polynomial $p_2 q$ is an harmonic function that vanishes on the hyperplane $\{p_2=0\}$ along with its normal derivative, thus $p_2q\equiv 0$ (by reflection, $p_2 q$ is both even and odd with respect to $\{p_2=0\}$, thus $q\equiv 0$). As $V_m$ has finite dimension, the map $\delta_m$ is invertible. In particular, $R_1=\delta_1^{-1}(0)=0$, whichever $p_2$ or $\nu$ are.

Assume that the full statement is proved up to some $k\geq 2$. Fix $(p_2,\ldots,p_{k+1})\in \PP_{k+1}$ and $\nu$, and for simplicity set $x_\nu:=(\nu \cdot x)$. Notice that for every $(R_1,
\ldots,R_{k})$ we have $\curlyA_{k+1,\nu}=\curlyA_{k,\nu}+ x_\nu R_{k} + p_{k+1}/x_\nu.$ A direct computation again gives
\begin{align*}
 \lap \left(\frac 1 2 \curlyA_{k+1,\nu}^2\right)&= \pi_{\leq k-1}\left(\lap \left(\frac 1 2 \curlyA_{k,\nu}^2\right)\right)+ \pi_{k}\left( \lap\left(\frac 1 2 \curlyA_{k,\nu}^2\right)\right)\\
 &\quad+ \lap p_{k+1} + \lap\frac{p_3p_{k+1}}{2p_2}+ \lap (2p_2 R_{k})+O(|x|^{k+1}).
\end{align*}
Hence we have $\lap \frac1 2 \curlyA_{k+1,\nu}^2 =1+O(|x|^{k+1})$ if and only if
$$
\begin{cases}
\lap \frac 1 2 \curlyA_{k,\nu}^2=1+O(|x|^k), \\
\pi_{k}\left( \lap\frac 1 2 \curlyA_{k,\nu}^2\right) + \lap\frac{p_3p_{k+1}}{2p_2}+ \lap 2p_2 R_{k}=0,
\end{cases}
$$
by inductive assumption the first equation determines uniquely  $(R_1,\ldots,R_{k-1})$ and thus $\curlyA_{k,\nu}$. The second equation then determines uniquely $R_k$, indeed, as in the base step, we set 
$$R_k:=-\frac{1}{2}\delta_k^{-1}\left(\pi_{k}\lap \left( \frac 1 2 \curlyA_{k,\nu}^2\right)+\lap \frac{p_3p_{k+1}}{2p_2}\right).$$
 Finally, by inductive assumption $\frac 1 2 \curlyA_{k,\nu}^2=\frac 1 2 \curlyA_{k,-\nu}^2$, thus it's manifest that $R_k$ does not depend on which $\nu$ we had chosen.
\end{proof}
This Lemma shows that we can construct a function $\frac 1 2\curlyA_k^2\colon \PP_{k}\to \R[x]$ defined by \begin{equation}\label{eq:defcurlyA}
\curlyA_k^2\colon (p_2,\ldots,p_k) \mapsto  \curlyA_{k,\nu}^2,\text{ where }p_2(x)=(\nu\cdot x)^2.
\end{equation} 

\begin{definition}\label{def:curlyP}
Given $k\geq 2$ we define $\curlyP_k \colon \PP_{k} \to \R[x]$ by
\begin{equation}\label{eq:defcurlyP}
    \curlyP_k(p_2,\ldots,p_k):=\pi_{\leq k+1} \left(\frac 1 2 \curlyA_{k}^2\right),
\end{equation}
where $\curlyA_{k}^2=\curlyA_{k,\pm\nu}^2$ are constructed from $(p_2,\ldots,p_k)$ as in Lemma \ref{lemma:ansatzalg}.
\end{definition}
The following proposition gives some simple properties of $\frac 1 2\curlyA_k^2$ and $\curlyP_k$, given a unit vector $e$ we denote $\de_e u = e \cdot \nabla u$. 
\begin{proposition}\label{prop:curly A vs curly P} 
Let $k\geq 2, (p_2,\ldots,p_k)\in \PP_k$ and $\tau>0$ be such that $ |(p_2,\ldots,p_k)|\leq \tau$. Choose some unit vector $\nu$ for which $p_2(x)=\frac 1 2 (\nu\cdot x)^2$. Then the polynomials $\frac{1}{2}\curlyA_k^2(p_2,\ldots,p_k)$ and $\curlyP_k(p_2,\ldots,p_k)$ satisfy
\begin{enumerate}[label={\upshape(\roman*)}]
    \item $\lap \curlyP_k=1$ and  $\de_e(\frac 1 2\curlyA_k^2)=\de_e \curlyP_k + O(|x|^{k+1})$ for any unit vector $e$.
    \item We have that
    $
    \curlyP_k(p_2,\ldots,p_k)=\curlyP_{k-1}(p_2,\ldots,p_{k-1})+ p_k+O(|x|^{k+1}).
    $
     \item For all $|x|\leq r_0$ we have $\frac 1 2 \leq |\de_\nu \curlyA_{k,\nu}(x)| \leq 2$ and thus
     \begin{equation*}
         \frac 1 2 |\curlyA_k(x)|\leq |\de_\nu\left(\tfrac 1 2 \curlyA_k^2(x)\right)| \leq 2 | \curlyA_{k}(x)|,
     \end{equation*}
    where $r_0=r_0(n,k,\tau)\in (0,1)$.
    \item If $u$ is a solution as in \eqref{eq:obstacle}, $0\in \Sigma_{n-1}$ and $r^{-2}u(r\,\cdot\,)\to p_2$ then by \eqref{eq: important eq1} we have for all $0<r<1/2$
    $$
    \sup_{B_{r}\cap\{u=0\}}|\de_\nu\curlyP_k|\leq C r^{1+\alpha_\circ},$$ for some constant $C=C(n,k,\tau)>0$.
\end{enumerate}
\end{proposition}
\begin{proof}
  The first point follows immediately from $\lap( \frac 1 2 \curlyA_k^2)=1+O(|x|^{k})$ and the fact that $\curlyP_k-\frac 1 2 \curlyA_k^2=O(|x|^{k+2})$. The second point can be easily checked by direct computation using the structure of the polynomial $\curlyA_{k,\nu}$. For (iii) we compute $\de_\nu \curlyA_{k,\nu} =\de_\nu( \nu\cdot x +O(|x|^2))=1+O(|x|)$. Lastly, for the fourth point note that by construction $\curlyP_k= p_2+ O(|x|^3)$ and apply Lemma \ref{lem:importantlemma}. 
\end{proof}

\subsection{Regularity estimates near the free boundary}
\label{subsec:lipschitz}
We now turn to the Lipschitz estimate on functions of the form $u-\curlyP_k$, these estimates are a priori in the sense that $\curlyP_k$ is generic. 
\begin{proposition}\label{prop:lipschitz}
Let $u$ be a solution of the obstacle problem \eqref{eq:obstacle} with $f\equiv 1,\mu=1$, and suppose $0\in \Sigma_{n-1}$. Let $k\geq 2$ be an integer, $\tau>0$ and $\beta\in (0,\tfrac{\alpha_\circ^2}{k+2})$. Let $(p_2,\ldots,p_k)\in \PP_k$ such that $|(p_2,\ldots,p_k)|\leq \tau$. Suppose that 
$$r^{-2}u(r\cdot\,)\to p_2=\frac 1 2 x_n^2 \quad\text{ and set }\quad v:=u-\curlyP_k,$$ where
$\curlyP_k=\curlyP_k(p_2,\ldots,p_k)$ is the polynomial Ansatz from Definition \ref{def:curlyP}. Then
\begin{enumerate}[label={\upshape(\roman*)}]
    \item there are $C>0$ and $r_0\in(0,1)$ depending on $n,k,\tau$, such that for each $j=1,\dots,n-1$ and $0<r<r_0$ it holds
\begin{equation}\label{eq:tangentiallip}
 \|\de_j v_r\|_{L^\infty(B_{1})}\leq C \big(\|v_{ \theta r}\|_{L^2(B_2\setminus B_{ 1/2})}+r^{k+2}\big),
\end{equation} where $\theta=\theta(n,k).$
\item there are $C>0$ and $r_0\in(0,1)$ both depending on $n,k,\tau,\beta$, such that for every $0<r<r_0$ it holds
\begin{equation}\label{eq:normallip}
 \|\de_n v_r\|_{L^\infty(B_{1})}\leq C \big(\|v_{ \theta r}\|_{L^2(B_2\setminus B_{ 1/2})}+r^{k+2}\big)^{1-\beta},
\end{equation}
 where $\theta=\theta(n,k,\beta).$  
\end{enumerate} 
\end{proposition}
As this result will be crucial let us explain briefly its proof. As $p_2=x_n^2/2$, we split $\R^n$ into the``tangential'' directions $e_1,\ldots,e_{n-1}$ and the ``normal'' direction $e_n$. The idea is to study the quantity $\sup_{B_r\cap \{u=0\}} |\de_n\curlyP_k|$ for small $r$. 

Geometrically this quantity tells us how much the zero set of $\curlyP_k$ is a good approximation of the contact set $\{u=0\}$ around the origin, see Proposition \ref{prop:curly A vs curly P}.

Analytically, $\sup_{B_r\cap \{u=0\}} |\de_n\curlyP_k|$ is crucial because it will be the ``pivot'' linking Lipschitz estimates along ``tangential'' directions with the one along the ``normal'' direction; let us see how.

First, taking difference quotients along tangential directions, we will prove that when $j\neq n$ we have
$$
r\|\de_j v\|_{L^\infty(B_r)}\lesssim_{n,k,\tau} r^2\cdot \sup_{\{u=0\}\cap B_r}|\de_n\curlyP_k|+\|v_r\|_{L^2(B_2\setminus B_{1/2})}+r^{k+2},
$$
the fact that we have $r^2$ (and not $r$!) in front of $\sup_{\{u=0\}\cap B_r}|\de_n\curlyP_k|$ is the crucial gain, peculiar of the tangential derivatives.

Now we have to bound $\sup_{\{u=0\}\cap B_r}|\de_n\curlyP_k|$ from above. As this is much more complex, we just give the heuristics. First, notice that, as $u$ is $C^1$, we have $\de_n\curlyP_k\equiv\de_n v$ in $\{u=0\}$. As $v$ is harmonic in $\Omega:=B_r\cap \{u>0\}$, elliptic regularity suggests that we should be able to control
$\|\de_n v\|_{C^{0,\beta}(\Omega)}$ with $\|v\|_{C^{1,\beta}(\de\Omega)}+\|v\|_{L^\infty(\Omega)}$. Furthermore, by Lemma \ref{lem:importantlemma}, $\de \Omega$ is close to the hyperplane $\{x_n=0\}$, so we expect that the main contribution in $\|v\|_{C^{1,\beta}(\de\Omega)}$ comes from the tangential derivatives $\|\de_j v\|_{C^{0,\beta}(\de\Omega)}$ for $j\neq n$. If we could take $\beta=0$ and knew that $\de\Omega$ was regular enough, this argument would give a bound on $\sup_{B_r}|\de_n\curlyP_k|$ in terms of $\|\de_j v\|_{L^\infty(B_r)}$. But this is too much to ask: Schauder estimates break down at the Lipschitz scale and $\de\Omega$ could be wild, this is why we lose a power $\beta$ on the right hand side of \eqref{eq:normallip}. The first issue will be fixed choosing $\beta$ small and interpolating, the second will require to construct a different set $\Omega\subseteq \{u>0\}$ trough a geometric barrier argument.

We start with a preliminary $L^2-L^\infty$ estimate.
\begin{lemma}\label{lem:L2toLinfty}
In the setting of Proposition \ref{prop:lipschitz}, there exist a constant $C=C(n,k,\tau)$ such that
\begin{equation} 
    \|v_r\|_{L^\infty(B_1)}\leq C \|v_r\|_{L^2(B_2\setminus B_{1/2})}+C r^{k+2},
\end{equation}
for all $0<r<\frac 1 2$.
\end{lemma}
\begin{proof}
Recalling $\lap \curlyP_k =1$, we have $\lap v=\lap (u - \curlyP_k ) = -\chi_{\{u=0\}}\leq 0$. Using the mean-value inequality for super-harmonic functions, for some $z\in \de B_r$ we have:
$$
\min_{\overline{B_r}} v=v(z)\geq \fint_{B_{1/2}(z)} v_r\gtrsim_n -\|v_r\|_{L^2(B_2\setminus B_{1/2})},
$$
this provides the estimate from below. 
Inside $\{u=0\}\cap B_r$ we have
$$
v=-\frac{1}{2}\curlyA_k^2 +O(|x|^{k+2})\leq C r^{k+2},
$$
for some $C=C(n,k,\tau)$. Then we ``glue'' the functions
$C r^{k+2}$ and $v$, which is harmonic in $B_r\cap\{u>0\}$,
so that 
$$V:=\max\{C r^{k+2},v\}\ \ \ \text{ is subharmonic in }B_r.$$ Thus to estimate from above $v$ on $B_r\cap\{u>0\}$ we just use the mean value property on $V$ as above: this gives the upper bound up to corrections of size $C r^{k+2}$.
\end{proof}
With a similar technique, we bound the first derivatives of $v$ with $\sup_{\{u=0\}\cap B_r}|\de_n\curlyP_k|$.
\begin{lemma}\label{lem:bound on de_j v}
In the setting of Proposition \ref{prop:lipschitz}, there exists a constant $C=C(n,k,\tau)$, such that for each $j=1,\ldots,{n-1}$ and $0<r<1/2$ it holds
\begin{align*} 
 \|\de_j v_r\|_{L^\infty(B_{1})}&\leq Cr\cdot \sup_{\{u=0\}\cap B_r}|r\de_n\curlyP_k|+C \|v_r\|_{L^2(B_2\setminus B_{ 1/2})}+Cr^{k+2},\\
 \|\de_n v_r\|_{L^\infty(B_{1})}&\leq C  \sup_{\{u=0\}\cap B_r}|r\de_n\curlyP_k|+C \|v_r\|_{L^2(B_2\setminus B_{ 1/2})}+Cr^{k+2}.
\end{align*}
We remark that $\|\de_\ell v_r\|_{L^\infty(B_1)}=r\, \|\de_\ell v\|_{L^\infty(B_r)}$, for all $\ell$'s and $r>0$.
\end{lemma}
\begin{proof}
We address first the case $j\neq n$. By construction, we have $\lap \curlyP_k=1$, so $\lap v= -\chi_{\{u=0\}}.$ Hence, $\de_jv$ is harmonic in $B_r\cap \{u>0\}.$ On the other hand, in $B_r\cap \{u=0\}$ we have $\de_j v = - \de_j \curlyP_k$ so by Proposition \ref{prop:curly A vs curly P}
$$
|\de_j \curlyP_k|\leq |\de_j\curlyA_k||\curlyA_k|+ O(|x|^{k+1})\lesssim_{n,k,\tau} |x||\de_n\curlyP_k(x)|+|x|^{k+1},
$$
here we crucially used that $|\de_j \curlyA_k |\lesssim_{n,k,\tau}|x|$ as $\curlyA_k(x)=x_n+O(x^2)$.
Hence, 
$$
\sup_{B_r\cap\{u=0\}}|\de_jv|\leq Cr^{k+1} + Cr\cdot \sup_{B_r\cap \{u=0\}}|\de_n\curlyP_k| =:K
$$ 
for some $C=C(n,k,\tau)$. In order to estimate $\de_j v$ on $B_r\cap\{u>0\}$ we truncate it at levels $K$ and $-K$ to obtain that
$$
f:= \max\left\{ K,\de_j v\right\}\quad \text{ and }\quad g:=\min\left\{ -K,\de_jv\right\}
$$ are, respectively, subharmonic and superharmonic in $B_r$. Choose $x\in \de B_r$ a maximum point of for $f$ in $\overline{ B_r}$ and use the mean value property: 
$$
\sup_{B_r\cap\{u>0\}}\de_jv\leq \sup_{ B_r} f =f(x)\leq \fint_{B_{1/2}(x)}f_r\lesssim_n K+\frac{1}{r}\|\de_j v_r\|_{L^1(B_{3/2}\setminus B_{1/2})},
$$
where we used that $B_{r/2}(x)\subseteq B_{3r/2}\setminus B_{r/2}$. By standard elliptic estimates we find
$$
\|\nabla v_r\|_{L^1(B_{3/2}\setminus B_{1/2})}\lesssim_n \|\lap v_r\|_{L^1(B_{2}\setminus B_{1/2})}+\| v_r\|_{L^2(B_{2}\setminus B_{1/2})}.
$$
Recalling $\lap v_r=-r^2\chi_{\{u_r=0\}}\leq 0$, we integrate by parts with some smooth cut-off function $\chi_{B_{2}\setminus B_{1/2}}\leq\psi\leq \chi_{B_{3}\setminus B_{1/3}}$ with $\|\psi\|_{C^2}\lesssim_n 1$ we have
    \begin{align*}
      \int_{B_{2}\setminus B_{1/2}} |\lap v_r|=      -\int_{B_{2}\setminus B_{1/2}} \lap v_r \leq      -\int_{B_{3}\setminus B_{1/3}} \lap v_r\psi
      \lesssim_{n} \|v_r\|_{L^2(B_{3}\setminus B_{1/3})}.
  \end{align*}
The same computation with $g$ instead of $f$ provides an analogous estimate from below on $\de_j v$. In conclusion we proved that in every case, in $B_r$ it holds
$$
|\de_jv|\lesssim_n K+ \frac{1}{r}\|v_r\|_{L^2(B_3\setminus B_{1/3})},
$$ multiplying by $r$ on both sides we find the first estimate.

The second estimate is proven with the same reasoning, just replacing $K$ with 
$$K' := Cr^{k+1} + C\cdot \sup_{B_r\cap \{u=0\}}|\de_n\curlyP_k|,$$
without the ``extra $r$''.
\end{proof}

We now use global Schauder estimates to bound from above the term
$\sup_{\{u=0\}\cap B_r}|r\de_n\curlyP_k|.$
\begin{lemma}\label{lem:bound on de_n P}
In the setting of Proposition \ref{prop:lipschitz}, for any $\beta \in (0,\tfrac{\alpha_\circ^2}{k+2})$ there exists $C=C(n,k,\tau,\beta)$ such that for all $0<r<1/2$ it holds
\begin{equation}\label{eq:before abs}
\begin{split}
\sup_{\{u=0\}\cap B_{r/4}}|r\de_n \curlyP_k| &\leq  C r^{\frac{\beta}{\alpha_\circ-\beta}}\|\de_n v_r\|_{L^\infty( B_1)}+C r^{\beta+\frac{\beta}{\alpha_\circ}}\|\der_{x'} v_r\|_{L^\infty( B_1)}^{1-\frac{\beta}{\alpha_\circ}}\\
&\quad +C \|v_r\|_{L^\infty (B_1)}+C r^{k+2}.
\end{split}
\end{equation}
We recall that the dimensional constant $\alpha_\circ>0$ has been defined in Lemma \ref{lem:importantlemma}.
\end{lemma}
\begin{proof}
We will split the coordinates $x=(x',x_n)$ and denote with $B'_r$ the intersection $B_r$ and $\{x_n=0\}$. First of all, we recall from Lemma \ref{lem:importantlemma}
\begin{equation}\label{eq:basic bound}
\sup_{\{u=0\}\cap B_r}|\de_n \curlyP_k| = \sup_{\{u=0\}\cap B_r}|x_n+O(|x|^2)| \leq C_\circ r^{1+\alpha_\circ}\qquad \text{ for all }0<r<1,
\end{equation} 
for some $\alpha_\circ,C_\circ$ depending only on $n,k$. 
It is enough to prove the claim for $r\in(0,r_0)$, for some $r_0$ whose size will be constrained during the proof in terms of $n,k,\tau$ (recall that $\tau\geq |(p_2,\ldots,p_3)|$). We will prove in detail only the upper bound, as the lower bound is derived in the same way, with a ``symmetric'' argument.

We choose a point $z\in \overline {B_{r/4}}\cap\{u=0\}$ such that
$$
\sup_{\{u=0\}\cap B_{r/4}}r\de_n \curlyP_k=r\de_n\curlyP_k(z).
$$

\textbf{Step 1.} If $r_0$ is small enough, then for all $r\in(0,r_0)$ we can find an open set $\Omega$ such that the following holds:
\begin{itemize}
	\item[(i)]$\Omega$ is a smooth domain inside $B_{r_0}$, that is $\Omega\cap B_{r_0}=\{x_n>\gamma(x')\}\cap B_{r_0}$ for some $\gamma\in C^\infty(B_{r_0}')$. Furthermore we have the following estimates on $\gamma$:
	\begin{equation}
	\|\gamma\|_{L^\infty(B_r')}\leq C r^{1+\alpha_\circ},\ \|\nabla' \gamma\|_{L^\infty(B_{r}')}\leq C r^{\alpha_\circ},\  [\nabla' \gamma]_{C^{\alpha_\circ}(B_{r}')}\leq C,
	\end{equation}\label{eq:boundsongamma}
	for some $C=C(n,k,\tau)$. 
	\item[(ii)] $\Omega\subseteq \{u>0\}$ and there exists $z^* \in\de\Omega\cap\{u=0\}\cap B_{r/2}$.
\end{itemize}
We emphasize that $\Omega,\gamma$ and $z^*$ may depend on $r$, but the constant in \eqref{eq:boundsongamma} does not.

\textit{Proof of Step 1.} We fix some $r\in(0,r_0)$. For all $b \in \R$ and $L=L(n,k,\tau)$ to be determined, we define the domains
$$
\Omega(b):=\left\{x\in B_{r}\colon \de_n\curlyP_k(x)>\de_n\curlyP_k(z)+ L r^{\alpha_\circ-1}|x'-z'|^{2} +b \right\}.
$$
Roughly speaking $\Omega(b)$ looks like a perturbed paraboloid with vertex at $(z',z_n+b)$, provided $r_0\lesssim_{n,k,\tau} 1$. Now, starting with $b$ large, we decrease it until $\Omega(b)$ touches the contact set in $B_r$. That is, define 
$$
\Omega:=\Omega(b^*),\quad \text{ where }\quad b^*:=\inf\left\{b \in \R : \Omega(b)\cap B_r\subseteq \{u>0\}\right\}.
$$
We start checking that $b^*$ is well-defined. Thanks to \eqref{eq:basic bound}, if there exists $x\in \{u=0\}\cap B_r\cap\Omega(b)$ then
\begin{align}\label{eq:bandx'}
L r^{\alpha_\circ-1}|x'-z'|^2 + b < |\de_n\curlyP_k(x)|+|\de_n\curlyP_k(z)|\leq 2C_\circ r^{1+\alpha_\circ}.
\end{align}
This shows that $b^*$ is well defined and $b^*\leq 2C_\circ r^{1+\alpha_\circ}$, in fact $\{u=0\}\cap B_r\cap \Omega(b)$ must be empty for larger $b$'s. We also notice that $b\geq 0$, as $z\in \{u=0\}\cap B_r\cap\Omega(b)$ for all $b<0$.

We now prove (ii). Take $b<b^*$, by definition there exists $x_b\in \{u=0\}\cap \Omega(b)$, inequality \eqref{eq:bandx'} shows that $|x_b'-z'|\leq r/8$, for an appropriate choice of $L \gtrsim C_\circ$. Using the triangular inequality and \eqref{eq:basic bound} to estimate $|(x_b)_n-z_n|$, we find, for $r_0$ is small, that
\begin{equation}\label{eq:touchingcloseness}
x_b\in \{u=0\}\cap B_r\cap \Omega(b) \implies x_b \in B_{r/2} \ \text{ and }\ b\leq  2C_\circ r^{1+\alpha_\circ}.
\end{equation} 
We take $z^*\in \overline{B_{r/2}}$ to be any accumulation point of $x_b$ as we let $b\uparrow b^*$. We clearly have $z^*\in \de\Omega\cap\{u=0\}\cap \overline{B_{r/2}}$.

Let us now prove (i). Consider the map $\mathit{\Phi}\colon x\mapsto (x',\de_n\curlyP_k(x))$, as $\de_n\curlyP_k(x)=x_n+O(x^2)$, we have that it is a diffeomorphism from $B_{r_0}$, provided $r_0$ is small. Let $\mathit{\Psi}$ denote the $n$-th component of its inverse which is defined in some ball $B_{R_0}$. Clearly we have that $\mathit{\Psi}(y)=y_n+O(|y|^2)$, thus it holds
\begin{equation}\label{eq:boundsonpsi}
|\nabla'\mathit{\Psi}(y)|\lesssim_{n,k,\tau} |y|\, \text{ and }\, \de_n\mathit{\Psi}(y)\geq 1/2,\quad\text{ for all }y\in B_{R_0}. 
\end{equation}
Therefore we conclude that $x\in \Omega$ if, and only if, $x\in B_{r}$ and
$$
x_n=\mathit\Psi(x', \de_n\curlyP_k(x))>\mathit\Psi\left(x',\de_n\curlyP_k(z)+ L r^{\alpha_\circ-1}|x'-z'|^2 +b^*\right):=\gamma(x').
$$
we have that $\gamma\in C^\infty(B'_{r_0})$ is well-defined as for $r_0$ small
$$
(x',\de_n\curlyP_k(z)+ L r^{\alpha_\circ-1}|x'-z'|^2 +b^*)\in B_{R_0}.
$$
This proves that $\Omega$ is a smooth domain. We remark that, while $\gamma$ depends on $r$, $\mathit{\Psi}$ only depends on $r_0$, this observation along with \eqref{eq:boundsonpsi} easily gives the estimates on $\gamma$ with constants independent of $r$. For example, we estimate for all $x'\in B'_{r}$:
\begin{align*}
|\der' \gamma(x')|& \leq |\der' \mathit\Psi \left(x',\de_n\curlyP_k(z)+ L r^{\alpha_\circ-1}|x'-z'|^2 +b^*\right)|\\
& +2 L r^{\alpha_\circ -1} |x'-z'|\, |\de_n \mathit\Psi \left(x',\de_n\curlyP_k(z)+ L r^{\alpha_\circ-1}|x'-z'|^2 +b^*\right)| \\
& \lesssim_{n,\tau}|x'|+\bigg|\de_n\curlyP_k(z)+L r^{\alpha_\circ-1}|x'-z'|^2 +b^*\bigg|\\
&+L r^{\alpha_\circ-1}\|\de_n\mathit\Psi\|_{L^\infty(B_{R_0})}|x'-z'|
\end{align*}
and all the terms on the right hand side are of order $r^{\alpha_\circ}$ or higher thanks to \eqref{eq:bandx'}. The other estimates can be proven identically, using \eqref{eq:bandx'} and the fact that $\mathit{\Psi}(0,0)=|\der'\mathit{\Psi}(0,0)|=0$.

\textbf{Step 2.} For each $\beta \in (0,\alpha_\circ]$ there is $C=C(n,k,\tau,\beta)$ such that
\begin{equation*}
    \frac 1 C \sup_{\{u=0\}\cap B_{r/4}}r\de_n \curlyP_k \leq \left[\der '(v\circ\Gamma)_r\right]_{C^\beta(B_{3/4}')} + \|v_r\|_{L^\infty(B_1)},
\end{equation*}
where $\Gamma\colon B_r'\to \R^n$ is the graph of $\gamma$, that is $\Gamma(x')=(x',\gamma(x'))$.

\textit{Proof of Step 2.} We observe that, by definition of $z^*$, $\de_n\curlyP_k(z)\leq \de_n\curlyP_k(z^*)$, so we have 
$$
\sup_{\{u=0\}\cap B_{r/4}}r\de_n \curlyP_k=r\de_n\curlyP_k(z)\leq r\de_n\curlyP_k(z^*)\leq \|\de_nv_r\|_{L^\infty(\de\widetilde\Omega\cap B_{1/2})},
$$
where we used the rescaled domain $\widetilde\Omega:=\Omega/r$, whose boundary is the graph of $\widetilde\gamma=\gamma(r\,\cdot\,)/r$. We now employ global Schauder estimates in $\widetilde{\Omega}$ (see e.g., \cite[Theorem 8.33]{GT01}) to control the right hand side:
\begin{align}\label{eq in the middle}
\|\de_n v_r\|_{L^\infty(\de\widetilde\Omega\cap B_{1/2})}\lesssim_{n,\beta}\left[\der '(v\circ\Gamma)_r\right]_{C^\beta(B_{3/4}')}+\|\lap v_r\|_{L^\infty(\widetilde\Omega)}+\|v_r\|_{L^\infty ( B_1)}.
\end{align}
Note that $\widetilde\Omega$ lies in $\{u_r>0\}$ where $\lap v_r \equiv 0$. We remark that the previous estimate holds with a constant which depends on $\|\nabla\widetilde\gamma\|_{C^{\alpha_\circ}(B_1)}$ which is bounded independently from $r$, by \eqref{eq:boundsongamma}:
$$
\|\der\widetilde\gamma\|_{L^\infty(B_1')}=\|\der\gamma\|_{L^\infty(B_r')}\leq C r^{\alpha_\circ}\quad \text{ and }\quad [\der \widetilde\gamma]_{C^{\alpha_\circ}(B_1')}=r^{\alpha_\circ}\|\der\gamma\|_{C^{\alpha_\circ}(B_r')}\leq C r^{\alpha_\circ}.
$$
Furthermore, we also used that, thanks to \eqref{eq:boundsongamma} the graph of $\widetilde\gamma$ is contained in the strip $\{|y_n|\leq 1/10\}\cap B_1$, provided that $r_0$ is small. 

\textbf{Step 3.} There is $C=C(n,k,\tau)$ such that
\begin{align*}
    \frac 1 C \left\|\der '(v\circ\Gamma)_r\right\|_{L^\infty(B_{3/4}')} &\leq \|\der' v_r\|_{L^\infty(B_1)} + r^{\alpha_\circ} \|\de_n v_r\|_{L^\infty(B_1)},\\
    \frac 1 C \left[\der '(v\circ\Gamma)_r\right]_{C^{\alpha_\circ}(B_{3/4}')} &\leq  r^{1+\alpha_\circ}.
\end{align*}
\textit{Proof of Step 3.} Let $\widetilde\Gamma$ denote the graph of $\widetilde\gamma$, that is $\widetilde\Gamma(y')=(y',\widetilde\gamma(y')))$. We have
$$
\der'(v\circ\Gamma)_r=\der'(v_r\circ\widetilde\Gamma)=\der'v_r + \left(\de_n v_r \circ \widetilde\Gamma\right) \der' \widetilde\gamma,
$$
so the first bound follows taking the $L^\infty(B_{3/4}')$ norms, using that $\widetilde\Gamma(B_{3/4}')\subseteq B_1'$ and that $\|\der\widetilde\gamma\|_{L^\infty(B_1')}\leq C r^{\alpha_\circ}$. For the second bound we use the optimal regularity of $u$ (see \eqref{eq:optimalreg+nondeg})
\begin{align*}
    \left[\der '(v\circ\Gamma)_r\right]_{C^{\alpha_\circ}(B_{3/4}')} &=r^{1+\alpha_\circ}[(\der' v \circ \Gamma) (\der' \Gamma)]_{C^{\alpha_\circ}(B_{3r/4}')}\\
    &\leq r^{1+\alpha_\circ}\left(\|\der v\|_{L^\infty(B_{r})}[\der'\gamma]_{C^{\alpha_\circ}(B_r')}+[\der v]_{C^{0,1}(B_r')}[\Gamma]_{C^{\alpha_\circ}(B_r')}\|\der'\Gamma\|_{L^\infty(B_r')}\right)\\
    &\leq r^{1+\alpha_\circ}\|\gamma\|_{C^{\alpha_\circ}(B_r')}\|v\|_{C^{1,1}(B_{1/2})}\\
    &\lesssim_{n,k,\tau} r^{1+\alpha_\circ}.
\end{align*}
\textbf{Step 4.} There is $C=C(n,k,\tau,\beta)$ such that
\begin{equation}\label{eq:step4}
    \frac 1 C \left[\der '(v\circ\Gamma)_r\right]_{C^\beta(B_{3/4}')}\leq r^{\frac{\alpha_{\circ}^2}{\beta}} + r^{\frac{\beta}{\alpha_\circ-\beta}}\|\de_n v_r\|_{L^\infty( B_1)} + r^{\beta+\frac{\beta}{\alpha_\circ}}\|\der' v_r\|_{L^\infty( B_1)}^{1-\frac{\beta}{\alpha_\circ}}
\end{equation}
\textit{Proof of Step 4.} The idea is to bound the $C^\beta$ norm in \eqref{eq in the middle} with an interpolation of the $L^\infty$ and the $C^{\alpha_\circ}$ norms, which we bounded in Step 3. This gives
\begin{align*}\label{eq:after the middle}
\left[\der '(v\circ\Gamma)_r\right]_{C^\beta(B_{3/4}')}&\leq\left[\der '(v\circ\Gamma)_r\right]_{C^{\alpha_\circ}(B_1')}^{\frac{\beta}{\alpha_\circ}}\left[\der '(v\circ\Gamma)_r\right]_{L^\infty(B_1')}^{1-\frac{\beta}{\alpha_\circ}}\\
&\leq Cr^{\beta+\beta/\alpha_\circ}\left(\|\der' v_r\|_{L^\infty(B_1)} + r^{\alpha_\circ} \|\de_n v_r\|_{L^\infty(B_1)}\right)^{1-\frac{\beta}{\alpha_\circ}}\\
&\leq Cr^{\alpha_\circ}\left(r^{\frac{\beta}{\alpha_\circ-\beta}}\|\de_n v_r\|_{L^\infty( B_1)}\right)^{1-\frac{\beta}{\alpha_\circ}}+Cr^{\beta+\frac{\beta}{\alpha_\circ}}\|\der' v_r\|_{L^\infty( B_1)}^{1-\frac{\beta}{\alpha_\circ}},
\end{align*}
where in the last passage we used the subadditivity of $t\mapsto t^{^{1-{\beta}/{\alpha_\circ}}}$.
Finally we obtain \eqref{eq:step4} using Young inequality on the first term, with exponent $\frac 1 p= \frac{\beta}{\alpha_\circ}$:
$$
r^{\alpha_\circ}\left(r^{\frac{\beta}{\alpha_\circ-\beta}}\|\de_n v_r\|_{L^\infty( B_1)}\right)^{1-\frac{\beta}{\alpha_\circ}}\lesssim_{\alpha_\circ,\beta} r^{\frac{\alpha_{\circ}^2}{\beta}} + r^{\frac{\beta}{\alpha_\circ-\beta}}\|\de_n v_r\|_{L^\infty( B_1)}.
$$

Combining Steps 2 and 4 and recalling that $\frac{\alpha_{\circ}^2}{\beta}>k+2$ we obtain \eqref{eq:before abs} for all $r\in(0,r_0)$, as the constants do not depend on $r$.
\end{proof}
 With help of the two previous lemmas we can now prove our main Lipschitz estimate by using $\sup_{B_r\cap\{u=0\}}|\de_n\curlyP_{k}|$ as ``pivot''.
 
 \begin{proof}[Proof of Proposition \ref{prop:lipschitz}]
 We first address (ii), the estimate in the normal direction. All constants will depend on $n,k,\tau,\beta$. Linking Lemma \ref{lem:bound on de_n P} with the second estimate of Lemma \ref{lem:bound on de_j v} we get
 \begin{align*}
 	\frac 1 C \|\de_nv_r\|_{L^\infty(B_{1/4})}\leq  r^{\frac{\beta}{\alpha_\circ-\beta}}\|\de_n v_r\|_{L^\infty( B_1)}+ r^{\beta+\frac{\beta}{\alpha_\circ}}\|\der_{x'} v_r\|_{L^\infty( B_1)}^{1-\frac{\beta}{\alpha_\circ}}+\Lambda(r)
 \end{align*}
 where we set for brevity
 $$
 \Lambda(r):=\|v\|_{L^\infty(B_r)}+r^{k+2},
 $$
 notice that $\Lambda$ is increasing in $r$. Now we use trivial bounds and the tangential estimate of Lemma \ref{lem:bound on de_j v} to deal with the central terms at the right hand side
 \begin{align*}
 	r^{\beta+\frac{\beta}{\alpha_\circ}}\|\der_{x'} v_r\|_{L^\infty( B_1)}^{1-\frac{\beta}{\alpha_\circ}}&\leq C r^{\beta+\frac{\beta}{\alpha_\circ}}\left(r\|\de_nv_r\|_{L^\infty(B_1)}+\Lambda(2r)\right)^{1-\frac{\beta}{\alpha_\circ}}.
 \end{align*}
 Then, using Young inequality with $1/p=\beta/\alpha_\circ$ and $r^{\frac{\alpha_\circ}{\beta}}\leq r^{k+2}$, we obtain
 \begin{align*}
 r^{\beta+\frac{\beta}{\alpha_\circ}}\left(r\|\de_nv_r\|_{L^\infty(B_1)}\right)^{1-\frac{\beta}{\alpha_\circ}}&=r\left(r^{\frac{\alpha_\circ\beta}{\alpha_\circ-\beta}}\|\de_n v_r\|_{L^\infty(B_1)}\right)^{1-\frac{\beta}{\alpha_\circ}}\\
 &\lesssim_{\alpha_\circ,\beta} r^{\frac{\alpha_\circ\beta}{\alpha_\circ-\beta}}\|\de_nv_r\|_{L^\infty(B_1)}+r^{k+2}.
 \end{align*}
 Thus, by subadditivity of $t\mapsto t^{1-\beta/\alpha_\circ}$ and enlarging the constants, we finally arrive to
 \begin{equation}\label{eq:preiteration}
 \|\de_nv_r\|_{L^\infty(B_{1/4})}\leq C r^{\frac{\alpha_\circ\beta}{\alpha_\circ-\beta}}\|\de_nv_r\|_{L^\infty(B_1)} +C\Lambda(2r)^{1-\frac{\beta}{\alpha_\circ}}.
 \end{equation}
 We conclude by iteration of this inequality. Let us set $f(r):=\|\de_nv_r\|_{L^\infty(B_{1})}=r\|\de_n v\|_{L^\infty(B_r)}$ and $\delta:={\frac{\alpha_\circ\beta}{\alpha_\circ-\beta}}$, then \eqref{eq:preiteration} reads as
 	$$
 	f(r/4)\leq C r^\delta f(r) + C\Lambda(2r)^{1-\frac{\beta}{\alpha_\circ}},\quad \text{ for all }0<r<1/4.
 	$$
 Since $f$ and $\Lambda$ are increasing functions we can iterate this inequality $N\sim k/\delta$ times ($N$ does not depend on $r$) it becomes
 $$
 f(r)\leq C_N\Lambda(4^N r)^{1-\frac{\beta}{\alpha_\circ}}+C_Nf(1/4)r^{k+2}, \quad \text{ for all }0<r<4^{-N},
 $$
and $f(1/4)$ is again bounded by a dimensional constant, by optimal regularity. Finally, we use Lemma \ref{lem:L2toLinfty} to replace $\|v_{4^N r}\|_{L^\infty(B_1)}$ with $\|v_{4^N r}\|_{L^2(B_2\setminus B_{1/2})}+r^{k+2}$. We have proved that
 \begin{equation*}
 	\|\de_nv_r\|_{L^\infty(B_1)}\leq C\left(\|v_{\theta r}\|_{L^2(B_2\setminus B_{1/2})}+r^{k+2}\right)^{1-\frac{\beta}{\alpha_\circ}}
 \end{equation*}
 with $C=C(n,k,\tau,\beta)$ and $\theta(n,k,\beta)=4^N$, for all $r\in(0,r_0)$. Since $\beta \in (0,\alpha_\circ /(k+2))$ we proved \eqref{eq:normallip}.
 
We turn to the proof of the tangential estimate \eqref{eq:tangentiallip}. Combining the previous step with Lemma \ref{lem:bound on de_j v} we have for $C>0$ and $r\in(0,r_0)$ that:
\begin{align*}
\frac 1 C \|\de_j v_r\|_{L^\infty(B_{1})}&\leq r \left(\|v_{\theta r}\|_{L^2(B_2\setminus B_{1/2})}+r^{k+2}\right)^{1-\frac{\beta}{\alpha_\circ}}+\|v_{r}\|_{L^2(B_2\setminus B_{1/2})}+r^{k+2},\\
&\leq \big(r^{1-\frac{\beta(k+2)}{\alpha_\circ}}+1\big)\big( \|v_{\theta r}\|_{L^2(B_2\setminus B_{1/2})}+r^{k+2}\big).
\end{align*}
We used that $\theta $ is large and \ref{lem:L2toLinfty} to bound
$$
\|v_{r}\|_{L^2(B_2\setminus B_{1/2})}\leq \|v_{r}\|_{L^\infty(B_2)} \leq \|v_{\theta r}\|_{L^2(B_1}\lesssim_{n,k,\theta} \|v_{\theta r}\|_{L^2(B_2\setminus B_{1/2})}.
$$
Note that with the choice $\beta:=\frac{\alpha_\circ^2}{k+4}$, we get rid of the $\beta$ dependence in the constants and obtain the claimed estimate.
\end{proof}

\section{Monotonicity of the truncated frequency}
\label{sec:monotonicity frequency}
In this section we show that the truncated frequency function $\phi^\gamma(r,u-\curlyP_k)$ from \eqref{eq:trun freq}, is almost monotone for $\gamma <k+2$, whichever $p_3,
\ldots,p_k$ are. All the proofs in this section do not change for a generic smooth $f$, as we use the inequalities of the previous section as ``black boxes''.

The core of the monotonicity is the following computational lemma.
 \begin{lemma}
 \label{lem:estimate on the contact set}
 Let $k\geq 2, \tau>0$ and $u\colon B_1\to \R$ be a solution of the obstacle problem \eqref{eq:obstacle} with $f\equiv 1,\mu=1$. Assume $r^{-2}u(r\cdot\, )\to p_2$ and take $(p_2,\ldots,p_k)\in \PP_k$ such that $|(p_2,\ldots,p_k)|\leq \tau$. Consider $v:=u-\curlyP_k$, where $\curlyP_k=\curlyP_k(p_2,\ldots,p_k)$ is constructed as in Definition \ref{def:curlyP}. For each $\gamma \in [0,k+2)$ and $\beta \in (0,\frac{\alpha_\circ}{k+2})$ set 
 \begin{equation*}
 \epsilon:=\min\left\{\alpha_\circ-\beta(k+2),k+2-\gamma\right\}>0,
 \end{equation*}
 where the dimensional constant $\alpha_\circ$ is the one of Lemma \ref{lem:importantlemma}. Then there exists $r_0=r_0(k,n,\tau,\beta)\in(0,1)$, such that for all $0<r<r_0$
 \begin{equation}
     \label{eq:almost mono}
    \frac{d}{dr} \phi^\gamma(r,v) \ge -Cr^{\epsilon-1}(g^\gamma(r,v)+1)(\phi^\gamma(r)+1), \end{equation} 
    and 
     \begin{equation}
     \label{eq:help almost mono}
     \frac{ \int_{B_r}|v_r \lap v_r |}{ H(r,v) + r^{2\gamma} }\leq C r^\epsilon (g^\gamma(r,v)+1), \end{equation} 
 where we set
$$
g^\gamma (r,v) :=  \frac{ \|v_{\theta r}\|_{L^2(B_2\setminus B_{1/2})}^2}{H(r,v) + r^{2\gamma}}.
$$
$\theta=\theta(k,\beta),C=C(n,k,\tau,\beta)$ are constants.
\end{lemma}
\begin{proof}
Throughout the proof $C$ will be a constant depending on $n,k,\tau,\beta$. Up to a rotation of the coordinate axes we may assume $p_2=\frac 1 2 x_n^2$. We begin with recalling the estimate on the derivative of $\phi^\gamma$, cf.\ \eqref{eq:der trun freq},
\begin{equation*}
    \frac{d}{dr} \phi^\gamma(r,v) \ge \frac{2}{r}  \frac{ \phi^\gamma(r,v)\int_{B_1} v_r\lap v_r -\int_{B_1} (x\cdot \nabla v_r)\lap v_r }{ H(r,v) + r^{2\gamma} }.
\end{equation*}
Using $\supp \lap v_r\subseteq \{u_r=0\}$, we reduce \eqref{eq:almost mono} to a bound from below on the quantity 
 $$\frac{2}{r}  \frac{  \|\phi^\gamma(r)v_r - x\cdot \nabla v_r\|_{L^\infty(B_1\cap \{u_r=0\})} \int_{B_1\cap \{u_r=0\}} |\lap v_r| }{ H(r,v) + r^{2\gamma} }.$$
 For $x\in B_1\cap\{u_r=0\}$ and $r<r_0(n,k,\tau,\beta)$ we have 
 \begin{enumerate}[label={\upshape(\roman*)}]
     \item $ |x_n|\leq C r^{\alpha_\circ}$, see Lemma \ref{lem:importantlemma};
      \item   $ |v_r(x)|\leq C r |\de_n( \curlyP_k)_r(x)|+Cr^{k+2}$, see (ii) in Proposition \ref{prop:curly A vs curly P};
      \item $ |\de_j v_r(x)|\leq C r |\de_n( \curlyP_k)_r(x)|+Cr^{k+2},$ for all $j\neq n$, see Lemma \ref{lem:bound on de_j v};
      \item $r|\de_n \curlyP_k(x) |\leq  C( \|v_{\theta r}\|_{L^2(B_2\setminus B_{1/2})}+r^{k+2})^{1-\beta},$ see Proposition \ref{prop:lipschitz}.
 \end{enumerate}
 Putting together all these bounds we get for all $x\in {B_1\cap \{u_r=0\}}$
 \begin{align*}
     |\phi^\gamma(r)v_r - x\cdot \nabla v_r| &\leq(\phi^\gamma(r)+1)(|v_r|+|x_n||\de_n(\curlyP_k)_r|+|\nabla_{x'}v_r|)\\
     &\leq C (\phi^\gamma(r)+1)\big(r^{\alpha_\circ}\|\de_n (\curlyP_k)_r \|_{L^\infty(B_1\cap\{u_r=0\})}+r^{k+2}\big)\\
     &\leq C (\phi^\gamma(r)+1)r^{\alpha_\circ} \big( \|v_{\theta r}\|_{L^2(B_{3}\setminus B_{1/2})}+r^{k+2}\big)^{1-\beta} + C(\phi^\gamma(r)+1)r^{k+2}\\
    & \leq C(\phi^\gamma(r)+1)r^{\epsilon} \big( \|v_{\theta r}\|_{L^2(B_{3}\setminus B_{1/2})}+r^{k+2}\big) +C(\phi^\gamma(r)+1)r^{k+2},
 \end{align*}
 where in the last passage we argued  
 $$( \|v_{\theta r}\|_{L^2(B_{3}\setminus B_{1/2})}+r^{k+2})^{1-\beta}\leq r^{-\beta(k+2)}( \|v_{\theta r}\|_{L^2(B_{3}\setminus B_{1/2})}+r^{k+2}).$$
 Thus, using that ${\big(H(r,v)+r^{2\gamma}\big)^{-1/2}}\leq r^{-\gamma}$ and that $r^{k+2-\gamma} \leq r^{\epsilon}$, we get
 \begin{align*}
   \frac{\|\phi^\gamma(r)v_r - x\cdot \nabla v_r\|_{L^\infty(B_1\cap \{u_r=0\})}}{\big(H(r,v)+r^{2\gamma}\big)^{1/2}}&\leq r^\epsilon (\phi^\gamma(r)+1)( g^\gamma(r)^{1/2}+1).
 \end{align*}
 Now, using (ii) and (iv) to estimate $|\der v_r(x)|$ for $x\in B_1\cap\{u_r=0\}$ as above, we get 
 \begin{align*}
   \frac{\|v_r\|_{L^\infty(B_1\cap \{u_r=0\})}}{\big(H(r,v)+r^{2\gamma}\big)^{1/2}}&\leq r^\epsilon ( g^\gamma(r)^{1/2}+1).
 \end{align*}
Thanks to the this observation, to prove both \eqref{eq:almost mono} and \eqref{eq:help almost mono}, we only need to show 
  \begin{equation*}\label{eq:l2}  \frac{ \int_{B_1\cap \{u_r=0\}} |\lap v_r| }{ (H(r,v) + r^{2\gamma})^{1/2} }\leq C g^\gamma(r)^{1/2}.\end{equation*}
  As $\lap v_r = -r^2\chi_{\{u_r=0\}},$ we can integrate by parts with some regular cut-off  $\chi_{B_1}\leq\psi\leq \chi_{B_2},$
  \begin{align*}
     \int_{B_1\cap \{u_r=0\}} |\lap v_r|=      -\int_{B_1} \lap v_r \leq      -\int_{B_2} \lap v_r\psi
      \lesssim_n \|v_r\|_{L^1(B_2\setminus B_1)},
  \end{align*} which concludes the proof.
  \end{proof}
We now make a specific choice of $\beta$ and derive from this preliminary bounds the monotonicity of the truncated frequency.
   \begin{proposition}
 \label{pro:almost monotonicity}
 Let $k\geq 2, \tau>0$ and $u$ be a solution of the obstacle problem \eqref{eq:obstacle} with $f\equiv 1$ and $\mu=1$. Assume $r^{-2}u(r\cdot\, )\to p_2$ and take $(p_2,\ldots,p_k)\in \PP_k$ such that $|(p_2,\ldots,p_k)|\leq \tau$. Consider $v:=u-\curlyP_k$, where $\curlyP_k=\curlyP_k(p_2,\ldots,p_k)$ is constructed as in Definition \ref{def:curlyP}, and for each $\gamma \in (0,k+2)$ set 
 $$\eps(\gamma):=\min\left\{\alpha_\circ/2;k+2-\gamma\right\},$$
 where the dimensional constant $\alpha_\circ$ is the one of Lemma \ref{lem:importantlemma}. Then, there exist $C(n,k,\tau,\gamma)>0$ and $r_0(n,k,\tau)\in (0,1)$, such that for all $0<r<r_0$ we have
 \begin{equation}\label{eq:almostmonotonicity}
 \frac{d}{dr} \phi^\gamma(r) \ge  -C r^{\eps-1},\quad \phi^\gamma(r)\leq C \quad \text{ and }\quad \frac{ \int_{B_r}|v_r \lap v_r |}{ H(r,v) + r^{2\gamma} }\leq C r^{\eps}.
 \end{equation}
 In particular, $\phi^\gamma(0^+,v)=\lim_{r\downarrow 0}\phi^\gamma(r,v)$ exists and $\phi^\gamma(0^+,v)\leq \gamma$.
   \end{proposition}
\begin{proof}
  	Fix $\gamma_\circ \in (0,k+2)$ and let $\eps_\circ=\eps(\gamma_\circ)$. Let $r_0=r_0(n,k,\tau,\beta=\frac{\alpha_\circ}{2(k+2)})$ be as in Lemma \ref{lem:estimate on the contact set}. By Lemma \ref{lem:estimate on the contact set} and estimate \eqref{eq:almost mono} we only need to show that the functions $g^{\gamma_\circ}(\cdot)$ and $\phi^{\gamma_\circ}(\cdot)$ are uniformly bounded in the interval $(0,r_0]$. We are going to prove it increasing iteratively the parameter $\gamma$, exactly as in the proof of \cite[Lemma 4.3]{FRS19}. Throughout the proof $C_\gamma$ will denote a general constant depending on $n,k,\tau,\gamma$ and similarly $C_{\gamma,\delta}$.
  
  First notice that by \eqref{eq:optimalreg+nondeg} the functions $\phi^0(\cdot,v)$ and $g^0(\cdot)$ are uniformly bounded in $[0,r_0]$. Fix any $\gamma\in [0,\gamma_\circ]$. The core of the proof is the following observation
  \begin{equation}\label{eq:iterationphigamma}
  	\begin{cases}
  	\phi^\gamma+g^\gamma \leq C_\gamma \,\text{ in } (0,r_0], \\
  	0<5\delta\leq\eps_\circ,
  	\end{cases} \Longrightarrow \phi^{\gamma+\delta}+g^{\gamma+\delta}\leq C_{\gamma,\delta} \,\text{ in }(0,r_0].
  \end{equation}
  We iterate this observation to reach the conclusion. Define the sequence $\gamma_0=0, \gamma_{j+1}:=\gamma_j+\eps_\circ/5$, where $j\ge 0$. With a finite number of iteration we get closer than $\eps_\circ/5$ to $\gamma_\circ$, applying \eqref{eq:iterationphigamma} once more with an appropriate $\delta$ we get to $\gamma_\circ$. 
  
We prove \eqref{eq:iterationphigamma} for a generic $\gamma$. Keeping in mind $r,\delta\leq 1$, we estimate
$$
\phi^{\gamma+\delta}(r)  = \frac{D(r,v) + (\gamma+\delta)r^{2\gamma+2\delta}}{H(r,v) + r^{2\gamma+2\delta}} \le  \frac{1}{r^{2\delta}}  \frac{D(r,v) + \gamma r^{2\gamma}}{H(r,v) + r^{2\gamma}} +1  \leq \frac{C_\gamma}{r^{2\delta}},
$$
and 
$$
g^{\gamma+\delta}(r)  = \frac{ \|v_{\theta r}\|^2_{L^2(B_3\setminus B_{1/2})}}{ H(r,v) + r^{2\gamma+2\delta}} \le \frac{1}{r^{2\delta}}  \frac{ \|v_{\theta r}\|^2_{L^2(B_2\setminus B_{1/2})}}{H(r,v) + r^{2\gamma}}  \le \frac{C_\gamma}{r^{2\delta}}.
$$
 Now we apply Lemma \ref{lem:estimate on the contact set} with $\beta:=\frac{\alpha_\circ}{2(k+2)}$ and $\gamma\to\gamma+\delta$
  \begin{equation*}
 \frac{d}{dr} \phi^{\gamma+\delta}(r) \ge -Cr^{\epsilon-1}(g^{\gamma+\delta}(r)+1)(\phi^{\gamma+\delta}(r)+1) \ge -C_\gamma r^{\eps_\circ-4\delta -1}\ge - C_\gamma r^{\delta-1}, 
 \end{equation*}
 where we used $\epsilon(\beta,\gamma+\delta)\geq \eps_\circ$ and the smallness of $\delta$. Integrating this inequality and using $C^{1,1}$ estimates we obtain
 $$
 \phi^{\gamma+\delta}(r)\leq \phi^{\gamma+\delta}(r_0) +C_\gamma r_0^\delta/\delta\lesssim_{n,k,\tau} C_{\gamma,\delta}
 $$
for all $r\in (0,r_0]$. The uniform boundedness of $g^{\gamma+\delta}$ is now a consequence of Lemma \ref{lem:H ratio} (a), as we can take $\overline\lambda=C_{\gamma,\delta}$. Indeed, for any $0<2\theta r<r_0$ and any $R\in (r\theta/2,2\theta r)$, we have
 $$
 \frac{H(R,v) +R^{2\gamma+2\delta}}{H(r,v) +r^{2\gamma+2\delta}} \lesssim_{n,k,\tau} \left(\frac{R}{r}\right)^{C_{\gamma,\delta}}\leq  C_{\gamma,\delta}.
 $$ Thus for all $r\in (0,r_0/2\theta]$ we have
 $$
 g^{\gamma+\delta} (r)  = \frac{\|v_{\theta r}\|^2_{L^2(B_2\setminus B_{1/2})}}{H(r,v) +r^{2\gamma+2\delta}} \lesssim_{n,k,\tau} \fint_{\frac{\theta}{2}r}^{2\theta r} \frac{H(R,v)\,dR}{H(r,v) +r^{2\gamma+2\delta}} \leq C_{\gamma,\delta}.
 $$
 Finally $g^{\gamma+\delta}$ is clearly bounded in $[r_0/2\theta,r_0]$ by \eqref{eq:optimalreg+nondeg}. This concludes the proof.
\end{proof}
The following is an immediate corollary of Proposition \ref{pro:almost monotonicity}.
\begin{corollary}\label{cor:logmonneau}
With the same notations of Proposition \ref{pro:almost monotonicity}, we have that for all $r\in (0,r_0)$
\begin{equation}
    \frac{d}{dr} \log\left(r^{-2\lambda}\left(H(r,u-\curlyP_k)+r^{2\gamma}\right)\right)\geq - C r^{\eps-1}.\qedhere
\end{equation}
provided $\lambda\leq \phi^\gamma(0^+,v)$. Here $r_0$,$C$,$\eps$ are the same positive numbers of Proposition \ref{pro:almost monotonicity}.
\end{corollary}
\begin{proof}
Computing the derivative and using Proposition \ref{pro:almost monotonicity} we obtain
  \begin{align*}
      \frac{d}{dr} \log\left(r^{-2\lambda}\left(H(r,v)+r^{2\gamma}\right)\right)&=\frac 2 r \left(\phi^\gamma(r,v)-\lambda\right)+\frac 2 r \frac{\int_{B_1}v_r\lap v_r}{H(r,v)+r^{2\gamma}}\\
      &\geq - C \fint_0^r s^{\eps-1}\, ds- C r^{\eps-1}\geq -C r^{\eps-1}.\qedhere
  \end{align*}
\end{proof}
\section{The sets \texorpdfstring{$\bm{\Sigma^{kth}}$}{} and higher order blow-ups}
\label{sec:blowup}
The estimates obtained in the previous two sections did not assume any relationship between $u$ and $\curlyP_k$, beside the crucial fact that $p_2$ was the blow-up of $u$ at $0$. In this section, instead, we define the set of points at which $u$ admits a polynomial expansion of order $k$. This expansion will identify the polynomial $\curlyP_k$ up to order $k$, leaving some freedom for the $k+1$ terms (cf. with (ii) in \ref{prop:curly A vs curly P}). 

The results of this section never use explicitly the simplifying assumption $f\equiv 1$; they use it implicitly, though, employing the results of the previous two sections. Hence for a generic smooth $f$, all the results of this section apply without change.

Recall that $\PP_k$ was defined at the beginning of Section \ref{subsec:polynomial ansatz}.

\begin{definition}
\label{def:Sigmakth}
Let $u\colon B_1\to[0,\infty)$ solve \eqref{eq:obstacle} and $x_\circ \in \Sigma_{n-1}$. We say
\begin{enumerate}[label={\upshape(\roman*)}]
	\item $\Sigma^{2nd}:=\Sigma_{n-1}$.
    \item $x_\circ\in\Sigma^{3rd}$ if
$$r^{-3} \big( u(x_\circ+r\,\cdot\,)- p_{2,x_\circ}(r\,\cdot\,)\big) \to p_{3,x_\circ}  \quad \text{in}  \quad W^{1,2}_{\loc}(\R^n)\cap C^0_{\loc}(\R^n)$$ as $r\downarrow 0$, where $p_{3,x_\circ}$ is some $3$-homogeneous harmonic polynomial vanishing on $\{p_{2,x_\circ}=0\}$. Thus we have $(p_{2,x_\circ},p_{3,x_\circ}) \in\PP_3$.
\item $x_\circ \in \Sigma^{kth}$ for some integer $k \geq 4$, if $x_\circ\in \Sigma^{(k-1)th}$ and
$$r^{-k} \big( u(x_\circ+r\,\cdot\,)- \curlyP_{k-1,x_\circ}(r\,\cdot\,)\big) \to p_{k,x_\circ}  \quad \text{in}  \quad  W^{1,2}_{\loc}(\R^n)\cap C^0_{\loc}(\R^n)$$ as $r\downarrow 0$, where $\curlyP_{k-1,x_\circ}=\curlyP_{k-1}(p_{2,x_\circ},\ldots,p_{k-1,x_\circ})$ is the polynomial Ansatz from Definition \ref{def:curlyP} and the limit $p_{k,x_\circ}$ is some $k$-homogeneous harmonic polynomial vanishing on $\{p_{2,x_\circ}=0\}$. Thus we have $(p_{2,x_\circ},\ldots,p_{k,x_\circ}) \in\PP_k$. 
\end{enumerate}
When $x_\circ = 0$, we simply drop $x_\circ$ from $p_{j,x_\circ}$.
\end{definition}
\begin{remark} \label{rem:def sigma k}
Let $u\colon B_1\to[0,\infty)$ solve \eqref{eq:obstacle} and suppose $0\in \Sigma^{kth}$ for some $k\ge 2$. Then
\begin{enumerate}[label={\upshape(\roman*)}]
    \item $\phi^{\gamma}(0^+,u-\curlyP_k)$ exists for all $\gamma\in [k,k+2)$, see Proposition \ref{pro:almost monotonicity}.
    \item for every $\delta>0$ we have $r^{2\phi^\gamma(0^+,u-\curlyP_k)+\delta}\ll H(r,u-\curlyP_k)=o( r^{2k})$ as $r\downarrow 0$. The lower bound follows by Lemma \ref{lem:H ratio}, the upper bound by continuity of the embedding $W^{1,2}(B_1) \to L^2(\de B_1)$.
\end{enumerate}
\end{remark}

With help of Corollary \ref{cor:logmonneau} we prove a Monneau-type monotonicity formula, cf.\ \cite{M03}. The argument is an adaptation of \cite[Lemma 4.1]{FS19}. 
\begin{proposition}[Monneau-type monotonicity]\label{pro:monneau}
Let $u\colon B_1\to[0,\infty)$ solve \eqref{eq:obstacle}. Suppose $0\in \Sigma^{kth}$ for some $k\ge 3$. Let $q$ be any polynomial such that $(p_{2,0},\ldots,p_{k-1,0},q)\in\PP_k$ with $|(p_{2,0},\ldots,p_{k-1,0},q)|\leq \tau$ for some number $\tau>0$. Set
$$
w:=u-\mathcal P_k(p_{2,0},\ldots,p_{k-1,0},q).
$$
Then there exists $r_0=r_0(n,k,\tau)$ such that for all $r\in (0,r_0)$
\begin{equation*}
r^{-2k}H(r,w)\leq C \quad \text{ and }\quad\frac{d}{dr}\left(r^{-2k}H(r,w)\right)\geq -C r^{\eps-1},
\end{equation*} 
for some constants $C=C(n,k,\tau)$ and $\eps(n,k)>0$.
\end{proposition}
\begin{proof} For the sake of the proof let us fix $\gamma = k+1+\frac 1 2$, all constants are allowed to depend on $n,k,\gamma$, even if not explicitly stated and can change value from line to line. We begin by showing $\phi^\gamma(0^+,w)\geq k$.  By constructions of the $\curlyP_k$'s (cf. Proposition \ref{prop:curly A vs curly P}) we have
$$
w=u-\curlyP_{k,0}-p_{k,0}+q+O(x^{k+1}),
$$
so using $0\in \Sigma^{kth}$ we find
\begin{align*}
   H(r,w)^{1/2}&\leq H(r,u-\mathcal P_{k})^{1/2}+H(r,p_{k})^{1/2}+H(r,q)^{1/2}+H(r,O(|x|^{k+1}))^{1/2}\\
   &=o(r^k)+O(r^k).
\end{align*}
If $\phi^{\gamma}(0^+,w)< k$ were true, then Lemma \ref{lem:H ratio} would give for $r\ll 1$
$$
r^{2\phi^{\gamma}(0^+,w)+\delta} \ll H(r,w)+r^{2\gamma} \lesssim r^{2k} ,
$$
which would be a contradiction for $\delta>0$ small. Hence, $\phi^{\gamma}(0^+,w)\geq k$ and we can apply Corollary \ref{cor:logmonneau} with $\gamma=k+1+1/2$ to find that the function
$$
f(r):=\log\left(r^{-2k}\left(H(r,w)+r^{2\gamma}\right)\right)+Cr^{\eps}
$$
is increasing in $(0,r_0)$ for appropriate $r_0,C,\eps$ depending on $n,k,\tau$. Using Lemma \ref{eq:optimalreg+nondeg} we have that $\|w\|_{L^\infty(B_1)}\leq C$, and so $f(r_0)\leq C$. Thus, by monotonicity of $f$
$$
r^{-2k}\left(H(r,w)+r^{2\gamma}\right)\leq C
$$
holds for $r\in (0,r_0)$. Inserting again this estimate in Corollary \ref{cor:logmonneau} we get
$$
\frac{d}{dr}\left(r^{-2k}\left(H(r,w)+r^{2\gamma}\right)\right)\geq - C r^{\eps-1}\left(r^{-2k}\left(H(r,w)+r^{2\gamma}\right)\right)\geq -C r^{\eps-1},
$$
as $\frac{d}{dr}(r^{2(\gamma-k)})\ll r^{\eps-1}$ we conclude.
\end{proof}
The following result proves the continuity of the map $x\mapsto \curlyP_{k,x}$, defined on $\Sigma^{kth}$, for $k\geq 2$. Our argument is a direct adaptation of \cite[Proposition 4.5]{FS19} for the case $k=3$. 
\begin{proposition}
\label{pro:continuity}
Let $u\colon B_1\to[0,\infty)$ solve \eqref{eq:obstacle} and $k\ge 2$.
Then the map $\Sigma^{kth} \ni x\mapsto \mathcal P_{k,x}$ is continuous. Furthermore, there exist a constant $\tau(n,k)$ such that 
\begin{equation}
\sup\left\{|(p_{2,x},\ldots,p_{k,x})| : x\in \Sigma^{kth} \cap B_{1/2}\right\} \leq \tau(n,k).
\end{equation}
\end{proposition}
\begin{proof}
We first prove the bound on $\tau(n,k)$. Since $x\in\Sigma^{kth}$ implies $x\in\Sigma^{jth}$ for $j\leq k$, we proceed by induction on $k$. The inductive step follows from Proposition \ref{pro:monneau} applied the the functions $u(x+\,\cdot\,)$ and $q=0$, which allows to deduce that $p_{k+1,x}$ is bounded in terms of $n,k$ and $\tau(n,k-1)$. We can take as base step $k=2$, for which $|(p_{2,x})|\leq\frac 1 2 $. 

Let us prove continuity at $0$. Again, we proceed inductively and suppose that the statement is true for $k-1$ (see \cite{FS19} for $k-1=2$). Let $(x_\ell)_{\ell\in\N} \subseteq \Sigma^{kth}\cap B_{1/2}$ with $ x_\ell\to 0$, choose a sequence of rotations $(R_\ell)_{\ell\in\N} \subseteq SO(n)$ mapping $\{p_{2,x_\ell}=0\}$ to $\{p_{2,0}=0\}$ for each $\ell$, and satisfying $R_\ell \to \text{id}$. We apply Proposition \ref{pro:monneau} to the functions $u(x_\ell+\,\cdot\,)$ and the polynomials $q_\ell:=p_{k,0}\circ R_\ell$, since
$$
u(x_\ell + y)-\curlyP_{k,x_\ell}(p_{2,x_\ell},\ldots,p_{k-1,x_\ell},q)=u(x_\ell+y)-\curlyP_{k-1,x_\ell}(y)+q_\ell(y)+O(|y|^{k+1}),
$$
we find that the function 
\begin{equation}\label{eq:almost mono1}
r\mapsto \int_{\de B_1} \left|\frac{u(x_\ell+r\,\cdot\,)-\curlyP_{k-1,x_\ell}(r\,\cdot\,)}{r^k}-p_{k,0}\circ R_\ell +\frac{O(|r x|^{k+1})}{r^k}\right |^2 \, d\sigma+ Cr^{\eps}
\end{equation}
is increasing in $(0,r_0)$ for all $\ell\in \N$, for some $r_0$ and $C$ uniform in $\ell$. Using this information for the constant sequence $x_\ell=0$ we find that for any $\delta>0$ there is $r_\delta<\min\{r_0, \delta\}$ such that 
\begin{align}
    \label{eq:blabla}
 \int_{\de B_1} \left|\frac{u(r_\delta\, \cdot\,)-\curlyP_{k-1,0}(r_\delta\, \cdot\,)}{r_\delta^k}-p_{k,0} +\frac{O(|r_\delta x|^{k+1})}{r_\delta^k}\right |^2 \, d\sigma\leq \delta.
\end{align}
Using \eqref{eq:almost mono1} we estimate for each $\ell$ 
\begin{align*}
\int_{\de B_1} \left| p_{k,x_\ell} -p_{k,0}\circ R_\ell\right|^2
&= \lim_{r\downarrow 0} \int_{\de B_1} \left|\frac{u(x_\ell+r\, \cdot\,)-\curlyP_{k-1,x_\ell}(r\, \cdot\,)}{r^k}-p_{k,0}\circ R_\ell\right |^2 \, d\sigma
\\
&\leq \int_{\de B_1} \left|\frac{u(x_\ell+r_\delta\, \cdot\,)-\curlyP_{k-1,x_\ell}(r_\delta\, \cdot\,)}{r_\delta^k}-p_{k,0}\circ R_\ell +O(r_\delta)\right |^2 \, d\sigma+ Cr_\delta^{\eps}.
\end{align*}
As $\curlyP_{k-1,x_\ell} \to \curlyP_{k-1,0}$ by inductive assumption, taking the upper limit in $\ell$ on both sides and using \eqref{eq:blabla}, we find $\limsup_{\ell} \int_{\de B_1} \left| p_{k,x_\ell} -p_{k,0}\right|^2
\leq \delta$, letting $\delta\downarrow 0$ we conclude.
\end{proof}

The following definition is useful to quantify the rate at which $\curlyP_k$ approximates $u$.
 \begin{definition}[Frequency]
\label{def:lambda k}
For $u\colon B_1\to[0,\infty)$ a solution to \eqref{eq:obstacle} and $k\geq 2$, define the $k$-th frequency $\lambda_k\colon \Sigma^{kth}\to [k,k+2]$ by \begin{equation*}
\lambda_{k}(x):=\sup \big\{\phi^\gamma (0^+, u(x +\,\cdot \,)-\curlyP_{k,x}):\gamma \in [k,k+2)\big\}.
\end{equation*}
 At $x=0,$ we write $\lambda_k:=\lambda_k(0)$.
\end{definition}
We comment that in the definition above we indeed have $\lambda_k( \Sigma^{kth}) \subseteq [k,k+2]$. First, the fact that $\phi^\gamma(0^+,u-\curlyP_k)\geq k$ holds for every $\gamma \in [k,k+2)$ was observed in Proposition \ref{pro:monneau}. Second, we always have  $\phi^\gamma(0^+,u-\curlyP_k)\leq \gamma$, as observed in Proposition \ref{pro:almost monotonicity}. The following remark shows that indeed $\phi^\gamma$ is a truncation of the frequency, that is $\phi^\gamma(0^+,u-\curlyP_k)=\min\left\{\lambda_k,\gamma\right\}$.
\begin{lemma}
\label{rem:independence of gamma}
For $u\colon B_1\to[0,\infty)$ a solution to \eqref{eq:obstacle} and $k\ge 2$, consider $x_\circ \in  \Sigma^{kth}.$ Then for all $\gamma\in(\lambda_k,k+2)$ 
\begin{equation*}
    \lambda_k(x_\circ)=\phi^\gamma(0^+,u(x_\circ+\, \cdot\,)-\curlyP_{k,x_\circ})=\lim_{r\downarrow 0}\phi(r,u(x_\circ+\, \cdot\,)-\curlyP_{k,x_\circ})
\end{equation*}
holds, where $\phi(r,v):=D(r,v)/H(r,v)$ is the (non-truncated) Almgren frequency function.
\end{lemma}
\begin{proof} For simplicity, let $x_\circ =0$ and set $v:=u-\curlyP_k$. By Lemma \ref{lem:H ratio}, for each $\delta>0$ there is a constant $c_\delta$ and a radius $r_\delta$, so that $C_\delta r^{2\lambda_k+\delta} \ll H(r,u-\curlyP_k)+r^{2\gamma}$ holds for every $0<r<r_\delta$. Hence, after picking $0<\delta<(\gamma -\lambda_k)/10$, we find 
$$
\phi^\gamma(0^+,v)=\lim_{r\downarrow 0}\frac{\phi(r,v)+o(1)}{1+o(1)}=\lim_{r\downarrow 0}\phi(r,v)=:\widetilde\lambda,
$$
where $\widetilde\lambda$ does not depend on the choice of $\gamma\in(\lambda_k,k+2)$. On the one hand, $\widetilde\lambda\leq \gamma$ holds for any such $\gamma$, implying $\widetilde\lambda\leq \lambda_k$, see Proposition \ref{pro:almost monotonicity}. On the other hand, $\lambda_k\geq \phi^\gamma(0^+,v)= \widetilde\lambda$, by definition.
\end{proof}
We now give a more flexible characterization of $\Sigma^{kth}$.
\begin{lemma}\label{lem:sigmakthsimpledef}
For every solution to \eqref{eq:obstacle} $u\colon B_1\to[0,\infty)$ and $k\ge 2$, there holds
\begin{align*}
   \Sigma^{kth}\equiv \widetilde\Sigma^{kth}:=\Big\{x\in\ 
    &\Sigma^{(k-1)th}\colon \exists (q_2,\ldots,q_k)\in \PP_k, \exists r_\ell\downarrow 0 \text{ such that }\\
    &\left. r_\ell^{-k}\left(u-\curlyP_{k-1}(q_2,\ldots,q_{k-1})\right)_{r_\ell}\weak q_{k}  \quad \text{ in } W^{1,2}_{\loc}(\R^n) \right\}.
\end{align*}
\end{lemma}
\begin{proof}
We just need to show that $\widetilde \Sigma^{kth}\subseteq \Sigma^{kth}$, because the other inclusion follows by definition. Let $0\in \widetilde\Sigma^{kth}$, we know that
 $$r_\ell^{-k}\left(u-\curlyP_{k-1}(q_2\ldots,q_{k-1})\right)_{r_\ell}\weak q_k  \quad \text{ in } W^{1,2}_{\loc}(\R^n),
$$
for a certain $(q_2,\ldots,q_{k})\in \PP_k$ and a certain sequence $r_\ell\downarrow 0$.

We first show that necessarily $q_j=p_{j,0}$ for all $2\leq j\leq k-1$. To prove this we reason inductively and exploit that $0\in \Sigma^{jth}$, in particular, we have the uniform convergence
$$
\lim_{r\to 0} r^{-j}{(u-\curlyP_j(p_{2,0},\ldots,p_{j,0}))_r}= 0.
$$
Suppose $(p_{2,0},\ldots,p_{j-1,0})=(q_2,\ldots,q_{j-1})$ holds for some $j\ge2$. By Proposition \ref{prop:curly A vs curly P} (ii) we have in $L^2(\de B_1)$
\begin{align*}
    0&=\lim_\ell \frac{(u-\curlyP_k(q_{2},\ldots,q_{j}))_{r_\ell}}{r^j_\ell}\\
    &=\lim_\ell \frac{(u-\curlyP_j(p_{2,0},\ldots,p_{j,0}))_{r_\ell}}{r^j_\ell}+\frac{(\curlyP_j(p_{2,0},\ldots,p_{j,0})-\curlyP_k(q_2,\ldots,q_k))_{r_\ell}}{r^j_\ell}\\
    &=0+\lim_\ell \frac{\big(\curlyP_{j-1}(p_{2,0},\ldots,p_{j-1,0})+p_{j,0}-\curlyP_{j-1}(q_2,\ldots,q_{j-1})-q_j+O(|x|^{j+1})\big)_{r_\ell}}{r^j_\ell}\\
    &=p_{j,0}-q_{j}.
\end{align*}
This completes the inductive step. The same computation gives also the base step $q_2=p_{2,0}$.

Now let $\curlyP_k=\curlyP_k(p_2,\dots,p_{k-1},q_k)$, then $\curlyP_k=\curlyP_{k-1}+q_k+P$ for some $(k+1)$-homogeneous harmonic polynomial $P$ (see (ii) in Proposition \ref{prop:curly A vs curly P}). We set $v=u-\curlyP_{k}$ and notice that by assumption $r_\ell^{-k}v_{r_\ell} \weak 0$ in $W^{1,2}_{\loc}(\R^n)$. We will show that this convergence is strong, locally uniform and happens along the full range $r\downarrow 0$, which in turn implies that $0\in \Sigma^{kth}$.

Take $r_0$ as in Proposition \ref{pro:almost monotonicity} and fix some $\gamma \in (k+1,k+2)$. As in (ii) in Remark \ref{rem:def sigma k}, we obtain $H(r_\ell,v)=o(r_\ell^{2k})$ which as in Proposition \ref{pro:monneau} implies that $\lambda:= \lim_{r\to 0 } \phi^\gamma(r,v)\geq k.$
Since by Proposition \ref{pro:almost monotonicity} $\phi^\gamma(\cdot,v)$ is bounded by $C_\gamma$ in $(0,r_0)$, we have
\begin{equation*}
r^{-2k}\int_{B_R}|\nabla v_r|^2 \leq r^{-2k}\phi^\gamma(Rr,v)\left(H(Rr,v)+(Rr)^{2\gamma}\right)\leq C_\gamma r^{-2k}\left(H(Rr,v)+(Rr)^{2\gamma}\right),
\end{equation*}
provided $Rr<r_0$. We now exploit that the logarithm of the right hand side is almost monotone in $r$ thanks to Corollary \ref{cor:logmonneau} and get
\begin{align*}
\limsup_{r\downarrow 0}\log\left( r^{-2k}\int_{B_R}|\nabla v_r|^2 \right)&\leq \log C_\gamma+ \lim_{s\downarrow 0} \log\left( s^{-k}\left(H(s,v)+s^{2\gamma}\right)\right)\\
&=\log C_\gamma+ \lim_{\ell \to \infty} \log\left( r_\ell^{-k}\left(H(r_\ell,v)+r_\ell^{2\gamma}\right)\right)=-\infty,
\end{align*}
thus $\lim_{r\downarrow 0}\| r^{-k}\nabla v_k\|_{L^2(B_R)}=0$ for all fixed $R>0$. The proof of locally uniform convergence is very similar, namely using Lemma \ref{lem:L2toLinfty} and then Lemma \ref{lem:H ratio} we have
$$
 \|v_r\|_{L^\infty(B_R)}\leq C \|v_{Rr}\|_{L^2(B_2\setminus B_{1/2})}+C(Rr)^{k+2}\leq C \,  H(Rr,v)^{\frac{1}{2}}+C(Rr)^{k+2} $$ provided $Rr$ is small, thus we can divide by $r^k$ and  argue as before exploiting the log-monotonicity.
\end{proof}
With the same kind of reasoning we can prove the following basic lemma.
\begin{lemma}\label{lem:large freq}
 Let $u\colon B_1\to[0,\infty)$ be a solution to \eqref{eq:obstacle} and $k\ge 2$, suppose $0 \in \Sigma^{kth}$ with $\lambda_k>k+1$. Then $0\in\Sigma^{(k+1)th}$ and $p_{k+1}=0.$
\end{lemma}
\begin{proof}
  Set $v:=u-\curlyP_k$ and pick any $\gamma\in (\lambda_k,k+2)$, so that $\phi^\gamma(0^+,v)>k+1$. Arguing as in the proof of Lemma \ref{lem:sigmakthsimpledef}, we find
  \begin{equation*}
\label{eq:1}
r^{-2(k+1)}\int_{B_R}|\nabla v_r|^2 \lesssim  r^{-2(k+1)}\left(H(Rr,v)+(Rr)^{2\gamma}\right),
\end{equation*}
 provided $Rr<r_0\ll 1$. On the other hand, as $\phi^\gamma(0^+,v)> k+1$, we have $\phi^\gamma(r,u-\curlyP_k)> k+1$ for small $r\ll 1$. Thus, with Lemma \ref{lem:H ratio} we deduce $H(r,u-\curlyP_k)=o(r^{2(k+1)}).$ Taking the above estimate into account we conclude $r^{-2(k+1)}v_r\to 0$ in $W^{1,2}_{\loc}(\R^n),$ thus by Lemma \ref{lem:sigmakthsimpledef} $0\in \Sigma^{(k+1)th}.$
\end{proof}
Finally, we study the blow-up's of $u-\curlyP_k$ when Lemma \ref{lem:large freq} does not apply. Our argument is an adaptation of \cite[Proposition 2.10]{FS19}. We will study the sequence of functions $\tilde v_r:=H(r,u-\curlyP_k)^{-\frac 1 2}(u-\curlyP_k)(r\, \cdot \, )$ as $r\downarrow 0$. Any limit of $\tilde v_r$ will be a $\lambda_k$-homogeneous solution of a certain PDE (the Signorini Problem \eqref{eq:signorini}), but not necessarily a polynomial.
\begin{proposition}
\label{pro:convergence}
Let $0\in \Sigma^{kth}$ with $\lambda_k\leq k+1$. Let $(r_\ell)_{\ell \in \N}$ be an infinitesimal sequence and let $x_\ell \in \Sigma^{kth}\cap B_{r_\ell}$. For every $\ell$ set $v_{x_\ell}:=u(x_\ell+\cdot)-\curlyP_{k,x_\ell}$ and suppose that $\lambda_k(x_\ell) \to \lambda_k$. Consider the sequence 
\begin{equation*}
    \widetilde{ v}_{r_\ell,x_\ell}:=\frac{v_{x_\ell}(r_\ell\, \cdot)}{H(r_\ell,v_{x_\ell})^{\frac 1 2}},
\end{equation*} 
then
\begin{enumerate}[label={\upshape(\roman*)}]
    \item  $(\widetilde{ v}_{r_\ell,x_\ell})_{\ell \in \N}$ is bounded in $W^{1,2}_{\loc}(\R^n)$ and $C^{0,\frac{1}{n+1}}_{\loc}(\R^n)$.
    \item If $\widetilde{ v}_{r_\ell,x_\ell} \weak q\in W^{1,2}_{\loc}(\R^n)$, then the convergence is in fact strong and $q$ must be a nontrivial $\lambda_k$-homogeneous solution of the Signorini problem with obstacle $\{p_{2}=0\}$, that is
    \begin{equation}\label{eq:signorini}
        \begin{cases}
           \lap q\leq 0\text{ and }q\lap q= 0 & \text{ in }\R^n,\\
           \lap q=0 &\text{ in }\R^n\setminus\{p_{2}=0\},\\
           q\geq 0 &\text{ on }\{p_{2}=0\}.
        \end{cases}
    \end{equation}
     Finally, if $\lambda_k<k+1$, then $q$ is even with respect to the thin obstacle.
\end{enumerate}
\end{proposition}
\begin{proof}
For the sake of readability we set $v_\ell:=v_{x_\ell}$ and $\tilde v_\ell:=\tilde v_{r_\ell,x_\ell}$, furthermore we will omit the dependence of the constants from $n,k$ and we set $\delta:=\frac{\eps}{100}$ where $\eps(n,k)$ is the same as Proposition \ref{pro:monneau}.

Without loss of generality assume $\ell$ is so large that $x_\ell\in B_{1/2}$ and $\lambda_k(x_\ell)\leq \lambda_k+\delta$. Within this proof we fix $\gamma:=k+1+\frac{3}{4} $, so by Remark \ref{rem:independence of gamma} we have that $\lambda_k(x_\ell)=\phi^\gamma(0^+,v_\ell)$ for all $\ell$. 

\textbf{Step 1.} We claim that there are $\eps,r_0\in(0,1/2)$ and $C_0,c_0>0$ all independent on $\ell$, and $\ell_0$ such that if we set
\begin{align*}
    f_{x_\ell}(r):=\phi^\gamma(r,v_\ell)+C_0r^{\eps},\qquad h_{x_\ell}(r):=r^{-2k}H(r,v_\ell) +C_0r^{\eps},
\end{align*}
then we have
\begin{itemize}
    \item[(a)] $f_{x_\ell},h_{x_\ell}$ are continuous and increasing on $[0,r_0]$ and converge uniformly in this interval to $f_0,h_0$ respectively. Furthermore $h_{x_\ell}(0^+)=0$ identically.
    \item[(b)] We have $f_{x_\ell}(r)\leq \lambda_k+2\delta$ and 
    $ H(r,v_\ell)\geq c_0 r^{2\lambda_k+5\delta}$ for all $r\in[0,r_0]$ and all $\ell>\ell_0$.
\end{itemize}
The fact that for each $\ell$ both $f_{x_\ell}$ and $h_{x_\ell}$ are increasing is consequence of Proposition \ref{pro:almost monotonicity} and Proposition \ref{pro:monneau}, respectively. By Proposition \ref{pro:continuity} we can choose $\tau:=\tau(n,k)$ so that $r_0$ and $C_0$ can be taken uniform in $\ell$. Since $x_\ell\in \Sigma^{kth}$ we already observed in Remark \ref{rem:def sigma k} that $h_{x_\ell}(0^+)=0$, furthermore by assumption $f_{x_\ell}(0^+)=\lambda_k(x_\ell)\to \lambda_k=f_0(0^+)$, thus $f_{x_\ell}\to f_0$ and $h_{x_\ell}\to h_0$ pointwise. As they are monotone and the limit functions are continuous, the convergence must be uniform, thus (a) is proved. We turn to (b): possibly taking a smaller $r_0$, we have that $f_0\leq \lambda_k+\delta$ in $[0,r_0]$ thus by uniform convergence there is $\ell_0$ such that $f_{x_\ell}\leq \lambda_k+2\delta$. Now the last statement follows if we apply Lemma \ref{lem:H ratio} with $w=v_\ell,R=r_0, \overline\lambda=\lambda_k+\delta,\delta=\frac{\eps}{100}$.

We will use many times through the proof that
\begin{equation}\label{eq:growth ass}
    c_0 r^{2\gamma}\leq H(r,v_\ell) r^{\frac{1}{2}}\qquad \text{ in }[0,r_0],
\end{equation} uniformly in $\ell>\ell_0$. This is a direct consequence of (b) and the fact that $2\gamma-\frac {1}{2}>2\lambda_k+5\delta$.\\ 

\textbf{Step 2}. We prove (i). Fix some $R>1$ and for $Rr_\ell<r_0$ write
\begin{align*}
    D(R,\tilde v_\ell)\leq \phi^{\gamma}(Rr_\ell,v_\ell)\frac{H(Rr_\ell, v_\ell) +(Rr_\ell)^{2\gamma}}{H(r_\ell,v_\ell)}\leq C_R \frac{H(Rr_\ell, v_\ell)}{H(r_\ell,v_\ell)}+ C_R\frac{r_{\ell}^{2\gamma}}{H(r_\ell,v_\ell)} \leq C_R+o(1)
\end{align*}
where we used \eqref{eq:growth ass} and Lemma \ref{lem:H ratio}, to found the second and the first addendum, respectively. Thus $\|\nabla \tilde v_\ell\|_{L^2(B_R)}$ is bounded in $\ell$.

Now we want to combine the uniform $L^2$-bound on $\nabla\tilde v_\ell$ and the Lipschitz estimate on $\der_{x'} \tilde v_r$ to entail uniform H\"older bounds.  Fix some $\ell$ and choose coordinates such that $p_{2,x_\ell}=\frac 1 2 x_n^2$, by Proposition \ref{prop:lipschitz} we have for $2\theta Rr_\ell<r_0$ and $j\neq n$
\begin{align*}
    \|\de_j\tilde v_\ell \|_{L^\infty(B_R)}\leq C \|\tilde v_\ell \|_{L^2(B_{2\theta R} \setminus B_{\theta R /2})}+C\frac{r^{k+2}}{H(r_\ell,v_\ell)^{\frac{1}{2}}}
\end{align*}
the second term is $o(1)$ in $\ell$ again by \eqref{eq:growth ass}, while for the first we employ again Lemma \ref{lem:H ratio}:
\begin{equation}\label{eq:annulitoshells}
\| v_\ell(r_\ell\, \cdot) \|^2_{L^2(B_{2\theta R} \setminus B_{\theta R /2})}=\int_{\theta R/2}^{2\theta R} H(sr_\ell,v_\ell)s^{n-1}\, ds\leq C_RH(r_\ell,v_\ell).
\end{equation}
Hence, $\|\de_j\tilde v_\ell \|_{L^\infty(B_R)}\leq C_R$ for all $j\neq n$. By Lemma \ref{lem:hoelder}, this gives the H\"older bound:
$$
[\tilde v_\ell]_{C^{0,\frac{1}{n+1}}(B_{R/2})}\leq C_R(\|\nabla_{x'} \tilde v_\ell\|_{L^\infty(B_R)}+\|\nabla \tilde v_\ell\|_{L^2(B_R)})\leq C_R.
$$
This concludes the proof of (i).\\

\textbf{Step 3.} We turn to (ii) and prove that $q$ solves \eqref{eq:signorini}. Since $\lap v_\ell =-\chi_{\{u(x_\ell+\,\cdot\,)=0\}}\leq 0$, we have that $\lap q\leq 0$ weakly in $\R^n$. Furthermore, integrating by parts with some cut-off function $\chi_{B_R}\leq \psi \leq \chi_{B_{2R}}$ leads to  
$$
\int_{B_R} |\lap\tilde v_\ell|\leq -\int_{\R^n}  \lap\tilde v_\ell \psi \leq C_R \|\tilde v_\ell\|_{L^2(B_{2R}\setminus B_{R})}\leq C_R,
$$
where in the last step we argued as in \eqref{eq:annulitoshells}. Hence by compactness $\lap \tilde v_\ell \weakstar \lap q$ in $C_c(\R^n)^*$. On the other hand by (i) $\tilde v_\ell\to q$ locally
uniformly and so
$$
\tilde v_\ell \lap \tilde v_\ell \weakstar q\lap q \qquad \text{ in }C_c(\R^n)^*.$$ We now apply Proposition \ref{pro:almost monotonicity} to $v_\ell$ with our particular choice of $\gamma$ and recall that by \eqref{eq:almostmonotonicity} we have for $Rr_\ell<r_0$
\begin{equation*}
    \int_{B_R}|\tilde v_\ell \lap \tilde v_\ell|\frac{H(r_\ell,v_\ell)}{H(Rr_\ell,v_\ell)+(Rr_\ell)^{2\gamma}}\leq C_R r_\ell^{\eps}=o(1).
\end{equation*}
Notice that the constants are independent on $\ell$ as we can choose a uniform $\tau=\tau(n,k)$ for all $x_\ell\in \Sigma^{kth}\cap B_{1/2}$ by Proposition \ref{pro:continuity}. Sending $\ell\uparrow \infty$ we get $q\lap q=0$. In order to show $\lap q=0$ outside $\{p_{2,0}=0\}$, we exploit once again Lemma \ref{lem:importantlemma} to find
$$
B_{R/2}\cap\supp\left(\lap \tilde v_\ell\right)\subseteq\left\{\dist(\{p_{2,x_\ell}=0\},\,\cdot\,)\leq C_R r_\ell^{\alpha_\circ}\right\},
$$
as $p_{2,x_\ell} \to p_{2,0}$ and $R$ can be taken arbitrarily large we deduce that $\supp\lap q\subseteq \{p_{2,0}=0\}$. It remains to show that $q$ is non-negative on the thin obstacle, up to a rotation we can assume $p_{2,0}=\frac 1 2 x_n^2$. Pick $x_*\in \{x_n=0\}$  and consider some sequence $(y_\ell)_{\ell \in \N}$ such that
$$
y_\ell \in \{ \curlyA_{k,x_\ell}(r_\ell\, \cdot\,)=0\},\qquad y_\ell \to x_*,
$$
thus by locally uniform convergence and \eqref{eq:growth ass} $$q(x_*)=\lim_\ell \tilde v_{r_\ell}(y_\ell)=\lim_\ell \frac{u(r_\ell y_\ell)-\frac 1 2  \curlyA_{k,x_\ell}^2(r_\ell y_\ell)+O(r_\ell^{k+2})}{H(r_\ell,v_\ell)^{1/2}}\geq 0.$$ To construct such sequence set $y_\ell:=\mathit{\Phi}_\ell(x_*)$ where $\mathit{\Phi}_\ell\in C^\infty (B_{R_\ell}) $ are the inverse functions of $\mathit{\Psi}_\ell \colon x\mapsto (x',r_\ell^{-1}(\curlyA_{k,x_\ell})_{r_\ell})$. Notice that $\mathit{\Psi}_\ell \to \text{id} $ in $C^1_{\loc}$and that $R_\ell \uparrow +\infty$ as $\ell \to \infty$. So $\mathit{\Phi}_\ell \to \text{id}$, and $x_*\in B_{R_\ell}$ eventually, thus $y_\ell \to x_*$. Therefore
$$ q \ge 0 \quad \text{on} \quad \{x_n=0\}.$$
Hence we proved that $q$ is a global solution of the Signorini problem \eqref{eq:signorini}.\\
  
  \textbf{Step 4.} We show that $\tilde v_{\ell}\to q$ in $W^{1,2}_{\loc}(\R^n)$ and that $q$ is $\lambda_k$-homogeneous. Fix any $\eta \in C^\infty_c(\R^n)$ and exploit as before that $\|\tilde v_{r_\ell}\lap \tilde v_{r_\ell}\|_{L^1(B_R)}\to 0$ and integrate by parts in $\R^n$
  \begin{align*}
      \int |\der(\eta\tilde v_{\ell})|^2&=-\int \eta \tilde v_{\ell} \lap (\eta \tilde v_{\ell})\\
      &=-\int \left(\eta  \tilde v_{\ell}^2\lap \eta+2 \eta \tilde v_{\ell} \der \eta\cdot \der \tilde v_{\ell} +\eta^2 \tilde v_{\ell}\lap \tilde v_{\ell}\right)\\
      &\leq -\int \left(\eta \lap \eta \tilde v_{\ell}^2+2 \eta \tilde v_{\ell} \der \eta\cdot \der \tilde v_{\ell} \right) +C(\eta) \|\tilde v_{\ell}\lap \tilde v_{\ell}\|_{L^1(B_R)}
  \end{align*}
  Taking the upper limit and using $\der \tilde{v}_{\ell}\weak \der q$ in $L^2_{\loc}(\R^n)$ and $\tilde v_{\ell}\to q$ in $C^0_{\loc}(\R^n)$, we get
  \begin{align*}
  \limsup_{\ell}\int |\der(\eta\tilde v_{\ell})|^2 \leq-\int \left(\eta \lap \eta q^2+2 \eta q \der \eta\cdot \der q + \eta^2 q\lap q\right)=\int |\der(\eta q)|^2,
  \end{align*}
  where we used $q\lap q=0$. By weak lower semicontinuity we always have the converse inequality
  thus $\der(\eta \tilde v_{\ell})\to \der (\eta q )$ strongly in $L^2(\R^n)$. This in particular gives for every $R>0$
  \begin{equation*}
  \phi(R,q)=\lim_\ell \phi(R,\tilde v_{\ell})=\lim_\ell \phi(Rr_\ell, v_\ell)=\lim_\ell \phi^\gamma (Rr_\ell,v_\ell)
  \end{equation*}
  where in the last passage we used \eqref{eq:growth ass}. On the other hand, (a) in Step 1 implies
  $$
  \lim_\ell \phi^\gamma (Rr_\ell,v_\ell)=\lim_\ell f_{x_\ell}(Rr_\ell)=f_0(0^+)=\lambda_k,
  $$
  thus $\phi(R,q)\equiv\lambda_k$ for all $R>0$. As a standard consequence $q$ is $\lambda_k$-homogeneous, see \cite{ACS08}.\\
  
  \textbf{Step 5.} We finally prove that for $\lambda_k<k+1$ we have $q^{\text{odd}}=0$. Notice that by $q^{\text{odd}}$ is harmonic thus has integral homogeneity, hence the only nontrivial case is when $\lambda_k=k$. We need to show that $q$ is orthogonal in $L^{2}(\de B_1)$ to every $k$-homogeneous harmonic polynomial $P$ vanishing on $\{p_{2,0}=0\}$. Pix such $P$ and apply Proposition \ref{pro:monneau} with $$w_\ell:=u(x_\ell+\cdot)-\curlyP_k(p_{2,x_\ell},\ldots,p_{k-1,x_\ell},p_{k,x_\ell}-P\circ R_\ell),$$ where $R_\ell$ are rotations sending $\{p_{2,x_\ell=0}\}$ to $\{p_{2,0}=0\}$ and $R_\ell \to \text{id}$. Thus we have with constants uniform in $\ell$ ($P$ is fixed)
\begin{align*}
        r^{-2k}H(r,w_\ell)+Cr^{\eps}&=Cr^{\eps}+\int_{\de B_1}\left(\frac{v_\ell(r\, \cdot)}{r^k}+P\circ R_\ell+O(|x|^{k+1}/r^{k})\right)^2\\
    &\geq \lim_{r\to 0}r^{-2k}H(r,w_\ell)+Cr^{\eps}=\int_{\de B_1} P^2.
\end{align*}
Now divide by the sequence $\eps_\ell :=\left(H(r_\ell,v_\ell)r_\ell^{-2k}\right)^{1/2}$, which by Step 1 (b) satisfies (recall $\lambda_{k}=k$)
$$
 r_\ell^{5\delta} \leq \eps_\ell^2\leq h_{x_\ell}(r_\ell)=o(1),
$$
we compute the squares and rearrange the terms to get:
$$
\int_{\de B_1}{\tilde v_\ell}^2 \eps_\ell + 2\int_{\de B_1} \tilde v_\ell P\circ R_\ell\geq -C \frac{r_\ell^{\eps}}{\eps_\ell}\geq -C r_\ell^{\eps-5\delta/2}.
$$
Since $\delta\leq \eps/100$ we can send $\ell \to \infty$ and get
$$
\int_{\de B_1} qP\geq 0,
$$
the conclusion follows by linearity in $P$.
\end{proof}
\begin{remark}
An important application of Proposition \ref{pro:convergence} is when the sequence $x_\ell\equiv 0$ identically.
\end{remark}
\begin{remark}\label{rmk:uscfrequency}
Step 1 of the proof shows that the function $\Sigma^{kth}\ni x\mapsto \phi^{\gamma}(u(x+\,\cdot\,)-\curlyP_{k,x},0^+)$ is upper semicontinuous. In fact with the same notations we have for each $r<r_0$ 
\begin{align*}
    \limsup_\ell  \phi^{\gamma}(u(x_\ell+\,\cdot\,)-\curlyP_{k,x_\ell},0^+)& \leq \limsup_\ell \phi^{\gamma}(u(x_\ell+\,\cdot\,)-\curlyP_{k,x_\ell},r)+ C r^{\eps}\\
    &=\phi^{\gamma}(u-\curlyP_{k},r)+ C r^{\eps},
\end{align*}
and the conclusion follows letting $r\downarrow 0$.
\end{remark}
Proposition \ref{pro:convergence} shows that in order to pursue our analysis further we need to have some basic knowledge about homogeneous solutions of the Signorini Problem \eqref{eq:signorini}. In the next chapter we will use extensively the results reported in Section \ref{subsec:signorini}.

\section{Estimating the size of the sets \texorpdfstring{$\bm{\Sigma^{kth}\setminus \Sigma^{(k+1)th}}{}$}{}}\label{sec:dimensionreduction}
Throughout this section $u$ will be a solution of \eqref{eq:obstacle} with $f\equiv 1$ and $\mu=1$. We will show that, for all $k \geq 2$, 
$$
\dim_{\mathcal H} (\Sigma^{kth}\setminus \Sigma^{(k+1)th})\leq n-2,\quad \text{ and is countable if } n=2.
$$
In the last subsection we will show how this constrains the geometry of $\Sigma$. We remark that by Caffarelli's analysis $\Sigma\setminus \Sigma^{2nd}$ has locally finite $\mathcal H^{n-2}$ measure (see e.g., \cite[Theorem 8 (c)]{C98}).

In this chapter we repetitively use the facts/notations concerning the Signorini problem recalled/established Section \ref{subsec:signorini}. In particular we will use: $\Sign_k, \Signeven, q^{\text{even}},q^{\text{odd}}, \Sigma(q),\ldots$

We need to understand the nature of points in $\Sigma^{kth}\setminus \Sigma^{(k+1)th}$, therefore suppose $0\in\Sigma^{kth}$ and $0\notin \Sigma^{(k+1)th}$. We necessarily have $\lambda_k\leq k+1$, see Lemma \ref{lem:large freq}. Notice that, with the notations of Proposition \ref{pro:convergence} and Lemma \ref{lem:H ratio}, we have $$
\frac{(u-\curlyP_k)(r\, \cdot\,)}{r^{k+1}}=\frac{{H(r,u-\curlyP_k)}^{\frac 1 2}}{r^{k+1}}\, \frac{(u-\curlyP_k)(r\, \cdot\,)}{H(r,u-\curlyP_k)^{\frac 1 2}}=O\left(r^{\lambda_k-(k+1)}\right)\ \tilde v_{r,0}.
$$ 
As every accumulation point of $\tilde v_{r,0}$ equals some non-zero $q\in \Sign_{k+1}(\{p_2=0\})$ (see Proposition \ref{pro:convergence}), in order to conclude $0\notin \Sigma^{(k+1)th}$ (cf.\ the flexible definition of this set from Lemma \ref{lem:sigmakthsimpledef}), either we must have $\lambda_k < k+1$ or $\lambda_k =k+1$ and every accumulation point $q$ satisfies $q^{\text{even}}\not\equiv 0$. 


These observations inspire the following trichotomy. If $x\in\Sigma^{kth}\setminus \Sigma^{(k+1)th}$ then exactly one of the following happens:
\begin{itemize}
    \item[(1)] $\lambda_k(x)=k$,
    \item[(2)] $\lambda_k(x)\in(k,k+1)$.
    \item[(3)] $\lambda_k(x)=k+1$, but every accumulation point of $r^{-(k+1)}(u(x+\cdot)-\curlyP_{k,x})(r\, \cdot\,)$ has a non-zero even part.
\end{itemize}
We rephrase these cases with a notation closer to the one adopted in \cite{FS19,FRS19}. Namely, for each $k\ge 2$ define
\begin{align*}
    \Sigma^{>k}:=\Sigma^{kth}\cap \left\{\lambda_k>k\right\},\qquad
    \Sigma^{\geq k+1}:=\Sigma^{kth}\cap \left\{\lambda_k\geq k+1\right\}.
\end{align*}
So that we have the descending chain of inclusions
$$
\Sigma_{n-1}=\Sigma^{2nd}= \Sigma^{>2}\supset\ldots\supseteq\Sigma^{kth}\supseteq\Sigma^{>k}\supseteq\Sigma^{\geq k+1}\supseteq\Sigma^{(k+1)th}\supseteq\ldots \supseteq\bigcap_{j\geq 2} \Sigma^{jth}=:\Sigma^\infty.
$$
With this notation case (1) corresponds to the set $\Sigma^{kth}\setminus \Sigma^{>k}$, case (2) to $\Sigma^{>k}\setminus \Sigma^{\geq k+1}$ and case (3) to $\Sigma^{\geq k+1}\setminus \Sigma^{(k+1)th}$.
In the next subsections we will address separately each case.

We point out that in cases (1) and (3) the parity of $k$ will play a role in our arguments. This is related to the different shape of the functions in $\Signeven_k$ according to the parity of $k$ (see Proposition \ref{pro:signorini}).
\subsection{The size of \texorpdfstring{$\bm{\Sigma^{kth}\setminus \Sigma^{> k}}{}$}{}}
\label{sec: frequency k}
We start by showing that when $k$ is even, the set $\Sigma^{kth}\setminus \Sigma^{> k}$ is in fact empty. This is a simple consequence of the following monotonicity formula, which is an extension of \cite[Lemma 4.14]{FRS19} to higher values of $k$.
\begin{lemma}
\label{lem: monotonicity k even} 
Let $u$ solve \eqref{eq:obstacle} and let $0\in \Sigma^{kth}$ for some $k\geq 2$. Let $v:=u-\curlyP_k$ and let $P$ be any $k$-homogeneous harmonic polynomial such that $P \geq 0$ on $\{p_2=0\}.$ Then there exists $\eps,r_0>0$ depending on $n,k$ such that for all $r\in (0,r_0)$ it holds
\begin{equation*}
    \frac{d}{dr} \left( r^{-k} \int_{\de B_1} v_r P\right) \leq Cr^{\eps-1},
\end{equation*}
for some constant $C$ depending only on $n,k,\|P\|_{L^2(\de B_1)}.$
\end{lemma}
\begin{proof}
The proof is identical as in \cite[Lemma 4.14]{FRS19}, we give it nevertheless for the reader's convenience. Integration by parts and the fact that $P$ is harmonic lead to
\begin{equation*}
\frac{d}{dr} \int_{\partial B_1}  v_r P  = \frac{1}{r} \bigg( \int_{\partial B_1}   v_r \partial_\nu P + \int_{B_1}  \lap v_r P\bigg)
= \frac{1}{r} \bigg( k\int_{\partial B_1}   v_r q + \int_{B_1}  \lap v_r P\bigg),
\end{equation*}

where we used the homogeneity of $P$ to deduce $\partial_\nu P=kP$ on $\partial B_1$.
As $\lap v_r=-r^2\chi_{\{u_r=0\}}$, we can rewrite this as
\[\frac{d}{dr}  \bigg( r^{-k} \int_{\partial B_1} v_r P\bigg) = -\frac{1}{r^{k-1}}\int_{B_1\cap \{u_r=0\}} P.\]
We have 
$$
r^{-k}\|v_r\|_{L^2(B_1)}=\|\tilde v_r\|_{L^2(B_1)}\, \left( r^{-2k}H(r,v)\right)^{\frac 1 2}\leq C(n,k)
$$
for $r\lesssim 1$ sufficiently small, thanks to Proposition \ref{pro:convergence} (i) and \ref{pro:monneau}. Combining this estimate with the Lipschitz bounds from Proposition \ref{prop:lipschitz} (ii), with $\beta \in (0,\frac{1}{k+2})$ to be chosen, we find 
$$\{u_r =0\}\cap B_1 \subseteq \left\{ x\in B_1 : r|\de_n \curlyP_k|(rx)=r^2|x_n+O(|x|^2)|\le Cr^{k(1-\beta)}\right\}$$
with some constant $C$ depending on $n,k$ only (as we can choose $\tau=\tau(n,k)$ from Proposition \ref{pro:continuity}). This shows that $|\{u_r=0\}\cap B_1|\leq C r^{k(1-\beta)-2}$. On the other hand, by the maximum principle, we have $-P \leq C|x_n|$ in $B_1$. Hence, using Lemma \ref{lem:importantlemma}, we obtain
\[-\int_{B_1\cap \{u_r=0\}} P \leq Cr^{\alpha_\circ}\,\big|\{u_r=0\}\cap B_1\big| \leq Cr^{k+\alpha_\circ-k\beta-2},\]
and the lemma follows choosing $\beta=\alpha_\circ/2k$.
\end{proof}
As a simple corollary we get our claim.
\begin{corollary}
\label{cor:k even impossible} For every even integer $k\ge 2$ it holds ${\Sigma^{kth}\setminus \Sigma^{> k}}=\emptyset$.
\end{corollary}
\begin{proof}
Let us assume, on the contrary, that $0\in{\Sigma^{kth}\setminus \Sigma^{> k}}$, that is $\lambda_k=k$. Then, by Proposition \ref{pro:convergence}, any accumulation point $q$ of $\tilde v_r= \frac{(u-\curlyP_k)_r}{H(r,u-\curlyP_k)^{1/2}}$ lies in $\Signeven_k(\{p_2=0\})\setminus\{0\}$. Furthermore, by Proposition \ref{pro:signorini}, any such $q$ satisfies the assumptions of Lemma \ref{lem: monotonicity k even}. As $0\in \Sigma^{kth},$ we moreover have $\frac{v_r}{r^k}\to 0$, thus after combining this with Lemma \ref{lem: monotonicity k even} with $P=q$, we find
\begin{equation*}
  r^{-k} \int_{\de B_1} v_r q \leq Cr^{\eps}
\end{equation*} for small $r\leq 1$. Dividing by $H(r,u-\curlyP_k)^{1/2}$ leads to 
\begin{equation*}
 \int_{\de B_1} \tilde v_{r_\ell} q \leq C\frac{r_\ell^{k+\eps}}{H(r_\ell,u-\curlyP_k)^{1/2}}.
\end{equation*}
Thanks to (ii) in Remark \ref{rem:def sigma k}, we deduce that the right hand side vanishes as $\ell\uparrow \infty$, implying $
 \int_{\de B_1} q^2 \leq0$ and contradicting $\|q\|_{L^2(\de B_1)} =1$. 
\end{proof}

Let us now consider an odd $k$. We point out that for $k=3$ it still holds $\Sigma^{3rd}\setminus \Sigma^{\geq 3}=\emptyset$, but the proof is more refined, see \cite[Proposition 5.8]{FRS19}. We will instead rely on a more robust argument which will be also employed later, to deal with the case $\lambda_k=k+1$ (cf. Lemma \ref{lem:k+1}). The main step is contained in the following lemma, based on a barrier argument.
\begin{lemma}\label{lem:kodd}
 Let $k\geq 3$ be odd. For all $x\in\Sigma^{kth}\setminus \Sigma^{>k}$ and $\eps>0$, there is $\varrho=\varrho(\eps,x)>0$ such that, for each $0<r<\varrho$, exists $q\in \Signeven_k(\{p_{2,x}=0\})$ such that
\begin{equation}\label{eq:reifenbergtype}
    \Sigma(u) \cap B_r(x)\subseteq \Sigma(q)+B_{\eps r}(x).
\end{equation}
Recall that $\Sigma(q):=\left\{q=|\nabla q|=0\right\}\cap \{p_{2,x}=0\}$ was defined in \eqref{eq:definitionsingularsetsignorini}.
\end{lemma}
\begin{proof}
  Up to an isometry suppose $x=0$ and $p_2=\frac 1 2 x_n^2$. We argue by contradiction and (rescaling the space) suppose that there are $\eps_\circ>0$ and $r_\ell\downarrow 0$ such that
  \begin{equation*}
      \Sigma(u_{r_\ell})\cap \{y\in B_1:\dist(y,\Sigma(q))>\eps_\circ \}\neq\emptyset\qquad \text{ for all } q\in\Signeven_k.
  \end{equation*}
  Thanks to Proposition \ref{pro:convergence}, we can extract a subsequence (that we do not rename) such that
  \begin{equation}\label{eq:uniformconvergence}
      \frac{(u-\curlyP_k)_{r_\ell}}{H(r_\ell,u-\curlyP_k)^{1/2}}\to \overline q\in\Signeven_k\setminus \{0\},\quad \text{ in }C^0_{\loc}(\R^n).
  \end{equation} 
  Thus, there are $y_\ell\in B_1$ such that
  \begin{equation*}
      y_\ell \in \Sigma(u_{r_\ell}),\ \  \dist(y_\ell,\Sigma(\overline q))\geq \eps_\circ,\ \  y_\ell\to y_\infty\in \{x_n=0\}.
  \end{equation*}
  By Proposition \ref{pro:signorini} we can write $\overline q(x)=-|x_n|(q_0(x')+x_n^2q_1(x))$ for some polynomials $q_0,q_1$, with $q_0\geq 0$. For brevity, we set $h_\ell:={H(r_\ell,u-\curlyP_k)}^{\frac 1 2}$ and remark that $r_\ell^{k+\delta }\leq C_0 h_\ell\ll r_\ell^{k}$ for some constants $C_0,\delta$ (see Remark \ref{rem:def sigma k}).
  
  As $y_\ell \to y_\infty$, in order to reach for a contradiction it suffices to show that for some small radius $R>0$ and all $\ell$ large
  \begin{equation}\label{eq:contradiction}
      \Sigma(u_{r_\ell})\cap B_{R}(y_\infty)=\emptyset.
  \end{equation}
  The rest of the proof is devoted to showing \eqref{eq:contradiction} for a suitable $R$ independent on $\ell$. 
  
  We start choosing a radius $\rho$ as follows. As $y_\infty\in \{x_n=0\}\setminus \Sigma(\overline q) \subseteq \{q_0 >0\}$ we can find some small $\rho,m\in(0,1)$ such that
\begin{equation}\label{eq:rhosmallness}
q(x)\leq -m |x_n| \qquad \text{ for }x\in B_{4\rho}(y_\infty),
\end{equation}
and we can also require that $\rho\leq m /{100}$.

Let us introduce some notation. Define the set of points that admit a barrier
\begin{equation*}
    Z_\ell:=\big\{z\in B_{\rho}(y_\infty): \exists\ \phi_{z,\ell} \text{ of class }  C^2\text{ in a neighbourhood of } B_\rho(z) \text{ solving } \eqref{eq:Pluto}\big\},
\end{equation*}
  where
  \begin{equation}\label{eq:Pluto}
	\begin{cases} 
	\phi_{z,\ell}(z)=0,  \\ 
	\phi_{z,\ell}  \geq 0 & \text{ in } B_\rho(z), \\ 
	\lap \phi_{z,\ell} < r_\ell^2 &\text{ in } \overline{ B_\rho(z)},\\
	u(r_\ell\,\cdot\,)<\phi_{z,\ell} & \text{ on } \de B_{\rho}(z).
	\end{cases}
\end{equation}
We also set 
\begin{equation*}
    \Gamma_\ell:=\left\{\curlyA_k(r_\ell\, \cdot\, )=0\right\}\cap B_2.
\end{equation*}
For $\ell$ large in terms of $n$,$k$ and $\tau$, $\Gamma_\ell$ is a smooth hypersurface inside $B_2$, which converges to the hyperplane $\{x_n=0\}$, say, in the $C^2$ norm. 
Furthermore, combining Lemma \ref{lem:importantlemma} and the fact that $\curlyA_k(x)=x_n+O(|x|^2)$, we get that
\begin{equation}\label{eq:zsmall}
\big(\{u_{r_\ell}=0\}\cup\Gamma_\ell\big)\cap B_2\subseteq \{|x_n|\leq C_1 r_\ell^{\alpha_\circ}\}\cap B_2
\end{equation}
for some constant $C_1=C_1(n,k)>0$.

\textbf{Claim.} There is $\ell_0$ such that for $\ell>\ell_0$ the following holds.
\begin{enumerate}[label={\upshape(\roman*)}]
\item All points in $B_\rho(y_\infty)$ belonging to the hypersurface $\Gamma_\ell$ admit a barrier, that is
$\Gamma_\ell \cap B_\rho(y_\infty)\subseteq Z_\ell.$
\item  The function $u$ vanishes on $Z_\ell$. Moreover, $Z_\ell$ is open and contained in the interior of the contact set $\left\{u(r_\ell\, \cdot\, )=0\right\}.$
\item There is a dimensional constant $N(n)>0$ such that $ \left( \Sigma(u_{r_\ell})\cap B_{\rho/2N}(y_\infty)\right) \setminus \Gamma_\ell=\emptyset.$ 
\end{enumerate}

 The combination of these three claims will give \eqref{eq:contradiction} with $R=\rho/2N$. In fact, as there are no singular points in the interior of the singular set, (i) and (ii) give $\Sigma(u_{r_\ell})\cap B_{\rho/2N}(y_\infty) \cap \Gamma_\ell=\emptyset$. We conclude with (iii).\end{proof}
 
 \begin{proof}[Proof of the Claim.]
 We begin by proving (ii). First, for any $z\in Z_\ell$ and any $\xi$ close to $z$ we can define
$$
\phi_{\xi,\ell}(x):=\phi_{z,\ell}(x+(z-\xi)).$$
By continuity of translations $\phi_{\xi,\ell}$ solves \eqref{eq:Pluto} on $B_\rho(\xi)$, for $|\xi-z|$ small enough, hence $Z_\ell$ is open. We now apply the comparison principle, using to the last two properties of the barrier in \eqref{eq:Pluto} to find $u|_{Z_\ell}\equiv 0.$ Notice that for $z\in Z_\ell$, two cases arise: either for all $c>0$ we have $u(r_\ell\cdot)<\phi_{z,\ell}+c $ on $B_\rho(z)$, or there exist the largest $c_\ast>0$ such that $u(r_\ell x_\ast)=\phi_{z,\ell}(x_\ast )+c_\ast$ for some $x_\ast \in \overline{B_\rho(z)}$. In the first case evaluating at $z$ and sending $c\downarrow 0$ we get $u(r_\ell z)=0$. In the second case we notice that by \eqref{eq:Pluto} we must have $x_\ast\notin \de B_\rho(z)$, thus we get
$$
r_\ell^2\chi_{\{u(r_\ell \,\cdot\,)>0\}}(x_\ast) =\lap u_{r_\ell}(x_\ast)\leq \lap \phi_{z,\ell}(x_\ast)<r_\ell^2 \quad \Rightarrow\quad  x_\ast\notin {\{u(r_\ell\, \cdot\,)>0\}}.
$$
Then $0=u(r_\ell x_\ast)=\phi_{z,\ell}(x_\ast)+c_\ast\geq c_\ast>0$, a contradiction.

Next, we turn to the proof of (iii). First recall that there exists a dimensional modulus of continuity (see \cite[Theorem 8 and Corollary 11]{C98}) such that for all $x\in \Sigma(u)$ we have
$$\|u(x+\cdot)-p_{2,x}\|_{L^\infty(B_r)}\leq r^2\omega(r).$$
Suppose by contradiction that we can find $$y^*_\ell \in \left( \Sigma(u_{r_\ell})\cap B_{\rho/2N}(y_\infty)\right) \setminus \Gamma_\ell$$ for arbitrarily large $\ell$. We set for brevity $p_\ell:=p_{2,r_\ell y^*_\ell}$.  As $p_\ell$ is convex and $\lap p_\ell \geq 1$, for every such $\ell$ we choose a unit vector $e_\ell$ such that  
$$
 p_{\ell}(x)\geq \frac{1}{2n }(e_\ell \cdot x)^2.
$$
 Now by item (ii), we have for all $z\in B_{\rho/N}(y_\infty)\cap \Gamma_{\ell}$
\begin{align*}
    0=u(r_\ell z)&=u_{r_\ell}(y^*_\ell +(z-y^*_\ell))\\
    &\geq p_{\ell}(z-y^*_\ell)-|z-y^*_\ell|^{2}\omega(|z-y^*_\ell|)\\
    &\geq \frac{1}{2n} (e_\ell\cdot (z-y^*_\ell))^2-|z-y^*_\ell|^{2}\omega(|z-y^*_\ell|).
\end{align*}
From this inequality we reach a contradiction: on one hand we have $ |z-y^*_\ell|\leq 1/N$, so we can require $N$ so large that $\omega(|z-y^*_\ell|)\leq 1/100n$. On the other hand we can choose $z$ in such a way that the nonzero vector $z-y^*_\ell$ is almost aligned with $e_\ell$, thus we get a contradiction dividing by $|z-y^*_\ell|^2$.
  
In order to prove (i) we need to construct a barrier for all $z\in B_{\rho}(y_\infty)\cap\Gamma_\ell$. Set
\begin{equation*}
\phi_{z,\ell}(x):=\left(1-\frac{h_\ell}{r_\ell^2}\right) \frac 1 2 \mathcal{A}^2_k(r_\ell x) + \frac{h_\ell}{4n}  |x'-z'|^2.
\end{equation*}
  We have to check that $\phi_{z,\ell}$ indeed satisfies \eqref{eq:Pluto} for $\ell$ large. The first two equations in \eqref{eq:Pluto} are clearly fulfilled for any $\ell$. For the third condition, let us compute for $x\in B_\rho(z)$:
\begin{align*}
\lap \phi_{z,\ell}&=r_\ell^2 -h_\ell +C_2 r_{\ell}^{k+2} +\frac{2(n-1)}{4n} h_\ell \leq r_\ell^2 -\frac{1}{2}h_\ell +C_0 C_2 r_{\ell} h_\ell,
\end{align*}
where we used that, for some $C_2=C_2(n,k)$, we have $\lap  \frac 1 2 \mathcal{A}^2_k \leq 1+C_2|x|^{k}$ and $r_\ell^{k+1 }\leq C_0 h_\ell$. Hence, $\lap \phi_{z,\ell}<r_\ell^2$, as soon as $\ell$ is large enough.

We turn to the last condition of \eqref{eq:Pluto}. For any fixed $\eta\leq \frac{\rho^2}{100n}$, we have by uniform convergence \eqref{eq:uniformconvergence} for $\ell$ large 
$$
u_{r_\ell}\leq \frac{1}{2} \mathcal{A}^2_k(r_\ell \,\cdot \,) +h_\ell \overline q+h_\ell \eta 
$$
in $B_2$. As for some constant $C_3=C_3(n,k)>0$ we have in $B_2$
$$
\frac 1 2 \mathcal{A}^2_k(x)\leq \frac{1}{2} x_n^2+ C_3|x|^3,
$$
recalling the choice of $\rho$ form \eqref{eq:rhosmallness}, we get for $x\in \overline {B_\rho(z)}$:
\begin{align*}
u(r_\ell x) &\leq \frac{1}{2} \mathcal{A}^2_k(r_\ell x ) -h_\ell m|x_n| +h_\ell \eta \\
&\leq \left(1-\frac{h_\ell}{r_\ell^2}\right)\frac{1}{2} \mathcal{A}^2_k(r_\ell x )+\frac {h_\ell}{2} x_n^2 -h_\ell m|x_n| +C_3h_\ell r_\ell |x|^3+h_\ell \eta\\
&=\phi_{z,\ell}(x)+h_\ell\left(\frac {1}{2} x_n^2 - m|x_n| +C_3 r_\ell |x|^3+ \eta-\frac{1}{4n}|x'-z'|^2\right).
\end{align*}
We show that whenever $x\in \de B_\rho(z)$ the term between parentheses is negative. Using equation \eqref{eq:zsmall} and the fact that $|x|\leq 2$, we get for all $x\in \de B_\rho(z)$
\begin{align*}
&\frac {1}{2} x_n^2 - m|x_n| +C_3 r_\ell |x|^3+ \eta-\frac{1}{4n}|x'-z'|^2\\
&= \frac{1}{4n}|x_n-z_n|^2+\frac {1}{2} x_n^2 - m|x_n| +C_3 r_\ell |x|^3+ \eta-\frac{\rho^2}{4n}\\
&\leq x_n^2 - m|x_n| +C^2_1 r_\ell^{2\alpha_\circ} + 8C_3 r_\ell + \eta-\frac{\rho^2}{4n}.
\end{align*}
We claim that this quantity is negative as soon as $r^{\alpha_\circ}_\ell\leq \min\left\{\frac{\rho}{C_1}, \frac{\eta}{8C_3},\frac{\eta}{C^2_1},1\right\}.$ In fact by \eqref{eq:zsmall} we have uniformly in $z$:
$$
|x_n|\leq |x_n-z_n|+|z_n|\leq \rho+C_1 r_\ell^{\alpha_\circ} \leq 2\rho \leq \frac{2m}{100}
$$
thus $x_n^2 - m|x_n|\leq 0$ and
$$
C_1 r_\ell^2 + 8C_3 r_\ell + \eta-\frac{\rho^2}{4n}\leq 3\eta -\frac{\rho^2}{4n} \leq \frac{3\rho^2}{100n} -\frac{\rho^2}{4n} <0.
$$
This completes the proof.\end{proof}

Exploiting some recent volume estimates for the tubular neighbourhood of the critical set of harmonic functions from \cite{NV17}, we can now deduce the following.
\begin{lemma}
\label{lem:GMT2a}
Given $\beta_1>n-2$ and $k\geq 3$ odd, there exists an $\hat\eps=\hat \eps(n,\beta_1)$ small such that the following holds.
Let $E\subseteq \R^n$ be any set satisfying 
$$
E\subseteq B_r(x) \cap\big(\Sigma(q)+B_{\hat\eps r}(x)\big)
$$ for some $r\in(0,1)$, $x\in\R^n$ and $q\in\Signeven_k(L)\setminus\{0\}$ for some hyperplane $L$.
Then $E$
can be covered with $ \lfloor \gamma^{-\beta_1} \rfloor$ balls of radius $\gamma r$ centered at points of $E$, for $\gamma=\hat\eps/5$.
\end{lemma}
\begin{proof} By translation and scaling we can recover the general case from the case $r=1,x=0$. Let $\hat\eps\in (0,1)$ be a parameter to be fixed later and take $q$ as in the statement, recall that $q$ vanishes on $L$ (see Proposition \ref{pro:signorini}). For simplicity we assume $L=\{x_n=0\}$ and consider $Q$, the odd (with respect to $L$) extension of $q\vert_{\{x_n>0\}}$ to $\R^n$. $Q$ is harmonic and is easily checked that $\Sigma(q)\subseteq \{Q=|\nabla Q|=0\}=:\Sigma(Q)$, hence a fortiori
$$
E\subseteq B_1 \cap \{\dist(\,\cdot\,,\Sigma(Q))\leq \hat\eps\}.
$$ As $Q$ is harmonic and non-zero, we can apply the volume estimates in \cite[Theorem 1.1]{NV17} to find
\begin{equation*}
    \mathcal H^n \left(B_2 \cap \{\dist(\,\cdot\,,\Sigma(Q)) \leq t\} \right)\leq C(n) t^2,
\end{equation*}
for all $t\in (0,1)$. Now, consider a covering of $E$ of the form
$\{B_{\hat\eps}(x)\}_{x\in E}$. By Vitali's covering lemma, there exists a disjoint subcollection $\{B_{\hat\eps}(x_i)\}_{i\in I}$ such that 
$$
E\subseteq \bigcup_{x\in E} \overline {B_{\hat\eps}(x)}\subseteq \bigcup_{i\in I} B_{5\hat\eps}(x_i).
$$
We need to estimate the cardinality of $I$. Denoting by $\omega_n $ the volume of the unit ball in $\R^n$ and using that $B_{\hat \eps}(x_i)\subseteq ( E+B_{\hat\eps}\subseteq \Sigma(Q)+B_{2\hat\eps})\cap B_2$, we have 
\begin{align*}
     \omega_n \hat\eps^n\#I = \mathcal H^n\left(\bigcup_{i\in I} \overline {B_{\hat\eps}(x_i)}\right)\leq \mathcal H^n\left( \{ \dist(\,\cdot\,,\Sigma(Q))\leq 2\hat\eps\} \right) \leq  C(n) \hat\varepsilon^2
\end{align*}
thus, $\#I\leq C(n) \hat\eps^{2-n}$. As $\beta_1>n-2,$ choosing $\hat\eps (n,\beta_1)$ small enough, we find $\#I\leq (\hat\eps/5)^{-\beta_1},$ which finishes the proof.
\end{proof}
We employ this Lemma 
\ref{lem:GMT2a} to get a Reifenberg-type result. We need to incorporate the lower-semicontinuous function $\tau$ into the statement, as we will use this result in the next section.
\begin{proposition}[{\cite[Proposition 7.5]{FRS19}}]
\label{prop:GMT2a}
Let $\tau\colon E\to \R$ be a lower-semicontinuous function and $E\subseteq \R^n$ be a measurable set with the following property.
For any $\eps>0$ and $x\in E$, there exists $\varrho= \varrho(x,\eps)>0$ such that, for all $r\in (0,\varrho)$ there exist a hyperplane $L$, an odd integer $k\ge 3$ and $q\in\Sign^{even}_{k}(L)\setminus\{0\}$ such that:
$$
E\cap \overline{B_r(x)}\cap \tau^{-1}([\tau(x),+\infty))\subseteq \Sigma(q)+\overline {B_{\eps r}(x)}.
$$
Then ${\rm dim}_{\mathcal H}(E)\leq n-2$.
\end{proposition}
 \begin{proof}
The result follows by iteration of Lemma \ref{lem:GMT2a} and we skip the details. See for example the proof of Proposition 7.3 or 7.5 in \cite{FRS19}.
 \end{proof}
This finally gives the desired dimensional estimate.
\begin{corollary}\label{cor:dimensionk}
Let $k\geq 3$ be odd, then ${\rm dim}_{\mathcal H}(\Sigma^{kth}\setminus\Sigma^{>k})\leq n-2$. Furthermore, if $n=2$, then $\Sigma^{kth}\setminus\Sigma^{>k}$ is discrete in $\Sigma$.
\end{corollary}
\begin{proof}
Recall that if $n=2$, then $\Sigma(q)=\{0\}$ for every $q\in \Signeven_{k}\setminus\{0\}$. Pick any $x\in \Sigma^{kth}\setminus\Sigma^{>k}$ and apply Lemma \ref{lem:kodd} with $\eps:=\frac 1 2$. Then, for all $r<\varrho(x,1/2)$ we have
$$
\Sigma\cap \left( B_r(x)\setminus \overline {B_{r/2}(x)}\right)=\emptyset.
$$
This clearly implies that $\Sigma\cap B_{\varrho}(x)=\{x\}$, thus $x$ is isolated in $\Sigma$.

For the case $n\geq 3$ we apply Proposition \ref{prop:GMT2a} to $E:=\Sigma^{kth}\setminus\Sigma^{>k}$ and the constant function $\tau \equiv 1$ . The hypothesis are satisfied thanks to Lemma \ref{lem:kodd}.
\end{proof}

\subsection{The size of \texorpdfstring{$\bm{\Sigma^{>k}\setminus \Sigma^{\geq k+1}}{}$}{}}
\label{sec: intermediate set}
The key idea behind this dimensional reduction is that at an accumulation points of ${\Sigma^{>k}\setminus \Sigma^{\geq k+1}}$, the blow-up gain a translation symmetry along the direction of the approaching sequence. This observation corresponds to Lemmas 6.8 or 6.9 in \cite{FRS19}.
\begin{lemma}
\label{lem:non-integer case}
Let $0\in {\Sigma^{>k}\setminus \Sigma^{\geq k+1}}$ for some $k\ge 2$. Suppose there exists an infinitesimal sequence $r_\ell\downarrow 0$ and points $x_\ell\in \Sigma^{kth}\cap B_{r_\ell},x_\ell\neq 0$, such that $\lambda_{k}(x_\ell)\to \lambda_k$. Assume further that, as $\ell\to \infty$, we have 
\begin{enumerate}[label={\upshape(\roman*)}]
    \item $x_\ell/r_\ell \to y_\infty\in \overline {B_1}$.
    \item $\tilde v_{r_\ell}=\frac{(u-\curlyP_k)_{r_\ell}}{H(r_\ell,u-\curlyP_k)^{1/2}}\to q$ in $C^0_{\loc}(\R^n)$, for some $q\in\Sign_{\lambda_k}(\{p_{2}=0\})\setminus \{0\}$.
\end{enumerate}
Then $y_\infty\in \{p_2=0\}$ and $q=q(y_\infty + \,\cdot\,).$
\end{lemma}
\begin{proof} Consider a sequence $(x_\ell)_{\ell\in \N}\subseteq \Sigma^{kth}\cap B_{r_\ell}$ as in the statement of the lemma. We begin by recalling that $y_\infty\in \{p_2=0\}$ holds because $r_\ell^{-2}u(r_\ell\, \cdot\, )\to p_2$ uniformly in $B_2$. Then we apply Proposition \ref{pro:convergence} with varying centers $(x_\ell)_{\ell\in\N}$ and after passing to a subsequence (denoted again with $r_\ell$) we have
\begin{equation*}
   \widetilde v_{r_\ell,x_\ell}:=\frac{u(x_\ell+r_\ell\,\cdot\, )-\curlyP_{k,x_\ell}({r_\ell}\,\cdot\,)}{\|(u(x_\ell+\,\cdot\,)-\curlyP_{k,x_\ell})_{r_\ell}\|_{L^2(\de B_1)}} \to Q
\end{equation*} 
 in $C^0_{\loc}(\R^n)$, for some $Q\in\Sign_{\lambda_k}(\{p_{2}=0\})\setminus \{0\}$. On the other hand, by uniform convergence,
 \begin{align*}
    q(y_\infty + \,\cdot\,) =\lim_\ell \tilde v_{r_\ell}(x_\ell/r_\ell+\cdot)= \lim_{\ell} \frac{u(x_\ell+ r_\ell\,\cdot\,)-\curlyP_k(x_\ell+ r_\ell\,\cdot\,)}{H(r_\ell,u-\curlyP_k)^{1/2}}.
\end{align*}
So putting everything together we can write
\begin{equation}\label{eq:h}
    \tilde v_{r_\ell}(x_\ell/r_\ell+\cdot)=\tilde v_{r_\ell,x_\ell}\cdot I_\ell +J_\ell,
\end{equation}
where
$$I_\ell:=\frac{ H(r_\ell,u(x_\ell+\,\cdot\,)-\curlyP_{k,x_\ell}(\,\cdot\,))^{1/2}}{H(r_\ell,u-\curlyP_k)^{1/2}}=\frac{ \|u(x_\ell+r_\ell\,\cdot\,)-\curlyP_{k,x_\ell}(r_\ell\,\cdot\,)\|_{L^2(\de B_1)}}{ \|u(r_\ell\,\cdot\,)-\curlyP_k(r_\ell\,\cdot\,)\|_{L^2(\de B_1)}}$$
is a numerical sequence and 
$$J_\ell:= \frac{\curlyP_{k,x_\ell}(r_\ell\,\cdot\,)-\curlyP_k(x_\ell+ r_\ell\,\cdot\,)}{H(r_\ell,u-\curlyP_k)^{1/2}}$$ is a sequence of harmonic polynomials of degree at most $k+1$. Now two cases arise:
$$\text{either }\sup_\ell I_\ell <\infty,\ \text{ or }I_{\ell_m}\uparrow \infty\text{ for some subsequence }\ell_m\to \infty.$$
Let us begin with the first case. Up to a subsequence that we do not rename, we have $I_\ell\to \alpha$ for some $\alpha\geq 0$. Equation \eqref{eq:h} then implies that $J_\ell\to J$ locally uniformly to some harmonic polynomial $J$ of degree at most $k+1$. Thus, sending $\ell\uparrow \infty$ in \eqref{eq:h}, we obtain
\begin{equation}\label{eq:blowupdifferentpoints}
    q(y_\infty + \,\cdot\,) = \alpha Q + J.
\end{equation}
We exploit homogeneity: for large $R>0$ it holds
\begin{align*}
   R^{\lambda_k} q\left(\frac{y_\infty}{R} + \,\cdot\,\right) &= R^{\lambda_k}\alpha Q + J(R\,\cdot\,),
\end{align*}
so $\lim_{R\uparrow\infty}R^{-\lambda_k}J(Rx)$ exists for every $x\in \R^n$. As $\lambda_k$ is not an integer and $J$ is a polynomial, the only possibility is that $\lim_{R\uparrow\infty}R^{-\lambda_k}J(Rx)=0$ for all $x$ and so $\deg J \leq k$. Hence the last identity reads 
\begin{align*}
   q &= \alpha Q.
\end{align*} 
Inserting this back in \eqref{eq:blowupdifferentpoints} we find
$ q(y_\infty + \,\cdot\,) =q + J$. 
Now, using again that $q$ is homogeneous, for any $R>0$ we have
\[
R\left(q\left(\frac{y_\infty}{R} + \,\cdot\,\right)-q(\, \cdot\,)\right)=R^{1-\lambda_k}J(R\,\cdot\,).
\]
Sending $R\to \infty$ the left hand side converges to $y_\infty\cdot\nabla q$, but as before the right hand side can only converge to $0$, so $y_\infty\cdot \nabla q=0$.

The second case is simpler. We divide equation \eqref{eq:h} by $I_{\ell_m}$ and find, after passing to a subsequence of $\ell_m$, that
\begin{align*}
   0= Q + \tilde J
\end{align*}
for some harmonic polynomial $\tilde J$ of degree at most $k+1$. This is a contradiction, since $Q\neq 0$ is a $\lambda_k\in(k,k+1)$ homogeneous function, hence not a polynomial. This finishes the proof.
\end{proof}
Lemma \ref{lem:non-integer case} triggers a Federer-type dimension reduction, exactly as in \cite{FS19}.
\begin{proposition}[{\cite[Proposition 7.3]{FRS19}}]\label{prop:GMT1}
Let $E\subseteq \R^n$, $f\colon E\to \R$ and $m\in\{1,\ldots,n\}$.
Assume that, for any $\eps>0$ and $x\in E$ there exists $\varrho= \varrho(x,\eps)>0$ such that, for all $r\in (0, \varrho)$, we have
$$
E\cap \overline{B_r(x)}\cap f^{-1}([f(x)-\varrho,f(x)+\varrho])\subseteq \Pi_{x,r}+ B_{\eps r},
$$
for some $m$-dimensional plane $\Pi_{x,r}$ passing through $x$ (possibly depending on $r$).
Then ${\rm dim}_{\mathcal H}(E)\leq m$.
\end{proposition}
We combine these Lemma \ref{lem:non-integer case} and Proposition \ref{prop:GMT1} to prove the dimensional estimate.
\begin{proposition}
\label{pro::dim red non integer}
For every $k\ge 2$ there holds $\dim_{\mathcal H}({\Sigma^{>k}\setminus \Sigma^{\geq k+1}})\leq n-2$. Moreover, if $n=2,$ then ${\Sigma^{>k}\setminus \Sigma^{\geq k+1}}$ consists of isolated points if $k$ is odd and is empty if $k$ is even.
\end{proposition}
\begin{proof} We want to apply Proposition \ref{prop:GMT1} with $E:={\Sigma^{>k}\setminus \Sigma^{\geq k+1}}$, $m=n-2$ and the function $f$ given on $E$ by $x\mapsto\lambda_k(x) \in (k,k+1).$ It suffices to show that for all $x\in E$ and for all $\eps >0$ there exist $\varrho= \varrho(x,\eps) >0$ and a $(n-2)$-dimensional plane $\Pi_{x,r}$ passing through $x$, such that
$$
E \cap B_r(x_\circ)\cap  \lambda_k^{-1}([\lambda_k(x)-\varrho,\lambda_k(x)+\varrho]) \subseteq \{x  :  \dist(x, \Pi_{x,r}) \le \eps r\} \quad \forall \,r\in (0, \varrho).
$$

We argue by contradiction. Assume that for $x=0$ and some $\eps_\circ>0$ the above does not hold, then we make the following simple geometric claim. For each $\ell$ there exists $r_\ell\in (0,2^{-\ell})$ and $n-1$ points $x_\ell^{(1)},\ldots,x_\ell^{(n-1)}$ in $E\cap B_{r_\ell}$ such that: 
$$
\left|x_\ell^{(1)}\wedge\ldots\wedge x_\ell^{(n-1)}\right|\geq \delta r_\ell^{n-1},\quad |\lambda_k(x_\ell^{(j)})-\lambda_k|\leq 2^{-\ell},
$$
for all $j\in\{1,\ldots, n-1\}$ and for some $\delta=\delta(n,\eps_\circ)\in (0,1)$. In particular $\{x_\ell^{(1)},\ldots,x_\ell^{(n-1)}\}$ span an hyperplane and for each fixed $j$ the sequence $(x_\ell^{(j)})_{\ell\in \N}$ lies in $E\subseteq \Sigma^{kth}$, with
$$\lambda_k(x_\ell^{(j)})\to \lambda_k.$$
We extract a finite number of subsequences to ensure ${x_\ell^{(j)}/r_\ell}\to y^{(j)}_\infty$ for each $j$. Exploiting the lower bound on the exterior product we again have that $\dim \Span \{y^{(1)}_\infty,\ldots,y^{(n-1)}_\infty\}=n-1$. 
Now we apply Proposition \ref{pro:convergence} to each $(x^{(j)}_\ell)_{\ell\in\N}$ and get
$$
\tilde v_{r_\ell}=\frac{(u-\curlyP_k)_{r_\ell}}{H(r_\ell,u-\curlyP_k)^{1/2}} \to q, \quad \text{ in }C^0_{\loc}
$$
for some $q\in\Sign_{\lambda_k}$. Notice that, taking each time a subsequence, $q$ can be taken the same for all $j$'s. By Lemma \ref{lem:non-integer case}, we conclude that $q$ is translation-invariant in the directions $y^{(j)}_\infty$ for all $1\leq j\leq n-1$, hence $q$ is a $1$-dimensional homogeneous solution to the obstacle problem vanishing at the origin. Thus after a rotation of coordinates it must be $q(x)=-A|x_n|+Bx_n$ for some constants $A\geq 0$ and $B\in \R$, which contradicts $\lambda_k>1.$ 

Let us sketch the geometric argument needed to construct such $\{x_\ell^{(1)},\ldots,x_\ell^{(n-1)}\}$. Fixed $\ell$, we pick any $(n-2)$-plane $\Pi_0$ and any $x^{(1)}\in (E\cap B_{r_\ell})\setminus( \Pi_0 +B_{\eps_\circ})$. Then we choose any plane $\Pi_1$ containing $x^{(1)}$ and any $x^{(2)}$ in $(E\cap B_{r_\ell})\setminus( \Pi_1 +B_{\eps_\circ})$. We can go on in this way and construct the whole set $\{x^{(1)},\ldots,x^{(n-1)}\}$. Finally, we notice that by compactness
$$
\delta:=\min\Big\{|z_1\wedge\ldots \wedge z_{n-1}| : z_j\in \R^n,\forall j\,\dist\left(z_j,\Span\{z_i:i\neq j,1\leq i\leq n-1\}\right)\geq\eps_\circ\Big\}>0.
$$
We are left with the case $n=2.$ Recall that if $q$ is a $\lambda$-homogeneous solution to the thin obstacle problem, then $\lambda \in \N_+ \cup \{2m-\frac1 2 : m\in \N_+\}$ (see Proposition \ref{pro:signorini}). Thus, having in mind Proposition \ref{pro:convergence}, we find that ${\Sigma^{>k}\setminus \Sigma^{\geq k+1}}$ is empty for $k$ even. If $k$ is odd, we find $\lambda_k(x_\circ)=k+\frac 1 2$ for every $x_\circ\in{\Sigma^{>k}\setminus \Sigma^{\geq k+1}}$. Was this set not discrete, we could apply Lemma \ref{lem:non-integer case} and reach a contradiction, obtaining a one-dimensional and $(k+\frac 1 2)$-homogeneous solution of the thin obstacle problem.
\end{proof}

\subsection{The size of \texorpdfstring{$\bm{\Sigma^{\geq k+1}\setminus \Sigma^{(k+1)th}}{}$}{}}
\label{sec: end point}
\begin{lemma}\label{lem:k+1} Let $k \geq 2$, $x\in\Sigma^{\geq k+1}\setminus \Sigma^{(k+1)th}$ and $\eps>0$. Then exists $\varrho=\varrho(\eps,x)>0$ such that for each $r\in (0,\varrho)$ there is $q\in\Signeven_{k+1}(\{p_{2,x}=0\})\setminus\{0\}$ such that
\begin{equation}\label{eq:barriersimple}
\Sigma(u)\cap B_{r}(x)\subseteq \Sigma(q)+B_{\eps r}(x).
\end{equation}
Recall that $\Sigma(q)=\{q=|\nabla q|=0\}\cap\{p_{2,x}=0\}$ was defined in \eqref{eq:definitionsingularsetsignorini}.
\end{lemma}
\begin{proof}[Proof of the case $(k+1)$ even]
Up to an isometry we can assume $x=0$ and $ p_{2}=\frac 1 2 x_n^2$. We argue by contradiction and rescale everything: we find $\eps_\circ>0$ and a sequence $r_\ell\downarrow 0$ such that $$y_\ell \in \Sigma(u(r_\ell\, \cdot\,))\cap B_1 \ \ \text{ and }\ \ \dist(y_\ell,\Sigma(q'))\geq \eps_\circ\  $$
for all $ q'\in\Signeven_{k+1}(\{x_n=0\})$. Up to taking subsequences we can assume $y_\ell\to y_\infty\in \{x_n=0\}$, and by Proposition \ref{pro:convergence}
$$
\frac{(u-\curlyP_k)(r_\ell\, \cdot\,)}{r_\ell^{k+1}}\to \overline q\in \Sign_{k+1}(\{x_n=0\})
$$ holds in $C^0_{\loc}(\R^n)$. Since $0\notin \Sigma^{(k+1)th}$ we have $\overline q^{\text{even}}\neq 0$. Rearranging the terms we can equivalently write
\begin{equation*}
    w_\ell:=\frac{(u-\frac 1 2 \curlyA_{k+1}^2(p_{2,0},\ldots,p_{k,0},\overline q^{\text{odd}}))(r_\ell\, \cdot\, )}{r_\ell^{k+1}}\to \overline q^{\text{even}},\quad  \text{ in }C^0_{\loc}(\R^n).
\end{equation*}
Now recall that, $k+1$ being even, we have $\Sigma(\overline q^{\text{even}})=\{\overline q^{\text{even}}=0\}$, thus $\eta:=\overline q^{\text{even}}(y_\infty)>0$ as $y_\infty$ lies on the thin obstacle. So we can find a small radius $\delta>0$ and a large $\ell_0$ such that for all $\ell>\ell_0$ we have
$$
\inf_{B_\delta(y_\infty)} w_\ell\geq \inf_{B_\delta(y_\infty)} \overline q^{\text{even}} -\|w_\ell -\overline q^{\text{even}}\|_{L^\infty(B_2)}\geq \frac{\eta}{2}-\frac{\eta}{4}=\frac{\eta}{4}.
$$
However, eventually we will have $y_\ell \in B_\delta(y_\infty)$ and this is a contradiction as
$$
0\geq -\frac{\curlyA_{k+1}^2(r_\ell y_\ell)}{2 r_\ell^{k+1}} =w_\ell(y_\ell)\geq \inf_{B_\delta(y_\infty)} w_\ell\geq \eta/4.
$$
We remark that we only used that $y_\ell \in \{u_{r_\ell}=0\}$, not that $y_\ell$ were singular points.
\end{proof}
\begin{proof}[Proof of the case $(k+1)$ odd.]
  Arguing by contradiction as in the even case we find $\eps_\circ>0$ and a sequence $r_\ell\downarrow 0$ such that for each $\ell$
  $$y_\ell \in \Sigma(u(r_\ell\, \cdot\,))\cap B_1 \ \ \text{ and }\ \ \dist(y_\ell,\Sigma(q))\geq \eps_\circ,\ \ \forall\, q\in\Signeven_{k+1}(\{x_n=0\})\setminus\{0\}.$$
We can also assume that $y_\ell\to y_\infty\in \{x_n=0\}$ and
\begin{equation*}
    w_\ell:=\frac{(u-\frac 1 2 \curlyA_k^2(p_{2,0},\ldots,p_{k,0},\overline q^{\text{odd}}))(r_\ell\, \cdot\, )}{r_\ell^{k+1}}\to \overline q^{\text{even}},\quad  \text{ in }C^0_{\loc}(\R^n).
\end{equation*}
  From this point the proof is conducted analogously to the proof of Lemma \ref{lem:kodd}, it suffices to replace $k$ with $k+1$ and $h_\ell$ with $r_\ell^{k+1}$. 
\end{proof}
\begin{remark}\label{rem:usefulremark}
In the case $k+1$ even the proof actually gives a stronger result: as we only used that $y_\ell\in\{u=0\}$, we can replace $\Sigma(u)$ with the full contact set. In other words we can replace equation \eqref{eq:barriersimple} with
\begin{equation}
\{u=0\}\cap B_{r}(x)\subseteq \Sigma(q)+B_{\eps r}(x).
\end{equation}
\end{remark}
We can conclude now exactly as in the case $\lambda_k=k$ for $k$ odd.
\begin{corollary}\label{cor:k+1}
Let $k\geq 2$, then ${ \dim}_\mathcal{H}(\Sigma^{\geq k+1}\setminus \Sigma^{(k+1)th})\leq n-2$. Furthermore if $n=2$ then $\Sigma^{\geq k+1}\setminus \Sigma^{(k+1)th}$ is discrete in the full $\Sigma$.
\end{corollary}
\begin{proof}
Recall that if $n=2$ then $\Sigma(q)=\{0\}$, for every $q\in \Signeven_{k+1}$. Pick $x\in \Sigma^{\geq k+1}\setminus \Sigma^{(k+1)th}$ and apply Lemma \ref{lem:k+1} with $\eps:=\frac 1 2$, this gives that for all $r<\varrho(x,1/2)$ we have
$$
\Sigma(u)\cap \left( B_r(x)\setminus \overline {B_{r/2}(x)}\right)=\emptyset.
$$
This clearly gives that $\Sigma(u)\cap B_{\varrho}(x)=\{x\}$, thus $x$ is isolated in $\Sigma(u)$.

For $n\geq 3$ we argue like in Corollary \ref{cor:dimensionk}, namely we apply Proposition \ref{prop:GMT2a} to $E:=\Sigma^{\geq k}\setminus \Sigma^{(k+1)th}$, the assumptions are satisfied thanks to Lemma \ref{lem:k+1}. 
\end{proof}
In particular we notice that in dimension $n=2$ this forces the sets $\Sigma^{kth}$ to be closed:
\begin{corollary}\label{cor:sigmakthclosed}
 If $n=2$ then the sets $\Sigma^{kth}$ are closed for all $k\geq 2$.
\end{corollary}
\begin{proof}
  We prove the assertion by induction on $k$. Let $(x_\ell)_{\ell\in\N}\subseteq \Sigma^{(k+1)th}\setminus \{0\}$ be a sequence with $x_\ell \to 0$, in particular we have $x_\ell \in \Sigma^{\geq k+1}$. By inductive assumption we can assume $0\in\Sigma^{kth}$ and by the upper semicontinuity of the truncated frequency we get $\lambda_k(0)\geq k+1$ (see Remark \ref{rmk:uscfrequency}). But by Corollary \ref{cor:k+1} the origin cannot lie in $\Sigma^{\geq k+1}\setminus \Sigma^{(k+1)th}$, because is an accumulation point of the sequence of singular points $x_\ell$. Since $\Sigma^{2nd}=\Sigma_{n-1}$ is closed, by lower semicontinuity of the rank, the proof is finished.
\end{proof}
\subsection{The geometry of \texorpdfstring{$\bm{\Sigma^{\infty}}{}$}{}}
\label{sec: geometry}
Let us put together the results obtained so far. By definition,
\begin{equation*}
    \Sigma^\infty:=\bigcap_{k\geq 2} \Sigma^{kth}.
\end{equation*}
In the last three subsections we proved the following.
\begin{proposition}\label{pro:dimensionreductionfinalstatement}
$\dim_{\mathcal H}\left(\Sigma\setminus\Sigma^\infty\right)\leq n-2$, if $n=2$ then $\Sigma\setminus\Sigma^\infty$ is countable.
\end{proposition}
\begin{proof}
   By definition have that
   $$
   \Sigma\setminus \Sigma^\infty =\left(\Sigma \setminus \Sigma_{n-1}\right) \cup \bigcup_{j\geq 2} \left(\Sigma^{jth}\setminus\Sigma^{>j}\right) \cup \bigcup_{j\geq 2}\left( \Sigma^{>j}\setminus \Sigma^{\geq j+1} \right)\cup \bigcup_{j\geq 3} \left(\Sigma^{\geq j}\setminus\Sigma^{jth}\right).
   $$
   But now 
   \begin{itemize}
       \item $\dim_{\mathcal{H}}\left(\Sigma \setminus \Sigma_{n-1} \right)\leq n-2$ (discrete, if $n=2$), by \cite[Theorem 8]{C98};
       \item $\dim_{\mathcal{H}}\left(\Sigma^{jth}\setminus\Sigma^{>j}\right)\leq n-2$ (discrete, if $n=2$), by Corollaries \ref{cor:k even impossible} and \ref{cor:dimensionk};
       \item $\dim_{\mathcal{H}}\left(\Sigma^{jth}\setminus\Sigma^{>j}\right)\leq n-2$ (discrete, if $n=2$), by Proposition \ref{pro::dim red non integer};
       \item $\dim_{\mathcal{H}}\left(\Sigma^{\geq j}\setminus\Sigma^{jth}\right)\leq n-2$ (discrete, if $n=2$), by Corollary \ref{cor:k+1}.\qedhere
   \end{itemize}
\end{proof}
At each point of $\Sigma^\infty$ we have Taylor polynomials of every order, and they vary smoothly in the sense of Whitney. This also gives that $\Sigma^\infty$ locally is contained in a smooth hypersurface. Let us first phrase a suitable statement 
\begin{theorem}\label{thm:whitney}
Let $E\subseteq \R^n$ be a any set and for each $k\in\N$ consider a collection of polynomials $\{P_{k,x}\}_{x\in E}$ of degree at most $k$. Suppose that these polynomials satisfy 
\begin{enumerate}[label={\upshape(\roman*)}]
    \item $P_{k,x}=\pi_{\leq k}\left(P_{k+\ell,x}\right)$ for all $k,\ell\in\N$ and $x\in E$;
    \item for each $k\in\N$ there is a constant $C(k)$ such that for each multi index $\alpha,|\alpha|\leq k$ we have
    \begin{equation*}
    \left|\de^\alpha P_{k,x}(0)-\de^\alpha P_{k,y}(x-y)\right|\leq C(k) |x-y|^{k-|\alpha|+1}\qquad\text{ for all }x,y\in E.
\end{equation*}
\end{enumerate}
Then there exist a function $F\in C^\infty(\R^n)$ such that for each $x\in E$ and $k\in\N$ there holds
$$
F(x+h)=P_{k,x}(h) +O(|h|^{k+1}) \qquad\text{ as } |h|\to 0.
$$
\end{theorem}
\begin{proof}
This is just a restatement of the of Whitney's Extension Theorem for smooth functions. The interested reader can find in Appendix \ref{sec:auxiliarylemmas} how to derive this formulation from the original one, namely \cite[Theorem I]{W34}.
\end{proof}
\begin{lemma}
\label{lem:Whitney}
Let $u$ be a solution to the obstacle problem \eqref{eq:obstacle}, then $\Sigma^{\infty}$ is closed and locally covered by one smooth manifold of dimension $n-1$.
\end{lemma}
\begin{proof}
  The main idea is to combine the implicit function Theorem and Whitney's Extension Theorem \ref{thm:whitney}. We will first prove the covering and then the closeness.
  
  As the statement is local we can assume  $0\in \Sigma^{\infty}$ and that $u$ solves \eqref{eq:obstacle} in  $B_2(0)\subseteq \R^n$. We want to apply Whitney's Extension Theorem \ref{thm:whitney} with $E:=\Sigma^\infty\cap B_1$ and the polynomials
  $$
  P_{k,x}:=\pi_{\leq k}\left(\curlyP_{k,x}\right)\quad \text{ for all }x\in\Sigma^\infty\cap B_1,k \geq 0. 
  $$
  Assumption (i) is holds because $\curlyP_{k+\ell}$ and $\curlyP_{k}$ agree up to order $k$ (see also Lemma \ref{lem:sigmakthsimpledef}). We need to show that (ii) holds, it is not restrictive to do it only for some fixed $k\geq 3$. To do so we exploit our previous analysis on $\Sigma^{(k+1)th}$. More precisely, combining Lemma \ref{lem:L2toLinfty} with the uniform estimate in Proposition \ref{pro:continuity} and growth estimates form Proposition \ref{pro:almost monotonicity} and Lemma \ref{lem:H ratio}, we find $R=R(n,k)$ and $C=C(n,k)$, such that for all $x\in  \Sigma^{(k+1)th}\cap B_1$ and $0<r<R<\frac 1 2$ it holds
  \begin{equation}
      \label{eq:123}
      \| u (x+ \,\cdot\,) - P_{k,x} \|_{L^\infty(B_r(0))} \leq C r^{k+1}.
  \end{equation}
  Thus this must hold, a fortiori, for all $x\in \Sigma^\infty\cap B_1$.
  Let now $x_1,x_2\in \Sigma^{\infty}\cap B_1$ such that $|x_1-x_2|\leq R/10$. Then, since $B_{2|x_1-x_2|}(x_1) \subseteq B_{4|x_1-x_2|}(x_2),$ by \eqref{eq:123} applied at $x_1$ with $r_1=2|x_1-x_2|$ and $x_2$ with with $r_2=4|x_1-x_2|$, together with the triangle inequality, we find
    \begin{equation}
      \label{eq:1234}
      \| P_{k,x_1}(\,\cdot\,-x_1)-P_{k,x_2}(\,\cdot\,-x_2) \|_{L^\infty(B_{2|x_1-x_2|}(x_1))} \leq C |x_1-x_2|^{k+1}.
  \end{equation} If we consider the polynomial $Q= P_{k,x_1}(\,\cdot\,-x_1)-P_{k,x_2}(\,\cdot\,-x_2)$ equation \eqref{eq:1234} reads
  $$
  \|Q(x_1+2|x_1-x_2|\,\cdot\,)\|_{L^\infty (B_1)}\leq C |x_1-x_2|^{k+1},
  $$ hence by the equivalence of norms on the space of polynomials of degree bounded by $k$, we conclude
   $$
  \|(\de^\alpha Q)(x_1+2|x_1-x_2|\,\cdot\,)\|_{L^\infty (B_1)}\leq  C |x_1-x_2|^{k+1-|\alpha|}
  $$ for all multi-index $|\alpha|\leq k$, with some $ C = C(n,k).$
  In particular, looking at the center of $B_1$, we get
  \begin{equation}
      \label{eq:12345}
      |\de^\alpha P_{k,x_1}(x_2-x_1)-\de^\alpha P_{k,x_2}(0)|\leq C(n,k) |x_1-x_2|^{k+1-l},
  \end{equation}
  and this proves that assumption (ii) holds. By Whitney extension theorem there exists a $C^\infty$ function $F\colon \R^n\to \R$ for which \begin{equation*}
      F(x)=P_{k,x_\circ}(x-x_\circ)+O(|x-x_\circ|^{k+1})
  \end{equation*} at every $x_\circ\in\Sigma^{\infty} \cap \overline{B_1}$. 
  
  We now conclude as in \cite[Proposition 8.1 c)]{FRS19}, using the implicit function theorem. As $\Sigma^{\infty} \subseteq \{\nabla F=0\}$ and $\nabla^2F(x_1)=\nabla^2p_{2,x_1}(0)$ has rank $1$, we conclude with the implicit function theorem that $\{\der F=0\}$ is a smooth hypersurface in some neighbourhood of $x_1$.
  
  Let us now prove that $\Sigma^\infty$ is closed. Suppose $0\in\overline{\Sigma^\infty}$ and so there is $x_\ell\in\Sigma^{\infty}$ such that $x_\ell\to 0$. First observe that $0\in \Sigma$ as the full singular set $\Sigma$ is closed, hence $p_2$ exists. We define by continuity $\curlyP_k:=\lim_\ell \curlyP_{k,x_\ell}$ for all $k\geq 3$; this is a well-posed definition as by the previous Lemma \ref{lem:Whitney} the map $x\mapsto \curlyP_{k,x}$ is Lipschitz on $\Sigma^\infty$, hence admits a canonical extension to the closure.
 Now, by Proposition \ref{pro:monneau}, for some constants $C,\eps>0$ and a radius $r_0$, both independent from $\ell$, it holds for all $r\in (0,r_0)$
 $$
 \frac{d}{dr}\left(r^{-2k}H(r,u(x_\ell+\,\cdot\,)-\curlyP_{k,x_\ell})\right)\geq -Cr^{\eps- 1}\text{ and }r^{-2k}H(r,u(x_\ell+\,\cdot\,)-\curlyP_{k,x_\ell})\leq C.
 $$
 We can pass both these inequalities to the limit $\ell\to \infty$ and apply the same reasoning to Proposition \ref{pro:almost monotonicity} to show that the sequence $r^{-k}(u-\curlyP)_r$ is uniformly bounded in $W^{1,2}_{\loc}(\R^n)$. Hence it is immediate that $0\in\widetilde{\Sigma}^{(k-1)th}$ (see Lemma \ref{lem:sigmakthsimpledef}), as $k$ was arbitrary we conclude $0\in\Sigma^\infty$.
\end{proof}
We conclude this section proving Theorem \ref{thm:maintheoremintro}.
\begin{proof}[Proof of Theorem \ref{thm:maintheoremintro}]
  The bound on the dimension of $\Sigma \setminus \Sigma^\infty$ is given by Proposition \ref{pro:dimensionreductionfinalstatement}, the covering by Lemma \ref{lem:Whitney}. The polynomial expansion has been shown in equation \eqref{eq:123} above, we simply take $P_{k,x_\circ}:=\pi_{\leq k}(\curlyP_{k,x_\circ})$. Although we often assumed $f\equiv 1$ and $\mu=1$ to simplify the notation, essentially no modifications are needed for a general $f$. The reader can find a complete account of the modifications needed in the statements and in the proofs in Appendix \ref{sec:rhs f}.
\end{proof}
In the following section we aim to explain in which sense the set $\Sigma^\infty$ is unstable and disappears after a slight perturbation of the boundary data in the obstacle problem. 
\section{Extension to a monotone family of solutions}
\label{sec:family}
	In this Section we aim to prove Theorem \ref{thm:instabilityintro} and Corollary \ref{thm:Hele-Shawintro}. For simplicity we take $f\equiv 1$. This allows us to use verbatim some lemmas from \cite{FRS19} and shortens the notation, without affecting the proofs. We list the changes needed in Appendix \ref{sec:rhs f}.
	
	We remark that in Sections 7.1 and 7.2 we only assume to have a monotone family of solutions, while in Section 7.3 we work under the ``uniform monotonicity'' assumption \eqref{uniform monotonicity}. 
\subsection{Setup and strategy}\label{subsec:setup}
 For the rest of the section, we let $u: \overline{B_1}\times [-1,1]  \rightarrow \R$, $u\ge 0$, be a monotone  1-parameter family of  solutions of the obstacle problem, namely
\begin{equation}\label{eq:UELL+t}
\begin{cases} 
\lap u(\,\cdot\,,t) = \chi_{\{u(\,\cdot\,,t)>0\}}\\
0\leq u(\,\cdot\,,s)  \le u(\,\cdot\,, t) \quad  \mbox{in } B_1, \text{ for } -1\le s\leq t\le1.
\end{cases}
\end{equation}
We will also use the notation $u^t:=u(\, \cdot\, ,t)$. We will assume in addition that $u\in C^0(\overline{B_1}\times [-1,1] )$. We remark that this continuity property in $t$ follows by the maximum principle, whenever $u|_{\partial {B_1}\times [-1,1]}$ is continuous. 

We will often think of $t$ as time parameter, as intuitively we are imaging to lift the boundary datum of a solution of \eqref{eq:obstacle}. However, we remark that this is only a formal analogy: no equation in $t$ is given.

For each fixed $t$, we can apply the results of the previous sections, so we introduce further notation for the following subsets of $\overline{B_1}\times [-1,1]$:
\begin{equation} \label{eq:Sigma t}
\begin{array}{rcl}
\bm{\Sigma} &:=& \{ (x_\circ, t_\circ): x_\circ \in \Sigma(u(\,\cdot\,,t_\circ)) \}, \vspace{1mm}\\
\bm{\Sigma}_{n-1} &:=&  \big\{ (x_\circ, t_\circ): x_\circ\in {\Sigma}_{n-1}(u(\,\cdot\,,t_\circ))\big\}, \vspace{1mm}\\
\bm{\Sigma}^{kth} &:=& \big\{ (x_\circ, t_\circ): x_\circ\in {\Sigma}^{kth}( u(\,\cdot\,,t_\circ))\big\},\, k\ge 2, \vspace{1mm}\\
\bm{\Sigma}^{>k} &:=& \big\{ (x_\circ, t_\circ): x_\circ\in {\Sigma}^{>k}(u(\,\cdot\,,t_\circ))\big\},\, k\ge 2, \vspace{1mm}\\
\bm{\Sigma}^{\ge k+1} &:=& \big\{ (x_\circ, t_\circ): x_\circ\in {\Sigma}^{\ge k+1}(u(\,\cdot\,,t_\circ))\big\},\, k\ge 2, \vspace{1mm}\\
\bm{\Sigma}^{\infty} &:=& \big\{ (x_\circ, t_\circ): x_\circ\in {\Sigma}^{\infty}(u(\,\cdot\,,t_\circ))\big\}.
\end{array}
\end{equation}
This setup (up to $k=4$) has already been considered in \cite{FRS19}. As we use the same notation, we begin recalling two important lemmas from \cite{FRS19} about the set $\bm \Sigma$.
\begin{lemma}[{\cite[Lemma 6.2]{FRS19}}]\label{lem:EG2B1bis}
Let $u\in C^0\big(\overline{B_1}\times [-1,1]  \big)$ solve \eqref{eq:UELL+t}, then:
\begin{enumerate}[label={\upshape(\roman*)}]
\item $\bm{\Sigma}\cap \overline B_\varrho\times [-1,1]$ is closed for any $\varrho<1$ and
\[ \mathbf{\Sigma} \cap \overline B_\varrho\times [-1,1]\ni(x_k,t_k)\to (x_\infty, t_\infty) \quad \Rightarrow \quad p_{2,x_k,t_k} \to p_{2,x_\infty, t_\infty}.\]

\item If $(x_\circ, t_1)$ and $(x_\circ, t_2)$ both belong to $\mathbf{\Sigma}$ and $t_1< t_2$, then there exists $r>0$ such that $u(x,t)$ is independent from $t$ for all $(x,t)\in B_r(x_\circ)\times [t_1, t_2]$.
\end{enumerate}
\end{lemma}
The next result concerns the quantitative behaviour of the first blow-up $p_{2,k}:= p_{2,x_k,t_k}$ with respect to the convergence $x_k\to 0$ (here it is assumed $(x_k,t_k) \in \mathbf{\Sigma}$ for some sequence of times).

\begin{lemma}[{\cite[Lemma 6.3]{FRS19}}]\label{lem:EG2B1bisbis}
Let $u\in C^0\big(\overline{B_1}\times [-1,1]  \big)$ solve \eqref{eq:UELL+t},  let $(x_k,t_k) \in \mathbf{\Sigma}$, $(0,0)\in \mathbf{\Sigma}$, and assume that $x_k\to 0$.
If we set $p_{2}:= p_{2,0,0}$, then we have
$$ \left\|p_{2,k} -p_2\left(\frac{x_k}{|x_k|}+\,\cdot\,\right) \right\|_{L^\infty(B_1)} \le C \omega(2|x_k|) \quad \mbox{and}\quad \|p_{2,k} -p_2 \|_{L^\infty(B_1)} \le  C \omega(2|x_k|),
$$
for some dimensional modulus of continuity $\omega$. In addition,
$${\dist}\left ( \frac{x_k}{|x_k|} , \{p_2=0\}\right ) \to 0 \qquad \text{ as } k\to \infty.
$$
\end{lemma}
Our strategy follows the one exhibited in \cite{FRS19}. For each $(x_\circ,t_\circ)\in\mathbf{\Sigma}^{\ge k+1}$, we first prove the approximation 
\begin{equation}
    \label{eq:approx t1}
    \|u(x_\circ + \,\cdot\,,t_\circ)-P_{k,x_\circ,t_\circ}\|_{L^\infty(B_r)}\leq Cr^{k+1},
\end{equation} where the polynomial $P_{k,x_\circ,t_\circ}$ of degree at most $k$ is unique and $\lap P_{k,x_\circ,t_\circ} =1$. Using the fact that the polynomials above are almost positive, together with barrier-type arguments (cf.\ Lemma \ref{lem:cleaning}), we are able to conclude a ``cleaning property'' in space-time in the following sense. For each $(x_\circ,t_\circ)\in \mathbf\Sigma^{\ge k+1}$ there exist $\varrho>0$ and $C>0$ (depending on $n,k,x_\circ$) such that 
\begin{equation}\label{eq:cleaningproperty}
\big\{ (x,t)\in B_\varrho(x_\circ)\times (t_\circ,1) \ :\   t-t_\circ> C|x-x_\circ|^{k} \big\} \cap \{u=0\}  = \emptyset.    
\end{equation}
This property express the instability of $\Sigma^{\geq k+1}(u^t)$ with respect to increments of the $t$ parameter. From here, we will conclude with the next Geometric Measure Theory result, that the set $\pi_t(\mathbf{\Sigma}^{\infty})$ has zero Hausdorff dimension.
\begin{proposition}[{\cite[Corollary 7.8]{FRS19}}]
\label{prop:GMT4b}
Let $E\subseteq \R^n\times [-1,1]$, let $(x,t)$ denote a point in  $\R^n\times [-1,1]$, and let $\pi_x:(x,t)\mapsto x$ and $\pi_t:(x,t)\mapsto t$ be the standard projections.
Assume that, for some  $\beta\in (0,n]$ and $s>\beta,$ we have:
\begin{itemize}
\item  ${\rm dim}_{\mathcal H}\big(\pi_x(E)\big) \le \beta$;

\item For all $(x_\circ,t_\circ) \in E$ and $\eps>0$, there exists  $\varrho= \varrho_{x_\circ, t_\circ, \eps}>0$   such that
\[
\big\{ (x,t)\in B_\varrho(x_\circ)\times[-1,1] \ :\   t-t_\circ>   |x-x_\circ|^{s-\eps} \big\}\cap E = \emptyset.
\]
\end{itemize}
Then ${\rm dim}_{\mathcal H} \big( \pi_t(E) \big) \le \beta/s$.

\end{proposition}
To obtain Theorem \ref{thm:Hele-Shawintro} we also need to take care of the points where the expansion \eqref{eq:approx t1} fails for some $k$. To this end we need to generalize to one-parameter solutions some of the previous results.
\subsection{Adaptation of previous sections to family of solutions}\label{subsect:dimredt}
In this section we establish an analogous of Theorem \ref{thm:maintheoremintro}, for monotone families of solutions. The generalization of the polynomial expansion is obvious in the set $\pi_x (\bm\Sigma^\infty)$, so the only nontrivial task is to show that $\pi_x(\bm\Sigma\setminus\bm\Sigma^\infty)$ has again Hausdorff dimension at most $n-2$. We will show this by repeating the arguments of Section \ref{sec:dimensionreduction}. While the arguments of Sections 6.1 and 6.3 adapt immediately exploiting monotonicity, the arguments of Section \ref{sec: intermediate set} require a bit more care. Specifically, we have to check Lemma \ref{lem:non-integer case} for varying times. 

We start observing that in Proposition \ref{pro:convergence} one can also consider the varying time parameter.
\begin{proposition}
\label{pro:convergence t}
Let $u\in C^0(\overline{B_1} \times [-1,1])$ solve \eqref{eq:UELL+t} and $(0,0)\in \mathbf\Sigma^{kth}$ with $\lambda_k=\lambda_{k,0}(0)\leq k+1$. Let $(r_\ell)_{\ell \in \N}$ be an infinitesimal sequence and let $x_\ell \in \Sigma^{kth}(u^{t_\ell})\cap B_{r_\ell}$. For every $\ell$ set $v_{\ell}:=u(x_\ell+\,\cdot\,,t_\ell)-\curlyP_{k,x_\ell,t_\ell}$ and suppose that $\lambda_{k,t_\ell}(x_\ell) \to \lambda_{k}$. Consider the sequence 
\begin{equation*}
    \widetilde{ v}_{\ell}:=\frac{v_{\ell}(r_\ell\, \cdot\,)}{H(r_\ell,v_{\ell})^{\frac 1 2}},
\end{equation*} 
then
\begin{enumerate}[label={\upshape(\roman*)}]
    \item  $( \widetilde{ v}_{\ell})_{\ell \in \N}$ is bounded in $W^{1,2}_{\loc}(\R^n)$ and $C^{0,\frac{1}{n+1}}_{\loc}(\R^n)$.
    \item If $ \widetilde{ v}_{\ell}\weak q\in W^{1,2}_{\loc}(\R^n)$ then the convergence is strong and $q$ must be a nontrivial $\lambda_k$-homogeneous solution of the thin obstacle problem \eqref{eq:signorini} with obstacle $\{p_{2}=0\}$, that is
    \begin{equation*}
        \begin{cases}
           \lap q\leq 0\text{ and }q\lap q= 0 & \text{ in }\R^n,\\
           \lap q=0 &\text{ in }\R^n\setminus\{p_{2}=0\},\\
           q\geq 0 &\text{ on }\{p_{2}=0\}.
        \end{cases}
    \end{equation*}
     Finally, if $\lambda_k<k+1$ then $q$ is even with respect to the thin obstacle.
\end{enumerate}
\end{proposition}
\begin{proof}
  Given the convergence assumption $\lambda_{k,t_\ell}(x_\ell) \to \lambda_{k}$ and Lemma \ref{lem:EG2B1bisbis}, the proof is almost identical to Proposition \ref{pro:convergence}. 
\end{proof}
We now turn to the time-dependent version of Lemma \ref{lem:non-integer case}.
\begin{lemma}
\label{lem:non-integer case t}
Let $u\in C^0(\overline{B_1} \times [-1,1])$ solve \eqref{eq:UELL+t}, let $k\ge 2$ and suppose $(0,0)\in \mathbf\Sigma^{>k}\setminus \mathbf\Sigma^{\geq k+1}$, that is $\lambda_k=\lambda_{k,0}(0)\in (k,k+1)$. Suppose there exists an infinitesimal sequence $r_\ell\downarrow 0$ and $(x_\ell,t_\ell) \in \mathbf\Sigma^{kth}\cap B_{r_\ell}$, such that $\lambda_{k,t_\ell}(x_\ell) \to \lambda_{k}$. 
Assume further that, as $\ell\uparrow \infty$, we have 
\begin{enumerate}[label={\upshape(\roman*)}]
    \item $x_\ell/r_\ell \to y_\infty\in \overline {B_1}$;
    \item $\tilde v_{r_\ell}=\frac{(u_0-\curlyP_{k,0,0})_{r_\ell}}{H(r_\ell,u_0-\curlyP_{k,0,0})^{1/2}}\to q$ in $C^0_{\loc}(\R^n)$, for some $q\in\Sign_{\lambda_k}(\{p_{2,0,0}=0\})\setminus \{0\}$.
\end{enumerate}
Then $y_\infty\in \{p_{2,0,0}=0\}$ and $q=q(y_\infty + \,\cdot\,).$
\end{lemma}
\begin{proof} 
Whenever $x=t=0,$ we simplify the notation by dropping the indices, e.g., $p_{2,0,0}=p_2$. 
Consider a sequence $(x_\ell,t_\ell)_{\ell\in \N}\subseteq \mathbf\Sigma^{kth}\cap B_{r_\ell}$ as in the statement of the lemma. Note that $y_\infty\in \{p_2=0\}$ holds due to Lemma \ref{lem:EG2B1bisbis}. Applying Proposition \ref{pro:convergence t} with varying centers $(x_\ell)_{\ell\in\N}$ and respective sequence of times $(t_\ell)_{\ell\in\N}$, we find (after passing to a subsequence)
\begin{equation*}
   \widetilde v_{\ell}:=\frac{u(x_\ell+r_\ell\,\cdot\,,t_\ell)-\curlyP_{k,x_\ell,t_\ell}(r_\ell \, \cdot\,)}{\|(u(x_\ell+\,\cdot\,,t_\ell)-\curlyP_{k,x_\ell,t_\ell})_{r_\ell}\|_{L^2(\de B_1)}} \to Q
\end{equation*} 
 in $C^0_{\loc}(\R^n)$, for some $Q\in\Sign_{\lambda_k}(\{p_{2}=0\})\setminus \{0\}$. On the other hand by uniform convergence
 \begin{align*}
    q(y_\infty + \,\cdot\,) = \lim_{\ell} \frac{u(x_\ell+ r_\ell\,\cdot\,)-\curlyP_k(x_\ell+ r_\ell\,\cdot\,)}{H(r_\ell,u-\curlyP_k)^{1/2}}.
\end{align*}
We write
\begin{align}\label{eq:bla}
    \frac{u(x_\ell+ r_\ell\,\cdot\,)-\curlyP_k(x_\ell+ r_\ell\,\cdot\,)}{H(r_\ell,u-\curlyP_k)^{1/2}}&=
    \frac{u(x_\ell+ r_\ell\,\cdot\,)-u(x_\ell + r_\ell\,\cdot\,,t_\ell)}{\|(u(x_\ell+\,\cdot\,,t_\ell)-\curlyP_{k,x_\ell,t_\ell})_{r_\ell}\|_{L^2(B_1)}}\cdot a_\ell I_{\ell}\\\nonumber
    &+ \widetilde v_{\ell} \cdot b_\ell I_\ell + J_\ell,
\end{align}
where
$$I_\ell:=\frac{\|(u(x_\ell+\,\cdot\,,t_\ell)-\curlyP_{k,x_\ell,t_\ell})_{r_\ell}\|_{L^2(B_1)}+ H(r_\ell,u(x_\ell+\,\cdot\,,t_\ell)-\curlyP_{k,x_\ell,t_\ell}))^{1/2}}{H(r_\ell,u-\curlyP_k)^{1/2}}$$
is a numerical sequence and 
$$J_\ell:= \frac{\curlyP_{k,x_\ell}(r_\ell\,\cdot\,)-\curlyP_k(x_\ell+ r_\ell\,\cdot\,,t_\ell)}{H(r_\ell,u-\curlyP_k)^{1/2}}$$ is a sequence of harmonic polynomials of degree at most $k+1$. The numerical sequences 
$$
a_\ell :=\frac{\|(u(x_\ell+\,\cdot\,,t_\ell)-\curlyP_{k,x_\ell,t_\ell})_{r_\ell}\|_{L^2(B_1)}}{\|(u(x_\ell+\,\cdot\,,t_\ell)-\curlyP_{k,x_\ell,t_\ell})_{r_\ell}\|_{L^2(B_1)}+ H(r_\ell,u(x_\ell+\,\cdot\,,t_\ell)-\curlyP_{k,x_\ell,t_\ell}))^{1/2}}
$$
and $$
b_\ell :=\frac{H(r_\ell,u(x_\ell+\,\cdot\,,t_\ell)-\curlyP_{k,x_\ell,t_\ell}))^{1/2}}{\|(u(x_\ell+\,\cdot\,,t_\ell)-\curlyP_{k,x_\ell,t_\ell})_{r_\ell}\|_{L^2(B_1)}+ H(r_\ell,u(x_\ell+\,\cdot\,,t_\ell)-\curlyP_{k,x_\ell,t_\ell}))^{1/2}}
$$ are both bounded by $1$ and hence up to a subsequence, converge to some $a$, resp. $b \in [0,1]$.
Now two cases arise
\begin{enumerate}[label={\upshape(\roman*)}]
    \item $\sup_\ell I_\ell <\infty$
    \item $I_{\ell_m}\uparrow \infty$ for some subsequence $\ell_m\to \infty.$
\end{enumerate}
Let us begin with the first case. Up to a subsequence that we do not rename, we have $I_\ell\to \alpha$ and passing to the limit as $\ell \to \infty$ in \eqref{eq:bla} in $L^2(B_1)$ implies that $J_\ell$ converges to some harmonic polynomial $J$ of degree at most $k+1$. We find that
\begin{equation}\label{eq:bla2}
    q(y_\infty + \,\cdot\,) =a\alpha w +b\alpha Q + J
\end{equation} 
holds in $L^2$, for some function $w$ having a constant sign. Combining this fact with exploiting homogeneity of $q$ and $Q$, as in Proposition \ref{pro:convergence}, we find
\begin{align}\label{eq:ght}
   q(\,\cdot\,) \leq a\alpha Q(\,\cdot\,) + R^{-\lambda_k} J(R\,\cdot\,)\quad\mbox{or} \quad   q \ge a\alpha Q+ R^{-\lambda_k} J(R\,\cdot\,).
\end{align} 
Next, we remark that $J$ does not have a constant sign, as it is harmonic and vanishes somewhere on the line segment $\overline{-y_\infty0}.$ The last property is seen as follows. First, note that for large $\ell$ we have $H(r_\ell, u-\curlyP_k)^{1/2}\gg r_{\ell}^{\lambda_k+\delta}\gg r_{\ell}^{k+1}$. On the other hand, we calculate 
$$H(r_\ell, u-\curlyP_k)^{1/2}J_\ell(0)\leq Cr_\ell^{k+2} \quad\mbox{and}\quad H(r_\ell, u-\curlyP_k)^{1/2}J_\ell(-\frac{x_\ell}{r_\ell})\geq -Cr_\ell^{k+2}.$$ Hence, there exist a sequence of points $\overline{y}_\ell \in \overline{-\frac{x_\ell}{r_\ell}0}$ with 
$|J_\ell(\overline{y}_\ell)| \leq C r_\ell$ and so $J$ vanishes at some point in the line segment $\overline{-y_\infty 0}.$

As $J$ does not have a constant sign, there are directions $x_\pm\in \mathbb{S}^{n-1}$ with $J(Rx_\pm)\to \pm\infty$ as $R\to \infty.$ Combining this with \eqref{eq:ght} we thus find $\deg J \leq k$ and 
\begin{align*}
   q\leq a\alpha Q \quad\mbox{or} \quad   q \ge a\alpha Q.
\end{align*}
Thus, in any of the cases we have found two ordered $\lambda_k$-homogeneous solutions to the thin obstacle problem. Then, they must be equal, see \cite[Lemma A.4]{FRS19}.
Inserting this back in \eqref{eq:bla2} we find
$ q(y_\infty + \,\cdot\,) =q + J$. 
Therefore, for any $R>0$ we have
\begin{align*}
   R\left(q\left(\frac{y_\infty}{R} + \,\cdot\,\right)-q(\, \cdot\,)\right) \leq R^{1-\lambda_k}J(R\,\cdot\,) \quad\mbox{or} \quad R\left(q\left(\frac{y_\infty}{R} + \,\cdot\,\right)-q(\, \cdot\,)\right) \ge R^{1-\lambda_k}J(R\,\cdot\,).
\end{align*}
As the left hand side is bounded (and converges to $y_\infty\cdot\nabla q$) as $R\to \infty$, we exploit the fact that $J$ does not have a constant sign to find that the $k$-th coefficients must vanish. And so 
\begin{align*}
  y_\infty\cdot\nabla q \leq 0 \quad\mbox{or} \quad y_\infty\cdot\nabla q \ge 0.
\end{align*}
Reasoning as in Step 3 in the proof of \cite[Lemma 6.5]{FRS19}, we find that in any of the cases we must have $y_\infty\cdot\nabla q\equiv 0,$ as otherwise $y_\infty\cdot\nabla q$ would be a multiple of an eigenfuction to some elliptic problem on a subset on the sphere, which contradicts the high homogeneity of $y_\infty\cdot\nabla q$. 

The second case is simpler. We divide equation \eqref{eq:bla} by $I_{\ell_m}$ and find, after passing to a subsequence of $\ell_m$,
\begin{align*}
   0= aw+bQ + \tilde J
\end{align*}
for some harmonic polynomial $\tilde J$ of degree at most $k+1$. This is a contradiction, as the three functions $\tilde J,\,w,\,Q$ are not linearly dependent. Indeed, $w$ has a sign, $Q\neq 0$ is a $\lambda_k\in(k,k+1)$ homogeneous function and $\tilde J $ a harmonic polynomial with no constant sign. The fact that $\tilde J$ vanishes somewhere can be checked as we checked that $J$ vanishes somewhere, using that
$$
I_\ell\geq \frac{ H(r_\ell,u(x_\ell+\,\cdot\,,t_\ell)-\curlyP_{k,x_\ell,t_\ell}))^{1/2}}{H(r_\ell,u-\curlyP_k)^{1/2}} \gg \frac{r_\ell^{k+1}}{H(r_\ell,u-\curlyP_k)^{1/2}}.
$$
This finishes the proof.
\end{proof}

In order to perform the necessary dimension reductions we need some adaptations of Section \ref{sec:dimensionreduction}. We start with the following variation of Proposition \ref{prop:GMT2a}, taken from \cite{FRS20}.
\begin{proposition}[{\cite[Proposition 7.6]{FRS20}}]
\label{pro:parabolicGMT}
    Let $k\geq 2$ and $E\subseteq \R^n\times \R$. Suppose that:
    $$
    \forall (x,t)\in E, \forall \eps>0, \exists \varrho>0,\forall r\in(0,\varrho), \exists\, L \text { hyperplane},\,\exists\, q\in \Signeven_k(L),
    $$
    such that
$$\pi_x\left(E \cap \left(\overline{B_r(x)}\times (-\infty,t]\right)\right)\subseteq \Sigma(q)+\overline{B_{\eps r}(x)}.
$$
Then $\dim_{\mathcal H}(E) \leq n-2$. 
\end{proposition}

We want to apply this Proposition to $E=\bm\Sigma^{kth}\setminus \bm \Sigma^{(k+1)th}$, we check in the next two lemmas that this is possible.
\begin{lemma}\label{lem:oddfrequencyt}
Let $k\geq 2$ and $(0,0)\in \bm\Sigma^{kth}\setminus \bm \Sigma^{(k+1)th}$ and suppose $\lambda:=\lambda_{k}(0,0)$ is an even integer. Then 
$$
\forall \eps>0,\exists \varrho>0, \forall r\in(0,\varrho), \exists q\in \Signeven_\lambda(\{p_{2,0,0}=0\}),
$$
such that
$$
\pi_x\left(\{u=0\} \cap \left(\overline{B_r(x)}\times [0,1]\right)\right)\subseteq \Sigma(q)+\overline{B_{\eps r}(x)}.
$$
\end{lemma}
\begin{proof}
   Notice that, by Lemma \ref{lem: monotonicity k even}, $\lambda=k+1$. Further we have by monotonicity $$\pi_x\left(\{u=0\} \cap \left(\overline{B_r(x)}\times [0,1]\right)\right)\subseteq \{u(\,\cdot\,, 0)=0\}\cap \overline{B_r(x)}.$$ Now by Lemma \ref{lem:k+1}, and taking Remark \ref{rem:usefulremark} into account, the set $\{u(\,\cdot\,, 0)=0\}\cap \overline{B_r(x)}$ can be covered with a tubular neighbourhood of the singular set of a Signorini solution, provided $r$ is small enough.
\end{proof}
For odd frequencies we use a control ``in the past''.
\begin{lemma}\label{lem:evenfreqeuncyt}
Let $k\geq 2$ and $(0,0)\in \bm\Sigma^{kth}\setminus \bm \Sigma^{(k+1)th}$ and suppose $\lambda:=\lambda_{k}(0,0)$ is an odd integer. Then
$$
\forall \eps>0,\exists \varrho>0, \forall r\in(0,\varrho), \exists q\in \Signeven_\lambda(\{p_{2,0,0}=0\}),
$$
such that
$$
\pi_x\left(\bm \Sigma \cap \left(\overline{B_r(x)}\times [-1,0]\right)\right)\subseteq \Sigma(q)+\overline{B_{\eps r}(x)}.
$$
\end{lemma}
\begin{proof}
   Notice that either $\lambda=k$ or $\lambda=k+1$. In the first case, we reproduce the proof of Lemma \ref{lem:kodd}. In the second case, we reproduce the proof of Lemma \ref{lem:k+1} for the odd case. In both cases, it suffices to replace $u$ with $u(\cdot, 0)$ and the argument for a single solution can be applied. The key point is that, by monotonicity, the barriers $\{\phi_{z,\ell}\}$ will work for all $u(\cdot, t)$ for $t\leq 0$. Indeed, following the proof with the same notations, one arrives to
$$
\{\curlyA_{k,0,0}(r_\ell\cdot)=0\}\cap B_\rho(y_\infty)\subseteq Z_\ell\subseteq \interior \{u(r_\ell \cdot,0)=0\}\subseteq \interior \{u(r_\ell \cdot,t)=0\}.
$$
Hence the contact set of $u(\cdot,t)$ is fat around $y_\infty$. This gives $\Sigma(u(r_\ell \cdot, t))\cap B_{\rho/N}(y_\infty)=\emptyset$ for some dimensional constant $N$ and for all $t\leq 0$ (see \cite[Theorem 7]{C98} or the proof of \cite[Lemma 9.4]{FRS19}). This is the sought contradiction as $y_\ell \to y_\infty$ where $y_\ell \in \Sigma(u(r_\ell\, \cdot , t_\ell))$. 
\end{proof}
Putting all these results together we can prove the main theorem of this section. It is an extension of the 5-th order approximation result \cite[Theorem 8.7]{FRS19} to every order. For a fixed solution, this is just the content of our main Theorem \ref{thm:maintheoremintro}.
\begin{theorem}
\label{thm:dim bound with t}
Let $u\in C^0\big(\overline{B_1}\times [-1,1]  \big)$ solve \eqref{eq:UELL+t}. Then $\dim_{\mathcal H}(\pi_x(\mathbf\Sigma \setminus \mathbf\Sigma^\infty))\leq n-2$ and the set is countable if $n=2$. Moreover, for every $k\ge 2$ there exist constants $C=C(n,k)$ and $\rho=\rho(n,k)$ such that
\begin{equation}
    \label{eq:approx t}
    \|u(x_\circ + \,\cdot\,,t_\circ)-P_{k,x_\circ,t_\circ}\|_{L^\infty(B_r)}\leq Cr^{k+1}
\end{equation} holds with a unique polynomial $P_{k,x_\circ,t_\circ}$ of degree at most $k$ and $\lap P_{k,x_\circ,t_\circ}=1$, for all $0<r<\rho$ and $(x_\circ,t_\circ) \in \mathbf\Sigma^\infty \cap B_{1/2} \times (-1,1).$ 
\end{theorem}
\begin{proof}
 We recall from \cite[Proposition 8.1]{FRS19} that 
 $\dim_{\mathcal H}(\pi_x(\mathbf\Sigma \setminus \mathbf\Sigma_{n-1}))\leq n-2$ and is countable if $n=2$. Thus we need to show that for all $k\ge 2$:
\begin{enumerate}[label={\upshape(\roman*)}]
    \item $\dim_{\mathcal H}(\pi_x(\mathbf\Sigma^{\ge k} \setminus \mathbf\Sigma^{kth}))\leq n-2$ (countable if $n=2$),
    \item $\dim_{\mathcal H}(\pi_x(\mathbf\Sigma^{kth} \setminus \mathbf\Sigma^{>k}))\leq n-2$ (countable if $n=2$),
    \item $\dim_{\mathcal H}(\pi_x(\mathbf\Sigma^{>k} \setminus \mathbf\Sigma^{ \geq k+1}))\leq n-2$ (countable if $n=2$).
\end{enumerate}
By Lemmas \ref{lem:evenfreqeuncyt} and \ref{lem:oddfrequencyt}, to prove (i) and (ii) we can use Proposition \ref{pro:parabolicGMT} (or an obvious version of it for future times) with $E=\mathbf\Sigma^{\geq k} \setminus \mathbf\Sigma^{kth}$ and $E=\mathbf\Sigma^{kth} \setminus \mathbf\Sigma^{>k}$, respectively.

We turn to the proof of (iii). We can apply Proposition \ref{prop:GMT1} to the set $E=\pi_x(\bm\Sigma^{>k} \setminus \bm\Sigma^{ \geq k+1})$ and using the function $f (x_\circ):=\lambda_{k,\tau(x_\circ)}(x_\circ)$, where $\tau\colon \pi_x(\bm\Sigma)\to[-1,1]$ is defined by $\tau(x_\circ):= \min \{t\in[-1,1]:(x_\circ,t)\in \bm\Sigma\}$. The assumptions of Proposition \ref{prop:GMT1} holds for such $E$: if not we could argue by contradiction and blow-up exactly as in Proposition \ref{pro::dim red non integer}. The only difference is that we have to use Lemma \ref{lem:non-integer case t} above, instead of Lemma \ref{lem:non-integer case}. 
\end{proof}
\subsection{Cleaning lemmas in the time variable}\label{subsection:cleaning}
Following \cite{FRS19}, in this section we consider any monotone family of solutions $\{u^t\}_{t\in(-1,1)}$ of \eqref{eq:obstacle} in $B_1$, which additionally satisfy the following ``uniform monotonicity'' condition:
\begin{equation}\label{uniform monotonicity}
\begin{split}
\text{For every $t \in (-1,1)$ and any compact set $K_t\subseteq \partial B_1\cap \{u^t>0\}$ there exists $c_{K_t}>0$ such that}\\
\inf_{x\in K_t} \big( u^{t'}(x)- u^t(x)  \big) \ge  c_{K_t} (t'-t) , \qquad \text{for all }-1 <t< t'< 1.\qquad\qquad
\end{split}
\end{equation}
This condition rules out the existence of regions that remain stationary as we increase the parameter $t$, combining this observation with (iii) in Lemma \ref{lem:EG2B1bis}, one easily gets that $\bm\Sigma$ is a graph above $B_1$ in the sense that
$$
x\in\Sigma(u^t)\cap\Sigma(u^s) \quad \Rightarrow\quad s=t.
$$ 

We now turn to the ``cleaning lemmas'', namely Lemma \ref{lem:barrier for t} and Lemma \ref{lem:cleaning}. Using a barrier argument, we show that if $u^0$ is $O(r^\kappa)$-close to a polynomial Ansatz in $B_r$, then $u^t$ is positive in $B_r$ as soon as $t\sim r^\kappa$: thus the contact set was ``cleaned'' from $B_r$. The larger $\kappa$, the faster this cleaning takes place. Then we combine this reasoning with the polynomials expansion given by the $\curlyP_k$.

\begin{lemma}\label{lem:barrier for t}
Let $u\in C^0\big(\overline{B_1}\times [-1,1]  \big)$ solve \eqref{eq:UELL+t} and satisfy the uniform monotonicity assumption \eqref{uniform monotonicity}. Assume $(0,0)\in\bm{\Sigma}$ and let $\curlyP$ be a solution of $\lap \curlyP=1$ such that 
$$
| u(\,\cdot\,, 0) -\curlyP | \le Cr^\kappa \quad \text{ in } B_r, \qquad  \text{ for all }  \,r\in (0,1/2),
$$
for some $C,\kappa>0$. Then there exist $r_\circ,c>0$ such that
$$
u(\,\cdot\,,t) \ge  \curlyP+ crt -   Cr^{\kappa}  \quad\text{ in } B_{r/4}, \qquad  \text{ for all }  \,r\in (0,r_\circ).
$$
\end{lemma}
\begin{proof}
   This is a combination of Lemmas 9.1 and 9.2 in \cite{FRS19}.
\end{proof} 
The next result shows that, if $(x_\circ,t_\circ)\in \bm \Sigma^{\geq k+1}$, then the contact set surely disappears from $B_r(x_\circ)$ after $t-t_\circ \sim r^{k}$ units of time. 
\begin{lemma}
\label{lem:cleaning}
Let $u\in C^0\big(\overline{B_1}\times [-1,1]  \big)$ solve \eqref{eq:UELL+t} and satisfy the uniform monotonicity assumption \eqref{uniform monotonicity}. Suppose $(0,0)\in \mathbf\Sigma^{\ge k+1}$ for some $k\ge 2.$ Then there exists $r,C_0>0$ depending on $n,k$, such that 
$$
\big\{ (x,t)\in B_r\times (0,1) \ :\   t> C_0|x|^{k} \big\} \cap \{u=0\}  = \emptyset.
$$
\end{lemma}
\begin{proof}
Since $0\in \Sigma^{\geq k+1}(u^t)$, there exists $C(n,k)>0$ such that for every $r\in (0,1/2)$ 
\[
|u(\,\cdot\,,0)-\curlyP_k|\leq C r^{k+1} \quad\text{ in } B_r.
\]
Moreover, recall from Proposition \ref{prop:curly A vs curly P} that $\curlyP_k$ is almost positive, in the sense of
\[ \curlyP_k \ge -C(n,k)|x|^{k+2} \quad \text{ in } B_1.
\] 
Combining this with Lemma \ref{lem:barrier for t} with $\curlyP=\curlyP_k$ and $\kappa =k+1$, we get
\[
u(\,\cdot\,,t) \ge  \curlyP_k+ crt -   Cr^{k+1} \ge -Cr^{k+1}+crt \quad\mbox{in } B_{r/4} \qquad  \forall  \,r\in (0,r_\circ),\forall t\geq 0
\]
for some $r_\circ,c>0$. Now evaluating this at $(x,t)\in \de B_r\times (0,1)$, with $t>C_0r^{k}$, we get $u(x,t)>0$ as soon as $r$ is small and $C_0$ is large in terms of $c$ and $C$.
\end{proof}
We finally prove Theorem \ref{thm:instabilityintro}, combining Lemma \ref{lem:cleaning} with Proposition \ref{prop:GMT4b}.

\begin{proof}[Proof of Theorem \ref{thm:instabilityintro}]
For any $k\geq 2$ we can apply Proposition \ref{prop:GMT4b} to the set $E=\bm{\Sigma}^{\geq k+1}$ with $\beta=n$ and $s=k+1$, as the assumptions are satisfied thanks to Lemma \ref{lem:cleaning}. Hence, we get
$$
\dim_{\mathcal H} \left(\pi_t(\bm\Sigma^{\infty})\right)\leq \dim_{\mathcal H} \left(\pi_t(\bm\Sigma^{\geq k+1})\right)\leq \frac{n}{k+1},
$$
and (i) follows letting $k\uparrow \infty$.

For (ii) it suffices to show that $\pi_t(\mathbf\Sigma\setminus\mathbf\Sigma^{\infty})$ has zero Hausdorff dimension. By Proposition \ref{pro:dimensionreductionfinalstatement}, the set $\pi_x(\mathbf\Sigma\setminus\mathbf\Sigma^{\infty})$ is countable, provided $n=2$. On the other hand, by the strict monotonicity assumption \eqref{uniform monotonicity}, $\bm\Sigma$ is a graph above the space variables, hence $\bm\Sigma\setminus\bm\Sigma^{\infty}$ is also countable, this finishes the proof. Finally, (iii) is contained in Theorem \ref{thm:dim bound with t}.
\end{proof}
We turn to the proof of Corollary \ref{thm:Hele-Shawintro}. We remark that for analytic $f$'s we have at most countable many singular times (combining Theorem \ref{thm:instabilityintro} with \cite[Theorem 1.1]{Sak93}). For smooth $f$'s Theorem \ref{thm:instabilityintro} gives that singular times have zero Hausdorff dimension.  
\begin{proof}[Proof of Corollary \ref{thm:Hele-Shawintro}]
We divide the proof into two steps.

\textbf{Step 1.} The set $\Sigma(u^t)\setminus\Sigma^\infty(u^t)$ is not empty at most for countably many times. 
   
   The result follows directly from Theorem \ref{thm:instabilityintro} (iii) provided we show that $\{u^t\}$ satisfies the uniform monotonicity assumption \eqref{uniform monotonicity}. For completeness we give the argument: fix $t,h>0$ and $K\Subset \{u^t>0\}$. For brevity, we work with the assumption that $\Omega$ is connected, and thus unbounded.  Notice that $w:=u^{t+h}-u^t$ is harmonic in $\{u^t>0\}$, which is connected. By Schauder estimates and Lipschitz regularity of $\de\Omega$ we have that $\dist\left(\{u^t=0\},O\right)>\delta$ for some $\delta=\delta(n,\de\Omega,t)>0$. Hence we can build an open and connected set $V$, with Lipschitz boundary, such that
   $$
   \overline O \cup K \subseteq V \Subset \{u^t>0\}.
   $$
   By comparison we have $w\geq h\cdot \phi$, where $\phi$ solves
   \begin{equation*}
\begin{cases}
\lap \phi =0 & \quad \text{in } V\setminus \overline O, \\
\phi = 1 & \quad\text{on } \de O=\de\Omega, \\
\phi=0 & \quad \text{in } \de V. 
\end{cases}
\end{equation*}
As $\phi>0$ in $V\setminus \overline O$, we have $c:=\min_K \phi >0$, so for all $h>0$ and $x\in K$ it holds
$$
u^{t+h}(x)-u^t(x)\geq h \min_K \phi = c\, h.
$$
We used that $V$, and hence $\phi$, did not depend from $h$.

\textbf{Step 2.} The set $\Sigma^\infty(u^t)$ is not empty for at most countably many times.

Assume $0\in\Sigma^\infty(u^0)$, then we will show that we have an instantaneous cleaning of the zero set, that is: there exists a universal $\delta>0$ such that $B_\delta\cap \{u^t=0\}=\emptyset$ for all $t>0$. In fact, referring to the classification provided in \cite[Theorem 1.1]{Sak93}, we have that $0$ must be a ``degenerate'' point (case 2a) that is $\{u^0=0\}\cap B_\delta$ must be an analytic arc (it cannot be an isolated point). In particular, $\lap u^0=1$ in $B_\delta$ and $u^t-u^0$ is harmonic and non-negative in $B_\delta$, thus it is strictly positive in $B_{\delta/2}$, since, by assumption \eqref{uniform monotonicity}, it cannot be the zero function.

We explain how to prove that $0$ is not a ``double point'' (case 2b) nor a ``cusp'' (case 2c). If $0$ was a double point, it would be the tangency point of two distinct analytic arcs, but since the expansion of $u$ holds at any order these two arcs should have the same Taylor expansion, hence they are the same arc (so we are in case 2a). If $0$ was a cusp point, the cusp should be of the form given in \cite[Proposition 4.1]{Sak93}, in particular, up to a rotation, we would have two different functions $\alpha,\beta\colon [0,\delta)\to \R$ such that
$$
\{u^0=0\}\cap B_\delta =\{(x,y):\alpha(x)\leq y\leq \beta(x), x\geq 0\}\cap B_\delta
$$
But by the Lipschitz estimate \eqref{eq:lipestimateintro} we get for all $k\geq 2$
$$
|\curlyA_k(x,\alpha(x))|+|\curlyA_k(x,\beta(x)|\lesssim \sup_{\{u^0=0\}\cap B_r} |\de_n\curlyP_k| \lesssim r^{k}, \quad x\in[0,\delta).
$$
This shows that the graphs of $\alpha$ and $\beta$ are both tangent to the manifold $\{\curlyA_k=0\}$, up to order $k-1$. As $k$ was arbitrary, this forces $\alpha$ and $\beta$ to have the same polynomials expansion. By Proposition 4.1 in \cite{Sak93}, this forces $\alpha\equiv \beta$, contradiction.
\end{proof} 
\appendix

 \renewcommand{\thetheorem}{A.\arabic{theorem}}
 \renewcommand{\theequation}{A.\arabic{equation}}
\section{Proof of Lemma \ref{lem:importantlemma}}\label{sec:importantlemma}
We quickly prove Lemma \ref{lem:importantlemma} for a solution of \eqref{eq:obstacle} with $f\in C^\delta(B_1)$ for some $\delta\in (0,1]$. This is just an adaptation of the argument given in \cite{FS19}.

In this section we will call ``universal'' any constant depending on $n,\mu,\delta,\|f\|_{C^{\delta}(B_1)}$. We also assume that $0\in \de\{u>0\}$ and $0\in\Sigma(u)$, meaning that there exists a sequence $r_k\downarrow 0$ such that
$$
\frac{|\{u=0\}\cap B_{r_k}|}{|B_{r_k}|}\to 0,\qquad \text{ as }k\to \infty.
$$
\begin{lemma}
	There is a universal constant $C$ such that for all $r\in (0,1/2)$
	\begin{equation}\label{eq:optimalregularity}
	r^2 \leq C\sup_{\de B_r} u,\quad  \|u\|_{L^\infty(B_r)}\leq Cr^2,\quad \|Du\|_{L^\infty(B_r)}\leq Cr, \quad \|D^2u\|_{L^\infty(B_r)}\leq C.
	\end{equation}
\end{lemma}
\begin{proof}
See \cite[Theorem 2, Lemma 5]{C98}.
\end{proof}
From this we classify all possible blow-ups:
\begin{lemma}\label{lem:A2}
	Up to subsequences we have that
	\begin{equation}\label{eq:blowup}
	r_k^{-2}u(r_k\, \cdot\, )\weak f(0)p_2 \quad \text{ in } C^{1,1}_{\loc}(\R^n),
	\end{equation}
	where $p$ is a 2-homogeneous non-negative polynomial with $\lap p_2=1$. We denote with $\PP$ the set of such polynomials.
\end{lemma}
\begin{proof}
   Set $v_k:=r_k^{-2}u(r_k\, \cdot\, ) \in C^{1,1}(B_{1/r_k})$ and by weak* compactness it has a limit point $v\in C^{1,1}_{\loc}(\R^n)$ with $v\geq 0, v(0)=0$ and
   $$
   \|\nabla^2 v\|_{L^\infty(\R^n)} \leq \liminf_k \|\nabla^2 v_k\|_{L^\infty(B_{1/2r_k})} \leq \|\nabla^2 u\|_{L^\infty(B_{1/2})}\leq C. 
   $$ Since $0\in \Sigma(u)$ we also have that $f(r_k\,\cdot)\chi_{\{v_k=0\}}\to f(0)$ in $L^1_{\loc}(\R^n)$. A non-negative entire function with laplacian $f(0)$ and bounded hessian must be in $\PP$.
\end{proof}
Now we show that the blow-ups are unique using Weiss monotonicity formula for the adjusted energy (see \cite{W99}), we set:
$$
W_\lambda(r,v):=r^{-2\lambda}\big\{D(r,v)-\lambda H(r,v)\big\}.
$$
\begin{lemma}\label{lem:weissmonotonicity}
	There is a universal constant $C$ such that for all $p\in\PP$ and $r\in(0,1)$ we have
	\begin{equation}\label{eq:weissmonotonicty}
	\frac{d}{dr} W_2(r,u-f(0)p)\geq -Cr^{\delta-1}.
	\end{equation} 
\end{lemma}
\begin{proof}
	Set $v:=u-f(0)p$ and by a direct computation
	$$
	\frac{d}{dr} W_2(r,u-f(0)p)\geq \frac{2}{r^{5}}\int_{B_1} (2 v_r-x\cdot \nabla v_r)\lap v_r.
	$$
	Notice that $|\lap v_r +r^2f_r\chi_{\{u_r=0\}}|\leq r^2 \sup_{B_r}|f-f(0)|$. And thus
	\begin{align*}
	\int_{B_1} (2v_r-x\cdot \nabla v_r)\lap v_r&\geq-r^2\int_{B_1\cap\{u=0\}} (2v_r-x\cdot \nabla v_r)f_r-C\int_{B_1} |2v_r-x\cdot \nabla v_r|r^{2+\delta}\\
	&\geq r^2\int_{B_1\cap\{u=0\}}\underbrace{(2p_r-x\cdot \nabla p_r)}_{=0}f_r-Cr^{4+\delta}\geq -C r^{4+\delta}.\qedhere
	\end{align*}
\end{proof}
We deduce uniqueness of blow-ups and Monneau's almost monotonicity formula.
\begin{lemma}\label{lem:monneau}
	For all $p\in\bm P$ we have $W_2(0^+,u-f(0)p)=0$ and 
	\begin{equation}\label{eq:monneau}
	\frac{d}{dr}\left(r^{-4}H(r,u-f(0)p)\right)\geq -Cr^{\delta-1}\quad\text{ for all }r\in (0,1),
	\end{equation}
	with $C$ universal. In particular, the blow-up is unique at singular points and there exists a universal modulus of continuity $\omega\colon(0,1)\to\R, \omega(0^+)=0$ such that
	$$
	r^{-4}H(r,u-f(0)p_2)\leq \omega(r) \quad\text{ for all } r\in (0,1);
	$$
	provided $p_2$ is the blow-up.
\end{lemma}
\begin{proof}
	Choose some subsequence $r_k\downarrow 0$ and $p\in \PP$ such that $r_k^{-2}u_{r_k}\to p$. Then by Lemma \ref{lem:weissmonotonicity} and \eqref{eq:blowup} we have
	\begin{align*}
		W_2(0^+,u-f(0)p)&=\lim_k W_2(r_k,u-f(0)q)\\
		&=\lim_k D(1,r_k^{-2}u_{r_k}-f(0)q)-2H(1,r_k^{-2}u_{r_k}-f(0)q)\\
		&=\int_{B_1} |\der(p-q)|^2-2\int_{\de B_1} (p-q)^2=0,
	\end{align*}
	where in the last step we used that $p,q$ are 2-homogeneous and $\lap  p=\lap q$.
	
	Integrating equation \eqref{eq:weissmonotonicty} we get $W_2(r,v)\geq -Cr^\delta$, so by direct computation:
	\begin{align*}
	\frac{d}{dr}\left(r^{-4}H(r,u-f(0)p)\right)&=\frac{2}{r}\left\{W_2(r,v)+\frac{1}{r^4}\int_{B_1} v_r\lap v_r\right\}\\
	&\geq \frac{2}{r}\left\{-Cr^\delta+\int_{B_1\cap\{u_r=0\}}\underbrace{f(0)pf(r\, \cdot\, )}_{\geq 0}-Cr^{\delta}\right\}\\
	&\geq -Cr^{\delta-1}.
	\end{align*}
	This immediately gives uniqueness of the blow-ups, let us prove the existence of a universal rate of convergence of $u$ to such blow ups. Arguing by contradiction one finds $\epsilon>0$ and $u_k$ such that
	$$
	r_k^{-4}H(r_k,u_k - f(0)p_{2,k}) \geq \epsilon.
	$$
	Setting $v_k:=r_k^{-2} u_{k}(r_k\cdot)$ and arguing as in Lemma \ref{lem:A2} one finds $q\in \PP$ such that $v_k \to f(0)q$ in $C^1_{\loc}(\R^n)$. Now we get a contradiction using Monneau's monotonicity on $u_k$ and $q$:
	\begin{align*}
	\epsilon&\leq H(1,v_k-f(0)p_{2,k}) \lesssim H(1,v_k-f(0)q) + H(1,f(0)q-f(0)p_{2,k})\\
	&\leq H(1,v_k-f(0)q) +r_k^{-4}H(r_k,u_k-f(0)q)+Cr_k^\delta\\
	&\leq 2H(1,v_k-f(0)q)+Cr_k^\delta \to 0,\quad \text{ as } k\to \infty.
	\qedhere
	\end{align*}
\end{proof}
From now on we will denote with $f(0)p_2$ the unique blow up. Let us give a bunch of preliminary estimates on the function $v:=u-f(0)p_2$:
\begin{lemma}\label{lem:aprioriestimates}
	Take any $p\in \PP$ and set $v_r:=(u-f(0)p)_r$ then the following estimates hold with universal constants for all $r\in (0,1/2):$
	\begin{align*}
	\lap v_r&=- r^2f_r \chi_{\{u_r=0\}}+O( r^{2+\delta});\\
	\|v_r\|_{L^\infty(B_1)}&\lesssim \|v_r\|_{L^2(B_2\setminus B_{1/2})}+r^{2+\delta};\\
	r^2|\{u_r=0\}\cap B_1|&\lesssim \|v_r\|_{L^2(B_2\setminus B_{1/2})}+r^{2+\delta};\\
	v_r\lap v_r &= r^4 f(0)f_rp_2\chi_{\{u_r=0\}}+v_rO(r^{2+\delta});\\
	\|\nabla v_r\|_{L^2(B_1)}&\lesssim \|v_r\|_{L^2(B_2\setminus B_{1/2})} +r^{2+\delta};\\
	[v_r]_{C^{0,\frac{\delta}{2\delta +n-1}}}&\lesssim \|v_r\|_{L^2(B_2\setminus B_{1/2})} +r^{2+\delta},\quad \text{ provided }\dim \{p=0\}=n-1.
	\end{align*}
\end{lemma}
\begin{proof}
	The first is a direct computation exploiting H\"older continuity of $f$.
	
	For the second we notice that
	\begin{itemize}
		\item $\lap v\leq C r^\delta$ in $B_r$ and $\lap v\geq -Cr^\delta$ in $B_r\cap \{u>0\}$;
		\item $v\leq 0$ in $B_r\cap \{u=0\}$.
	\end{itemize}
	Hence sub/super harmonic comparison gives the result as in Lemma \ref{lem:L2toLinfty}.
	
	For the third we choose $\chi_{B_1}\leq \eta \leq \chi_{B_2}$ and compute	
	\begin{align*}
	\mu r^2|\{u_r=0\}\cap B_1|&\leq \int_{B_1} r^2f_r\chi_{\{u=0\}}\\
	&\leq Cr^{2+\delta} -
	\int_{B_1} \lap v_r\\
	&=\int_{B_1}\underbrace{\lap \left(\frac{Cr^{2+\delta}|x|^2}{2n}-v_r\right)}_{\geq 0}\\
	&\leq \int_{B_2} \lap \left(\frac{Cr^{2+\delta}|x|^2}{2n}-v_r\right)\eta\\
	&\leq C_\eta\left(\|v_r\|_{L^2(B_2\setminus B_1)}+r^{2+\delta}\right).
	\end{align*}
	
	The fourth is a direct computation.

	Since $v_r\lap v_r\geq -Cr^{2+\delta}|v_r|$, for the fifth we can do the Caccioppoli inequality:
	\begin{align*}
	\int_{B_2}\eta^2 |\nabla v_r|^2&=-2\int_{B_2} \eta v_r \nabla v_r\cdot \nabla\eta -\int_{B_2}\eta^2 v_r\lap v_r\\
	&\leq 4\|\eta \nabla v_r\|_{L^2(B_2)}\|v_r\|_{L^2(B_2\setminus B_1)} +Cr^{2+\delta}\int_{B_2}|v_r|\\
	&\leq \frac{1}{2} \|\eta \nabla v_r\|^2_{L^2(B_2)} +C\left(\|v_r\|^2_{L^2(B_2\setminus B_1)}+r^{2(2+\delta)}\right),
	\end{align*}
	where $\eta$ is as above.
	
	For the sixth assume $p_2=\frac 1 2 x_n^2$ and  consider for $0<t<1, j\neq n$ the function
	$$
	w_{\pm}(x):=\frac{v_r(x\pm t e_j)-v_r(x)}{t^\delta}=\frac{u_r(x\pm t e_j)-u_r(x)}{t^\delta}.
	$$
	Notice that, with constants uniform in $t$, we have
	$$
	\|w_\pm\|_{L^2(B_2\setminus B_{1/2})}\lesssim t^{1-\delta}\, \|\nabla v_r\|_{L^2(B_4\setminus B_{1/4})}\lesssim \|v_r\|_{L^2(B_8\setminus B_{1/8})}+r^{2+\delta}.
	$$
	On the other hand in $\{u_r>0\}\cap B_1$ we have $\lap w_\pm\lesssim r^{2+\delta}$ and in $\{u_r=0\}\cap B_1$ we have $w_\pm\geq 0$. Thus the function 
	$$
	\min\left\{ w_\pm+Cr^{2+\delta}\frac{1-|x|^2}{2n};0\right\} \text{ is superharmonic in }B_1.
	$$
	Using the minimum principle and the previous estimate we get
	$$
	\min_{\overline{B_1}}w_\pm\geq -C\left(\|v_r\|_{L^2(B_4\setminus B_{1/4})}+r^{2+\delta}\right).
	$$
	By the symmetry $w_\pm(x\mp te_j)=-w_\mp(x)$, we also have the upper bound on a smaller ball. Since all the constants are uniform in $t$ we conclude using the following estimate (cf. Lemma \ref{lem:hoelder}):
	$$
	[f]_{C^{0,\frac{\delta}{2\delta+n-1}}(B_1)}\lesssim_{n,\delta} \sum_{j=1}^{n-1}\sup_{x\in B_1, |t|\leq 1}\frac{|f(x+te_j)-f(x)|}{|t|^\delta} +\|\partial_n f\|_{L^2(B_2)},
	$$
	valid for every $f\in \lipschitz(B_2)$.
\end{proof}
Since the blow up is well defined we can from now on assume to be in the top-dimensional stratum that is $p_2=\frac 1 2 x_n^2$. Arguing as is Section \ref{sec:monotonicity frequency} we exploit the truncated frequency $\phi^\gamma$ with some $\gamma(\delta)>2$.

\begin{lemma}\label{lem:blowup}
	Let $p\in \PP$ and $\gamma=2+\delta/8$ and set $v:=u-f(0)p$. Then there is $\eps=\eps(\delta)>0$ such that following inequalities hold for all $r\in (0,1)$:
	\begin{equation}\label{eq:freqalmostmon}
	\phi^\gamma(r,v)\geq 2-C {r^\eps},\quad\phi^\gamma(r,v)\leq C,\qquad\frac{d}{dr} \phi^\gamma(r,v)\geq -C r^{\eps-1},
	\end{equation}
 and $C$ is universal. Furthermore, we also have
	\begin{equation}\label{eq:auxiliarymonoticity}
	\frac{\int_{B_1}v_r\lap v_r}{H(r,v)+r^{2\gamma}}\geq -C r^{\eps}.
	\end{equation}
\end{lemma}
\begin{proof}
	For the first inequality in \eqref{eq:freqalmostmon} we employ Lemma \ref{lem:weissmonotonicity}:
	$$
	\phi^\gamma(r,v)-2=\frac{D(r)-2H(r) +(\gamma-2)r^{2\gamma}}{H(r)+r^{2\gamma}}\geq \frac{W_2(r)}{r^{-4}H(r)+r^{2\gamma-4}}\geq -Cr^{\delta-2(\gamma-2)},
	$$
	so we can set $\eps:=3\delta/4$. For the second we need to estimate from below with $-Cr^{\eps-1}$ the term
	$$
	\frac{2}{r(H(r)+r^{2\gamma})} \int_{B_1} \left(\lambda_r v_r -x\cdot \nabla v_r\right)\lap v_r \,dx,	
	$$
	where for brevity $\lambda_r:=\phi^\gamma(r,v)$ (cf. Proposition \ref{pro:almost monotonicity}). Recall that $|\lap v_r +r^2f_r\chi_{\{u_r=0\}}|\leq C r^{2+\delta}$ and estimate each term recalling that $p_2$ is 2-homogeneous:
	\begin{align*}
	\int_{B_1} \left(\lambda_r v_r -x\cdot \nabla v_r\right)\lap v_r&=-r^2\int_{B_1\cap\{u_r=0\}}\left(\lambda_r v_r -x\cdot \nabla v_r\right)f_r-Cr^{2+\delta}\int_{B_1}\left|\lambda_r v_r -x\cdot \nabla v_r\right|\nonumber\\\nonumber
	&\geq r^2\int_{B_1\cap\{u_r=0\}}\left(\lambda_r p_r -2p_r\right)f(0)f_r-Cr^{4+\delta}(\lambda_r+1)\\\nonumber
	&\geq  r^{4} \underbrace{(\lambda_r-2)}_{\geq-Cr^\eps}\underbrace{\int_{B_1\cap\{u_r=0\}}pf(0)f_r}_{\geq 0}-Cr^{4+\delta}(\lambda_r+1)\\\nonumber
	&\geq -Cr^4\left(r^{\eps}+r^{\delta}(\lambda_r+1)\right)
	\end{align*}
	so with crude bounds the frequency solves the ODI:
	\begin{equation}\label{eq:ODI}
	\lambda_r'\geq -C r^{3-2\gamma}\left(r^{\eps}+r^{\delta}(\lambda_r+1)\right)\geq -C r^{3+\eps-2\gamma}(\lambda_r+1).
	\end{equation}
	From here we see that $\log(1+\lambda_r)$ is almost monotone and bounded above by some constant, provided $\gamma<2+\eps/2$. Thus plugging this back into \eqref{eq:ODI} we get 
	$$
	\lambda_r'\geq -C r^{3+\eps-2\gamma},
	$$
	which was the claim up to re-defining $\eps$. Equation\eqref{eq:auxiliarymonoticity} follows as in the proof Lemma \ref{lem:monneau} above.
\end{proof}
Hence $\phi^\gamma(0^+,v)\geq 2$ exists for all $p$, we want to show that there is a universal number $\alpha_\circ>0$ such that $\phi^\gamma(0^+,u-f(0)p_2)\geq 2+2\alpha_\circ,$ provided $p_2$ is indeed the blow up at $0$. Let us show how to conclude from here. Up to universal constants we have the following: by Lemma \ref{lem:aprioriestimates} we have
$$
\|v\|_{L^\infty(B_1)}\lesssim \|v_r\|_{L^2(B_2\setminus B_{1/2})}+r^{2+\delta}.
$$
But $\phi^\gamma\leq C$ in $(0,1)$ so by Lemma \ref{lem:H ratio} we have in turn:
$$
\|v_r\|^2_{L^2(B_2\setminus B_{1/2})}\lesssim H(r/2)+r^{2\gamma},
$$
and $\gamma>2$. Now, since $\phi^\gamma(0^+,v)=2+2\alpha_\circ$, we have, again by Lemma \ref{lem:H ratio}, that
$$
H(r)+r^{2\gamma}\lesssim r^{2(2+2\alpha_\circ)},
$$
hence putting everything together we obtain Lemma \ref{lem:importantlemma}:
$$
\|v\|_{L^\infty(B_1)}\lesssim r^{2+2\alpha_\circ}.
$$
So we are left to show that
\begin{equation}\label{eq:frequencyjump}
\lambda_2(0):=\phi^{\gamma}(0^+,u-f(0)p_2)\geq 2+2\alpha_\circ,
\end{equation}
and it is also clear that we can work under the assumption that $\lambda_2(0)\leq 2+\delta/16$, otherwise \eqref{eq:frequencyjump} holds with $\alpha_\circ=\delta/64$. The following proposition is crucial and the proof follows the same line of Proposition \ref{pro:convergence} (or also of \cite[Proposition 2.12]{FS19}. As the only technical complications are settled by the bounds gathered in Lemma \ref{lem:aprioriestimates}, we omit the proof.
\begin{lemma}
	Assume $0\in \Sigma_{n-1}$ and $\lambda_{2}\leq 2+\alpha/16$. Then the sequence 
	$$
	\tilde{v}_r:=\frac{v_r}{\|v_r\|_{L^2(\de B_1)}}
	$$
	is bounded in $W^{1,2}_{\loc}(\R^n)\cap C^{\delta/(2\delta+n-1)}_{\loc}(\R^n)$. Furthermore, every accumulation point of $\{v_r\}_{r>0}$ solves the Signorini problem \eqref{eq:signorini} and is $\lambda_2$-homogeneous.
\end{lemma}
The following combination of Monneau monotonicity and the characterization of blow-ups will prove \eqref{eq:frequencyjump}. The proof is in fact very similar to Step 5 in the proof of Proposition \ref{pro:convergence}.
\begin{lemma}
	There cannot be a sequence $u_k,f_k,\mu_k,\delta_k$, with $0\in\de\{u_\ell>0\}$ and
	$$
	\sup_\ell \left(
	\|f_\ell\|_{C^{\delta_\ell}(B_1)}+\frac{1}{\delta_\ell}+\frac{1}{\mu_\ell}\right)<+\infty,
	$$ 
	such that $\lambda^{(k)}_2\downarrow 2$, where
	$$
	\lambda_{2}^{(k)}:=\phi^{2+\delta_k/8}(0^+,u_k-f_k(0)p_2^{(k)}).
	$$
	In particular \eqref{eq:frequencyjump} holds for some $\alpha_\circ=\alpha_\circ(n,k,\delta,\|f\|_{C^\delta(B_1)})\in(0,1)$.
\end{lemma}
\begin{proof}
	\textbf{Step 1.} If $\tilde v_{r_k}\weak q$ in $W^{1,2}_{\loc}$ then for all $p\in\PP$ we have
	\begin{equation}\label{eq:positivecorrelation}
		\int_{\de B_1} q(p_2-p)\geq 0.
	\end{equation}
	
	\textit{Proof of Step 1.} Define $\eps_k^2:=H(r_k,v)$ and notice that, by the growth Lemma \ref{lem:H ratio} and the compactness of the trace operator, we have for $k$ large
	$$
	r_k^\delta\ll \eps_k\to 0,
	$$ 
	where we used that $\phi^\gamma(r_k)\leq2+ \delta/100$ for all $k$ large enough. By Monneau monotonicity (Lemma \ref{lem:monneau}) applied to $p$ instead of $p_2$ we have:
	\begin{equation*}
		\int_{\de B_1}\left(\eps_k \tilde v_{r_k}+p_2-p\right)^2 +Cr_k^\delta\geq\int_{\de B_1}(p_2-p)^2,
	\end{equation*}
	computing the squares and dividing by $\eps_k$ we get
	$$
	\eps_k\int_{\de B_1}\tilde v_{r_k}^2+2\int_{\de B_1} \tilde v_{r_k}(p_2-p)+C\frac{r_k^\delta}{\eps_k}\geq 0,
	$$
	sending $k\uparrow \infty$ we obtain \eqref{eq:positivecorrelation}. We remark that all the constants in these computations are universal.
	
	\textbf{Step 2.} If $q$ is a 2-homogeneous harmonic polynomial such that \eqref{eq:positivecorrelation} holds for all $p\in\PP$ then $q\leq 0$ on the hyperplane $\{p_2=0\}$.
	
	\textit{Proof of Step 2.} This is exactly \cite[Lemma 2.12]{FS19}.
	
	\textbf{Step 3.} For each $u_k,f_k,\mu_k,\delta_k$ as in the assumptions, Lemma \ref{lem:blowup} gives $q_k$, a $\lambda^{(k)}_2$-homogeneous solution of the Signorini problem with $\|q_k\|_{L^2(\de B_1)}=1$. It is easy to see that, by compactness, $q_k\to q$ where $q$ is a $2$-homogeneous solution of Signorini with $\|q\|_{L^2(\de B_1)}=1$. Thus, $q$ is an harmonic polynomial, non-negative on the thin obstacle (see Proposition \ref{pro:signorini}). But this contradicts Step 1, up to taking a diagonal subsequence. A careful verification that all the bounds are uniform is the same as the Step 1 in the proof of Proposition \ref{pro:convergence}, and it is not repeated here.
\end{proof}
 \renewcommand{\thetheorem}{B.\arabic{theorem}}
 \renewcommand{\theequation}{B.\arabic{equation}}
\section{Adaptations for general right hand sides}\label{sec:rhs f}
In this section we collect the modification needed to work with a general $f$ and $\mu$. 

The main difference is that $u-\curlyP_k$ will not be harmonic in $\{u>0\}\cap B_r$, but its Laplace operator will be of size $O(r^{k})$. This is the size of the error we would anyway have in every estimate. Having this is mind, it is clear that all the arguments go trough with the same proof, provided we can indeed construct $\curlyP_k$ with the same properties as before. This is not an hard task. We will, for completeness, list also the other modifications needed. Let us remark that all constants that in the case $f\equiv 1$ depend on $n$ and $k$, will now also depend on $\mu,\|f\|_{C^k}$.

In the following, we provide a generalization of Section \ref{subsec:polynomial ansatz}. We begin with the respective polynomial Ansatz, which will additionally depend on the Taylor expansion of $f$ and on the center of expansion. We will denote by $F_{k,x}$ the $k^{\text{th}}$ Taylor polynomial of $f$ at $x$, that is
$$
F_{k,x}(h):=\sum_{|\alpha|\leq k}\frac{\de^\alpha f(x)}{\alpha!}h^\alpha.
$$
The sets $\PP_k$ and $V_j$ are the same of Section \ref{subsec:polynomial ansatz}. 
\begin{lemma}\label{lem:ansatzalgf}
	Let $k\geq 2, f\in C^{k-1}(B_1), x \in B_1$ and $(p_2,\ldots,p_k)\in \PP_k$ be given. Let $\nu$ be any unit vector such that $p_2(h)=\frac 1 2(h\cdot \nu)^2$. There exists a unique collection of polynomials
	$$
	(R_{1},\ldots R_{k-1})\in V_1\times\ldots\times V_{k-1},
	$$
	such that if we define the polynomial
	$$
	\curlyA_{x,k,\nu}(h):= (\nu \cdot h) +\sum_{j=1}^{k-1} (\nu \cdot h)R_j(y)+\sum_{j=3}^{k} \frac{p_j(h)}{(\nu \cdot h)},
	$$
	then 
	$$\lap \left(\frac {f(x)}{2} \curlyA_{x,k,\nu}^2\right)(h) =F_{k-1,x}(h)+O(|h|^k).$$
	Furthermore, each $R_{j}$ is determined (analytically) only by $(p_2,\ldots, p_{j+1})$ and the coefficients of $F_{k,x}$.
	In particular, each $R_{j}$ does not depend on $\nu$, so $\curlyA_{x,k,-\nu}=-\curlyA_{x,k,\nu}$.
\end{lemma}
\begin{proof}
	The proof is almost identical to the proof of Lemma \ref{lemma:ansatzalg}, the only difference being that we have to take into account the Taylor expansion of $f$. Let us work out explicitly the case $k=2$. By a direct computation we find
	$$
	\lap\left(\frac{f(x)}{2} \curlyA_{x,k,\nu}^2\right)(h)=f(x)+\lap(2f(x)p_2R_1)(h)+O(|h|^2).
	$$
	Thus, the right (and unique) choice for $R_1$ is
	$
	R_1:=\frac{1}{2f(x)}\delta_1^{-1}\left(F_{1,x}\right),
	$
	where the linear isomorphisms $\delta_m\colon V_m\to V_m$ were introduced in the Proof of Lemma \ref{lemma:ansatzalg}.
\end{proof}
Using Lemma \ref{lem:ansatzalgf} we can define the polynomial Ansatz functions $\curlyA_k^2,\curlyP_k\colon B_1\times \PP_k \to \R[h]$, which now depend explicitly also on the center of expansion $x$. We set
$$
\curlyA_k(x;p_2,\ldots,p_{k-1}):=\curlyA_{x,k,\nu}^2,\qquad \curlyP_{k}(x;p_2,\ldots,p_{k-1}):=\pi_{\leq k+1}\left(\frac{f(x)}{2}\curlyA_{x,k,\nu}^2\right),
$$
and notice that the dependence on $f$ is hidden into the dependence on $x$. Once again any norm of $\curlyP_k$ is bounded by constants depending on $n,k,\|f\|_{C^{k-1}(B_1)}$ and $|(p_2,\ldots,p_k)|$. Furthermore, the function $\curlyP_k(x;\,\cdot\,)$ is injective. 

With this construction we obtain:
\begin{equation}\label{eq:formalproperties}
\lap\curlyP_k(x;p_2,\ldots,p_k)(h)=f(x+h)+O(|h|^k),\quad \text{ and }\quad \curlyP_k(x;p_2,\ldots,p_k)(h)\geq-C|h|^{k+2}
\end{equation}
where the big $O$ is a $C^{k}$ function of $x,h$. Here comes the only difference with the case in which $f\equiv 1$. When we apply the Laplace operator to the function $v:=u(x_\circ+\, \cdot)-\curlyP_k(x_\circ;p_2,\ldots,p_k)$, we get in $B_r(x_\circ)$:
\begin{equation}\label{eq:lapvf}
\lap v=-f(x_\circ+\, \cdot)\chi_{\{u(x_\circ+\, \cdot)=0\}}+O(r^k),
\end{equation}
while in the case $f\equiv 1$ we had $\lap v=-\chi_{\{u(x_\circ +\,\cdot\,)=0\}}$ exactly. 

Let state and analogous of Proposition \ref{prop:curly A vs curly P}, which contained all crucial properties of the Ansatz.
\begin{proposition}\label{prop:curlyA vs curlyP f}
	Let $k\geq 2, (p_2,\ldots,p_k)\in \PP_k$ and $\tau>0$ be such that $ |(p_2,\ldots,p_k)|\leq \tau$. Choose some unit vector $\nu$ for which $p_2(x)=\frac 1 2 (\nu\cdot x)^2$. Let $f\in C^{k-1,1}(B_1)$ and let $|x_\circ|<\frac 1 2$.  Then the polynomials $\curlyA_{k}^2(x_\circ;p_2,\ldots,p_k)$ and $\curlyP_k(x_\circ;p_2,\ldots,p_k)$ satisfy
	\begin{enumerate}[label={\upshape(\roman*)}]
		\item $\lap \curlyP_k=f(x_\circ+\cdot)+O(|\cdot|^k)$ and  $\de_e(\frac 1 2\curlyA_k^2)=\de_e \curlyP_k + O(|\cdot|^{k+1})$ for any unit vector $e$.
		\item We have that
		$
		\curlyP_k(x_\circ;p_2,\ldots,p_k)=\curlyP_{k-1}(x_\circ,p_2,\ldots,p_{k-1})+ p_k+O(|\cdot|^{k+1}).
		$
		\item For all $|h|\leq r_0$ we have $\frac 1 2 \leq |\de_\nu \curlyA_{x_\circ,k,\nu}(h)| \leq 2$ and thus
		\begin{equation*}
		\frac 1 2 |\curlyA_k(h)|\leq |\de_\nu\left(\tfrac 1 2 \curlyA_k^2(h)\right)| \leq 2 | \curlyA_{k}(h)|,
		\end{equation*}
		where $r_0=r_0\big(n,k,\tau,\|f\|_{C^{k}(B_1)}\big)\in (0,1)$.
		\item If $u$ is a solution as in \eqref{eq:obstacle}, $0\in \Sigma_{n-1}$ and $r^{-2}u(r\cdot)\to p_2$ then by \eqref{eq: important eq1} we have for all $0<r<1/2$
		$$
		\sup_{B_{r}(x_\circ)\cap\{u=0\}}|\de_\nu\curlyP_k|\leq C r^{1+\alpha_\circ},$$ for some constant $C=C(n,k,\tau,\|f\|_{C^k})$.
	\end{enumerate}
\end{proposition}

Now that the polynomial Ansatz have the right formal properties (i.e., \eqref{eq:formalproperties}, \eqref{eq:lapvf} and the ones collected in Proposition \ref{prop:curlyA vs curlyP f}) it is simple to check that the rest of the arguments goes trough. The rest of this section is a list of the modification needed to obtain Theorem \ref{thm:maintheoremintro} in its full generality.
\begin{itemize}
	\item Lemmas \ref{lem:L2toLinfty} and \ref{lem:bound on de_j v} are the same: even if in $B_r\cap \{u>0\}$ the function $u-\curlyP_k$ is not harmonic, pointwise it holds $\lap(u-\curlyP_k)=O(r^k)$. The comparison principle we use remains valid in this case.
	\item The proof of Lemma \ref{lem:bound on de_n P} is identical, except for the fact that in $\Omega$ our function is not harmonic. This is used only in \eqref{eq in the middle}, where we have the term $\|\lap v_r\|_{L^\infty(\widetilde\Omega\cap B_1)}$, but it can be absorbed in the term $Cr^{k+2}$.
	\item In Lemma \ref{lem:estimate on the contact set} we also pick up an extra term, which however is much smaller than the one we are estimating. Indeed we have
	$$
	\int_{B_1} |v_r\lap v_r|\leq M\int_{B_1\cap\{u=0\}} |v_r|+Cr^k\int_{B_1}|v_r|,
	$$
	and the first integral is treated as in Lemma \ref{lem:estimate on the contact set}. For the second term, as around a contact point it holds $|v_r|\lesssim r^2$, we estimate
	$$
	\frac 1 r \frac{Cr^{k+2}}{H(r,v)+r^{2\gamma}}\lesssim r^{k+1-\gamma}=r^{\epsilon-1}.
	$$
	As $|x\cdot \nabla v_r|\lesssim r^2$ in $B_r$, the same reasoning applies to the term $\int_{B_1}|(x\cdot \nabla v_r)\lap v_r|$.
	
	The rest of Section \ref{sec:monotonicity frequency} goes on with exactly the same proofs.
	
	\item Section  \ref{sec:blowup} essentially uses the statements of two previous sections as black boxes. The only modification is in the very definition of the sets $\Sigma^{kth}$, namely for $x_\circ \in  \Sigma^{kth}$ we use the Ansatz
	$$
	\curlyP_{k,x_\circ}:=\curlyP_k(x_\circ;p_{2,x_\circ},\ldots,p_{k,x_\circ}),
	$$
	which is again continuous in the $x_\circ$ variable. Notice that to do our argument we want that in $\Sigma^{kth}$ it makes sense to construct $\curlyP_k$, hence we need $f\in C^{k}(B_1)$ at least.
	
	\item Concerning Section \ref{sec:dimensionreduction}, note that we cannot use any more that $\lap\curlyP_k=0$ in $\{u>0\}$. Thus in Lemma \ref{lem: monotonicity k even} we introduce a small modification, namely in the term
	$$
	\int_{B_1}P \lap v_r=-\int_{B_1\cap\{u_r=0\}}fP+O(r^{k+2})\int_{B_1}P,
	$$
	but we underline that the extra factor $O(r^{k+2})$ does not affect any subsequent computation.
	\item In Lemma \ref{lem:kodd} the definition of the barrier function needs to be adapted. Namely equation \eqref{eq:Pluto} is going to be replaced with
	\begin{equation*}
	\begin{cases} 
	\phi_{z,\ell}(z)=0,  \\ 
	\phi_{z,\ell}  \geq 0 & \text{ in } B_\rho(z), \\ 
	\lap \phi_{z,\ell} < r_\ell^2 f(r_\ell\, \cdot\,) &\text{ in } \overline{ B_\rho(z)},\\
	u(r_\ell\,\cdot\,)<\phi_{z,\ell} & \text{ on } \de B_{\rho}(z),
	\end{cases}
	\end{equation*}
	so that the proofs of Claims 2 and 3 are the same. For Claim 1 we have to use the following barrier:
	\begin{equation*}
	\phi_{z,\ell}(x):=\left(1-\frac{h_\ell}{r_\ell^2}\right) \frac {f(0)}{2} \mathcal{A}^2_k(r_\ell x) + \frac{h_\ell}{4nM}  |x'-z'|^2.
	\end{equation*}
	\item Sections \ref{sec: frequency k}, \ref{sec: intermediate set}, \ref{sec: end point} do not require further modifications.
	\item In the proof of \ref{lem:Whitney} the constant of equation \eqref{eq:123} depends on $\|\der^kf\|_{L^\infty}$.
	\item Being based on Section \ref{sec:dimensionreduction}, Sections \ref{subsec:setup} and Section \ref{subsect:dimredt} do not require any modification. 
	\item In Section \ref{subsection:cleaning}, it is easily checked that Lemma \ref{lem:cleaning} works if we assume that $\lap \curlyP = f +O(r^\kappa)$ instead of $\lap \curlyP = 1$. Then the cleaning works just as before.
\end{itemize}
 \renewcommand{\thetheorem}{C.\arabic{theorem}}
 \renewcommand{\theequation}{C.\arabic{equation}}
\section{Auxiliary lemmas}\label{sec:auxiliarylemmas}

\begin{lemma} \label{lem:hoelder}
For every $u\in \lipschitz(B_2)$, $1\leq j\leq n$ and $\beta\in(0,1]$ define
$$
[\delta_j u]_{\beta}:=\sup_{x\in B_1,|t|\leq 1} \frac{|u(x+te_j)-u(x)|}{|t|^\beta}.
$$
Then for all $p>1$ there holds
$$
[u]_{C^{0,\sigma}(B_1)}\leq C \left(\sum_{1\leq j<n}[\delta_j u]_{\beta}+\|\de_n u\|_{L^p(B_2)}\right),
$$
where $C=C(n,\beta,p)$ and $\sigma=\tfrac{\beta(p-1)}{\beta p+n-1}$.
\end{lemma}
\begin{proof}
By homogeneity we can assume that the right-hand side is $1$. Set $h=(0,\dots,0,r)$ and consider $A_r:=B'_{r^\alpha}\times [0,r]$ where $\theta>0$ is small. By Fubini's Theorem we can find some $z'\in B_{r^\theta}'$ such that
$$
\int_{0}^r|\de_nu(z',s)|^p\, ds\leq r^{\theta(1-n)}\|\de_n u\|_{L^p(A_r)}^p\leq r^{\theta(1-n)}.
$$
The fundamental theorem of calculus and H\"older's inequality give
\begin{align*}
|u(0)-u(h)|&\leq |u(0)-u(z',0)|+|u(z',0)-u(z',r)|+|u(z',r)-u(0,r)|\\
&\leq 2n|z'|^\beta+\int_{0}^r |\de_nu(z',s)|\, ds\\
&\lesssim_n r^{\theta\beta} + r^{\theta\frac{(1-n)}{p}}r^{\frac{1}{p'}}.
\end{align*}
Since $|h|=r$, $u$ is $\sigma$-H\"older continuous for every $\sigma\leq \min\{\theta\beta,\theta(1-n)/p+1/p'\}$, then maximizing with respect to the interpolation parameter $\theta>0$ we get the optimal value for $\sigma$.
\end{proof}
Let us finally give, for completeness, the proof of our statement of Whitney's theorem for $C^\infty$ functions.
\begin{proof}[Proof of Theorem \ref{thm:whitney}.]   
   Let us define for each multi-index $\alpha\in\N^n$ the functions $f_\alpha\colon E\to \R$ by 
   $$
   f_\alpha(x):=\de^\alpha P_{|\alpha|,x}(\cdot)\  \Big(=\de^\alpha P_{|\alpha|+\ell,x}(0)\quad\text{ for all }\ell\in\N\Big).
   $$
Assumption (ii) with $k=|\alpha|$ immediately gives
   $$
   |f_\alpha(x)-f_\alpha(y)|=|\de^\alpha P_{|\alpha|,x}(0)-\de^\alpha P_{|\alpha|,y}(x-y)|\leq C(|\alpha|)|x-y|.
   $$
   Thus each $f_\alpha$ admits a canonical Lipschitz extension to $\overline E$, which we don't rename. 
   
   For each $x,y,\in \overline E, m\in\N$ and $|\alpha|\leq m$ define the remainder
   \begin{equation}\label{eq:remainderwhitney}
   R_{m,\alpha}(x,y):=f_\alpha(x)-\sum_{|\beta|\leq m-|\alpha|}\frac{f_{\alpha+\beta}(y)}{\beta!}(x-y)^\beta.
   \end{equation}
   If we show that each remainder satisfies $|R_{\alpha,m}(x,y)|\leq C(m)|x-y|^{|\alpha|-m+1}$ for all $x,y\in E$, then we can extend it by continuity, so that it holds in the full $\overline {E\times E}$, and conclude, applying \cite[Theorem I]{W34} verbatim. To check this notice that the left hand side of assumption (ii) is just \eqref{eq:remainderwhitney} in disguise:
   \begin{align*}
     R_{m,\alpha}(x,y)&=\de^\alpha P_{m,x}(0)-\sum_{|\beta|\leq m-|\alpha|}\frac{\de^\beta(\de^\alpha P_{m,y})(0)}{\beta!}(x-y)^\beta\\
     &=\de^\alpha P_{m,x}(0)-\de^\alpha P_{m,y}(x-y)=O(|x-y|^{m-|\alpha|+1}),
   \end{align*}
   where we used that polynomials equals their Taylor expansion of sufficiently high degree (and here $\deg \de^\alpha P_{m,y}\leq m-|\alpha|$). This concludes the proof.
\end{proof}
\section*{Acknowledgements}

F.F. was supported by the Swiss NSF Ambizione Grant PZ00P2 180042 and by the  European Research Council (ERC) under the Grant Agreement No 721675.

\noindent W.Z. was supported by the European Research Council (ERC) under the Grant Agreement No 801867. 

The authors would like to thank their supervisors Joaquim Serra and Xavier Ros-Oton for introducing them to the problem and for the continuous guidance and support during the preparation of this work.

\Addresses
\end{document}